\documentclass[11pt]{amsart}
\usepackage[color=blue!20!white,textsize=tiny]{todonotes}
\usepackage{mathabx,epsfig}

\usepackage{amsmath} 
\usepackage{amsthm}
\usepackage{amssymb} 
\usepackage{amsfonts}
\usepackage{tikz-cd}
\usepackage{graphicx}
\usepackage{xcolor}
\definecolor{darkgreen}{rgb}{0,0.5,0}
\usepackage[normalem]{ulem}
\usepackage{comment}

\usepackage{caption,subcaption,pinlabel}
\usepackage{graphics}
\usepackage[pagebackref=true]{hyperref}
\usepackage{scalerel}
\usepackage{comment}
\usepackage{enumitem}
\usepackage{mathtools}
\usepackage{tikz}
\usetikzlibrary{matrix,arrows,decorations.pathmorphing}
\usepackage{mathabx}
\usetikzlibrary{matrix}
\usepackage{scrextend}

\usetikzlibrary{calc}
\usetikzlibrary{decorations.pathmorphing}

\tikzset{curve/.style={settings={#1},to path={(\tikztostart)
    .. controls ($(\tikztostart)!\pv{pos}!(\tikztotarget)!\pv{height}!270:(\tikztotarget)$)
    and ($(\tikztostart)!1-\pv{pos}!(\tikztotarget)!\pv{height}!270:(\tikztotarget)$)
    .. (\tikztotarget)\tikztonodes}},
    settings/.code={\tikzset{quiver/.cd,#1}
        \def\pv##1{\pgfkeysvalueof{/tikz/quiver/##1}}},
    quiver/.cd,pos/.initial=0.35,height/.initial=0}

\tikzset{tail reversed/.code={\pgfsetarrowsstart{tikzcd to}}}
\tikzset{2tail/.code={\pgfsetarrowsstart{Implies[reversed]}}}
\tikzset{2tail reversed/.code={\pgfsetarrowsstart{Implies}}}
\tikzset{no body/.style={/tikz/dash pattern=on 0 off 1mm}}

\newtheorem{theorem}{Theorem}[section]
\newtheorem{lemma}[theorem]{Lemma}
\newtheorem{proposition}[theorem]{Proposition}
\newtheorem{question}[theorem]{Question}

\newtheorem{corollary}[theorem]{Corollary}

\theoremstyle{definition}
\newtheorem{definition}[theorem]{Definition}
\newtheorem{example}[theorem]{Example}

\newtheorem{remark}[theorem]{Remark}

\def\H{\mathcal{H}}

\def\W{\scaleto{W}{3pt}}
\def\spinc {{\operatorname{spin^c}}}
\def\s{\mathfrak s}

\def\x{\mathbf{x}}

\newcommand\alphas{\boldsymbol\alpha}
\newcommand\betas{\boldsymbol\beta}
\newcommand\gammas{\boldsymbol\gamma}

\newcommand{\G}{{\Theta^{spin,\tau}_\mathbb{Q}}}

\newcommand{\Z}{\mathbb{Z}}

\let\int\relax
\newcommand{\int}{\mathring}

\usepackage{accents}

\DeclareMathSymbol{\wtilde}{\mathord}{largesymbols}{"65}

\title{Knot Floer homology and surgery on equivariant knots}
\author{Abhishek Mallick}
\address{Max-Planck-Institut f\"ur Mathematik, Vivatsgasse 7, 53111 Bonn, Germany}
\email{mallick@mpim-bonn.mpg.de}
\begin{document}

\maketitle

\newcommand{\scal}{{\rm scal}}
\newcommand{\dt}{\cdot}
\newcommand{\ep}{\epsilon}
\newcommand{\sg}{\sigma}
\newcommand{\Om}{\Omega}
\newcommand{\Mdis}[1]{M^\amalg_{#1}}
\newcommand{\Mcon}[1]{M^\#_{#1}}
\newcommand{\Ndis}[1]{N^\amalg_{#1}}
\newcommand{\Ncon}[1]{N^\#_{#1}}
\newcommand{\FreeHS}{\Z\{RHS^3\}}
\newcommand{\tn}{\otimes}
\newcommand{\Ha}{\mathbb{H}}
\newcommand{\so}{\mathfrak so}
\newcommand{\tens}[1]{%
  \mathbin{\mathop{\otimes}\limits_{#1}}%
}

\mathchardef\mhyphen="2D

\begin{abstract}
Given an equivariant knot $K$ of order $2$, we study the induced action of the symmetry on the knot Floer homology. We relate this action with the induced action of the symmetry on the Heegaard Floer homology of large surgeries on $K$. This surgery formula can be thought of as an equivariant analog of the involutive large surgery formula proved by Hendricks and Manolescu. As a consequence, we obtain that for certain double branched covers of $S^{3}$ and corks, the induced action of the involution on Heegaard Floer homology can be identified with an action on the knot Floer homology. As an application, we calculate equivariant correction terms which are invariants of a generalized version of the spin rational homology cobordism group, and define two knot concordance invariants. We also compute the action of the symmetry on the knot Floer complex of $K$ for several equivariant knots. 
\end{abstract}



\section{Introduction}\label{intro}

Let $K$ be a knot in $S^{3}$ and $\tau$ be an orientation preserving diffeomorphism of order $2$ of $S^{3}$ which fixes $K$ setwise. We refer to such pairs $(K,\tau)$ as an \textit{equivariant knot} (of order $2$), where the restriction of $\tau$ to $K$ acts as a symmetry of $K$.  
\noindent When the fixed set of $\tau$ is $S^{1}$, it can either intersect $K$ in two points or be disjoint from $K$. If there are two fixed points on $K$, we refer to $(K, \tau)$ as \textit{strongly invertible} and if the fixed point set is disjoint from $K$, we call $(K,\tau)$ \textit{periodic}. It is well-known that Dehn surgery on such equivariant knots $(K,\tau)$ induce an involution $\tau$ on the surgered $3$-manifold \cite{Montesinos1975}. In particular, Montesinos \cite{Montesinos1975} showed that a $3$-manifold is a double branched covering of $S^{3}$ if and only if it can be obtained as surgery on a strongly invertible link (defined similarly as knots). In fact, it follows from \cite{Montesinos1975} that in such cases one can identify the covering involution with the induced involution from the symmetry of the link on the surgered manifold. More generally, \cite[Theorem 1.1]{sakuma2001surgery}  proved an equivariant version of the Lickorish–Wallace theorem \cite{lickorish1962representation, wallace1960modifications}, namely he showed any finite order orientation preserving diffeomorphism on a compact $3$-manfifold can always be interpreted as being induced from surgery on an equivariant link with integral framing. Hence $3$-manifolds with involutions are in one-one correspondence with surgeries on equivariant links.

Involutions on $3$-manifolds can be quite useful in studying various objects in low-dimensional topology. For example, recently Alfieri-Kang-Stipsicz \cite{AKS} used the double branched covering action to define a knot concordance invariant, which takes the covering involution into account in an essential way. Another instance of $3$-manifolds being naturally equipped with involution are corks, which play an important part in studying exotic smooth structures on smooth compact $4$-manifolds. In \cite{dai2020corks} Dai-Hedden and the author studied several corks that can be obtained as surgery on a strongly invertible knot where the cork-twist involution corresponds to the induced action of the symmetry on the surgered manifold. In both the studies \cite{AKS} and \cite{dai2020corks}, the key tool was to understand the induced action of the involution on the Heegaard Floer chain complex of the underlying $3$-manifolds. On the other side of the coin, in \cite{abhishek2022equivariant} Dai-Stoffrengen and the author showed the action of an involution on the knot Floer complex of an equivariant knot can be used to produce equivariant concordance invariants which bound equivariant $4$-genus of equivariant knots, and it can be used to detect exotic slice disks. 

\subsection{Equivariant surgery formula}
In light of the usefulness (in both equivariant and non-equivariant settings) of studying the action of an involution on $3$-manfiolds and knots through the lens of Heegaard Floer homology, it is a natural question to ask whether we can connect a bridge between the $3$-manifolds and knots perspective with the overarching goal of understanding one from the other. The present article aims to establish such correspondence in an appropriate sense, which we describe now.

\begin{theorem}\label{t1}
Let $(K,\tau, w,z)$ be a doubly-based equivariant knot of order $2$ with the symmetry $\tau$ and let $p \in \mathbb{Z}$. Then for all $p\geq g(K)$, there exists a chain isomorphism $\Gamma^{+}_{p,0}$ so that the following diagram commutes up to chain homotopy, for $\circ \in \{-,+\}$.
\[
 \begin{tikzcd}[row sep =large]
{CF}^{\circ}(S^{3}_{p}(K),[0]) \arrow{r}{\tau} \arrow{d}{\Gamma^{\circ}_{(w,z)}} & {CF}^{\circ}(S^{3}_{p}(K),[0]) \arrow{d}{\Gamma^{\circ}_{(w,z)}} \\
{A^{\circ}_{0}(K)} \arrow{r}{\tau_K}  & {A^{\circ}_{0}(K)} 
\end{tikzcd}
\]

\end{theorem} 
We can interpret Theorem~\ref{t1} as the \textit{(large) equivariant surgery formula} in Heegaard Floer homology. The equivariant concordance invariants $\underline{V}^{\tau}_{0}$ and $\overline{V}^{\tau}_{0}$ that were defined in \cite{abhishek2022equivariant} can be interpreted as the \textit{equivariant correction terms} stemming from the equivariant surgery formula. This surgery formula was also used in  \cite[Theorem 1.6]{abhishek2022equivariant} to show that knot Floer homology detects exotic pairs of slice disks. 

As discussed in the introduction, the identification in Theorem~\ref{t1} includes the following two classes of examples.
\begin{itemize}

\item By Montesinos trick, surgery on a strongly invertible knot is a double cover $\Sigma_2$ of $S^{3}$. When the surgery coefficient is large we can use Theorem~\ref{t1} to identify the covering action on $HF^{\circ}(\Sigma_2)$ with an action of an involution on the knot Floer homology.

\item Let $(K,\tau)$ be an equivariant knot with $g(K)=1$. If $(S^{3}_{+1}(K), \tau)$ is a cork, then the cork-twist action on $HF^{\circ}$ is identified with the action of $\tau_K$ on a sub(quotient)-complex of $CFK^{\infty}(K)$ via the Theorem~\ref{t1}. We can apply this identification to many well-known corks in the literature, they include $(+1)$-surgery on the Stevedore knot, $(+1)$-surgery on the $P(-3,3,-3)$ pretzel knot (also known as the Akbulut cork), and the positron cork \cite{AMexotic, hayden2021corks}. In fact, the identification of the two involution for the positron cork was useful in \cite{abhishek2022equivariant} to re-prove a result due to \cite{hayden2021corks}.

\end{itemize}

We now explain various terms appearing in Theorem~\ref{t1}. 
Let $(Y,z,\s, \tau)$ represent a based $3$-manifold $(Y,z)$ decorated with a $\spinc$ structure, and an orientation preserving involution $\tau$ on $Y$ which fixes $\s$. In \cite{dai2020corks}, it was shown that when $Y$ is a $\mathbb{Q}HS^{3}$, $\tau$ induce an action on the Heegaard Floer chain complex $CF^{\circ}(Y,z,\s)$
\[
\tau:CF^{\circ}(Y,z,\s) \rightarrow CF^{\circ}(Y,z,\s).
\]
In a similar vein, given a doubly-based equivariant knot $(K,\tau, w , z)$, there is an induced action of the symmetry on the Knot Floer chain complex of the knot, 
\[
\tau_{K}: CFK^{\infty}(K,w,z) \rightarrow CFK^{\infty}(K,w,z),
\]
defined using the naturality results by Juh{\'a}sz, Thurston, and Zemke \cite{JTZ} (see Section~\ref{action_knot}).

Let us now cast the actions above in the context of symmetric knots. If we start with an equivariant knot $(K,\tau)$, then as discussed before, the surgered manifold $S^{3}_{p}(K)$ inherits an involution which we also refer to as $\tau$. Theorem~\ref{t1} then identifies the action $\tau_K$ with the action $\tau$ on the level of Heegaard Floer chain complexes \footnote{Here and throughout the paper we will use the $\tau$ and $\tau_K$ denote the action on the $CF^{\circ}$ and $CFK^{\circ}$ respectively.}. In fact, the identification is mediated by the \textit{large surgery isomorohism} map $\Gamma^{+}_{p,0}$, defined by Ozsv{\'a}th-Szab{\'o} and Rasmussen \cite{OSknots, rasmussenthesis}. Specifically, they defined a map
\[
 \Gamma^{+}_{p,s}: CF^{+}(S^{3}_{p}(K),[s]) \rightarrow A^{+}_{s}(K),
 \]
 by counting certain pseudo-holomorphic triangles and showed that it is a chain isomorohism between the Heegaard Floer chain complex $CF^{+}(S^{3}_{p}(K),[s])$ with a certain quotient complex $A^{+}_{s}$ of the knot Floer chain complex $CFK^{\infty}(K)$ (for $|s| \leq p/2$). In this context, \textit{large} means the surgery coefficient sufficiently large compared to the $3$-genus of the knot $K$. $[s]$ represents a $\spinc$-structure on $S^{3}_{p}(K)$ under the standard identification of $\spinc$-structures of $S^{3}_{p}(K)$ with $\mathbb{Z}/p\mathbb{Z}$. 

Theorem~\ref{t1} can also be thought of as an equivariant analog of the large surgery formula in involutive Heegaard Floer homology, proved by Hendricks and Manolescu \cite{HM}, where the authors showed a similar identification holds when we replace the action of involution on the Heegaard Floer chain complexes with the action of the $\spinc$-conjugation.

\begin{figure}[h!]
\centering
\includegraphics[scale=.83]{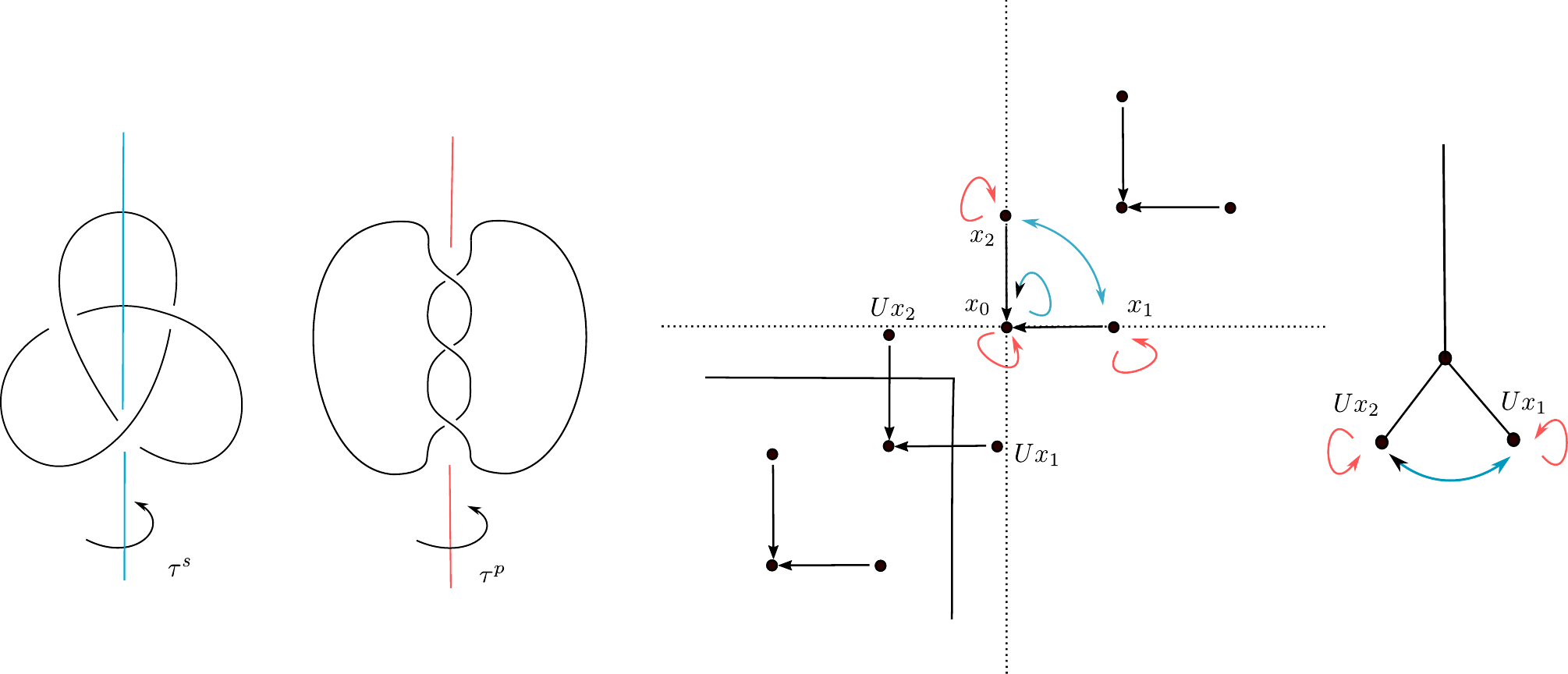} 
\caption
{{Left: $(+1)$-surgery on the left-handed trefoil $\overline{T}_{2,3}$ with two different symmetries, one equipped with the strong involution $\tau^{s}$ (in blue) and the other with a periodic involution $\tau^{p}$ (in red). Middle: $CFK^{\infty}(\overline{T}_{2,3})$ with $\tau^{s}_K$ and $\tau^{p}_K$ action. Right: $HF^{+}(S^{3}_{+1}(\overline{T}_{2,3}))$ with the $\tau^{s}$ and $\tau^{p}$ action.}}\label{trefoilintro}
\end{figure}
\begin{example} 
Let us now consider a rather simple example from Figure~\ref{trefoilintro}, where the identification from Theorem~\ref{t1} is demonstrated. We start with the strong involution $\tau^{s}$ on the left-handed trefoil first. The computation of the action $\tau^{s}_K$ shown in the middle, is rather straightforward and follows from the grading and filtration restrictions (see Proposition~\ref{lspacestrong}). Recall that the $A^{+}_{0}$-complex is defined by looking at the generators $[\x,i,j]$ with $i \geq 0$ and $j \geq 0$, i.e. the generators lying outside of the hook shape in the middle diagram. So by applying Theorem~\ref{t1}, we get the $\tau^{s}$ action on $HF^{+}(S^{3}_{+1}(\overline{T}_{2,3}))$. Note that the involution $\tau^{s}$ is a double-covering involution on $S^{3}_{+1}(\overline{T}_{2,3})= -\Sigma(2,3,7)$. Hence, we have identified a covering action on $HF^{+}(-\Sigma(2,3,7))$ with the action of $\tau^{s}_K$ on $A^{+}_{0}$. In fact, readers may explicitly compute the branching knot, identify it as a Montesinos knot and refer to \cite{AKS} to obtain that the action on $HF^{+}(-\Sigma(2,3,7))$ is exactly as in the right-hand side of Figure~\ref{trefoilintro}. Note however that, identifying the branching knot even in this simple case is not immediate, as the knot in question has $12$ crossings. Readers may also compare the computation above with the work of Hendricks-Lipshitz-Sarkar \cite[Proposition 6.27]{flexible_equivariant}, where the authors computed the $\tau^{s}$ action on the hat-version $\widehat{HF}(S_{+1}(\overline{T}_{2,3}))$ by explicitly writing down the action on the generators using an equivariant Heegaard diagram. In contrast, our approach of computing the action on the $3$-manifold will always be via identifying the action on knot Floer complex of the underlying equivariant knot.

We now look at the periodic involution $\tau^{p}_K$ acting on the left-handed trefoil as in Figure~\ref{trefoilintro}. It is easily seen that the action of $\tau_K$ on $CFK^{\infty}(\overline{T}_{2,3})$ is trivial (Proposition~\ref{lspaceperiodic}), hence by Theorem~\ref{t1}, we get the action of $\tau^{p}$ on $HF^{+}(-\Sigma(2,3,7))$ is trivial. Readers may again check that, for the case in hand the branching set is a $(3,7)$-torus knot, which necessarily imply that the covering action is isotopic to identity justifying the trivial action on $HF^{+}(S^{3}_{+1}(\overline{T}_{2,3})$.

\end{example}

\subsection{Equivariant $\iota$-surgery formula}

In \cite{HM}, Hendricks and Manolescu studied the $\spinc$-conjugation action $\iota$ on a based $3$-manifold $(Y,z)$ equipped with a $\spinc$ structure $\s$.
\[
\iota: CF^{\infty}(Y,z,\s) \rightarrow CF^{\infty}(Y,z,\bar{\s}).
\]
Moreover, given a doubly-based knot $(K,w,z)$ in $S^{3}$, they defined the $\spinc$-conjugation action $\iota_K$ on the knot Floer complex,
\[
\iota_K:CFK^{\infty}(K,w,z) \rightarrow CFK^{\infty}(K,w,z).
\]  
Hendricks and Manolescu, also proved a large surgery formula relating the action $\iota$ for $(S^{3}_{p}(K),[0])$ with the action of $\iota_K$, \cite[Theorem 1.5.]{HM}. 

In the presence of an involution $\tau$ acting on a $\mathbb{Z}HS^{3}$, $Y$, the action 
\[
\iota \circ \tau: CF^{\infty}(Y,z) \rightarrow CF^{\infty}(Y,z)
\]
was studied in \cite{dai2020corks}. In fact, this action turned out to be quite useful in studying such pairs $(Y, \tau)$ and it was shown to contain information that is essentially different from information contained in the $\iota$ and $\tau$ action, see for instance \cite[Lemma 7.3]{dai2020corks}. On the other hand, in \cite{abhishek2022equivariant}, it was shown that the action $\iota_K \circ \tau_K$ is useful to study equivariant knots.   This motivates the question, whether these two aforementioned actions can be identified. Indeed, as a corollary of Theorem~\ref{t1}, we show that we can also identify the action of $\iota \circ \tau$ with that of $\iota_K \circ \tau_K$ on $A^{+}_{0}(K)$ for large surgeries on equivariant knots. More precisely, let $HFI^{+}_{\iota \tau}(Y,\s)$ represent the homology of the mapping cone complex of the map
\[
CF^{+}(Y,\s) \xrightarrow{Q(\mathrm{id} + \iota \tau)}  Q.CF^{+}(Y,\s)[-1].
\]
Where $Q$ is a formal variable such that $Q^{2}=0$ and $[-1]$ denotes a shift in grading.  Similarly, given a symmetric knot $(K,\tau)$, let $AI^{+, \iota \tau}_{0}(K)$ represent the mapping cone chain complex of the map 
 \[
 A^{+}_{0}(K) \xrightarrow{Q(\mathrm{id} + \iota_K \tau_{K})} Q.A^{+}_{0}(K)[-1].
\] 
Here, we use $\tau_{K}$ to represent the map induced from $\tau_{K}$ on the quotient complex $A^{+}_{0}$. As a Corollary of Theorem~\ref{t1}, we have:
\begin{corollary}\label{corollary}
Let $(K,\tau, w,z)$ be a doubly-based equivariant knot. Then for all $p\geq g(K)$, as a relatively graded $\mathbb{Z}_{2}[U,Q]/(Q^{2})$-module we have the identification,
\begin{center}
$HFI^{+}_{\iota \tau}(S^{3}_{p}(K),[0]) \cong H_{*}(AI^{+,\iota \tau}_{0}).$
\end{center}
\end{corollary}
\noindent
By casting the example from Figure~\ref{t1} in this context, we see the colors of the actions are switched. Namely, $\iota_K \circ \tau^{s}_{K}$ acts as the red action, while $\iota_K \circ \tau^{p}_{K}$ acts as the blue action, and likewise for the $3$-manifold action.
\subsection{Rational homology bordism group of involutions.}
Surgeries on equivariant knots $S^{3}_{p}(K)$ are examples of $\mathbb{Q}HS^{3}$, equipped with an involution $\tau$. One may wish to study such pairs in general subject to some equivalence. Namely, one can study pairs $(Y,\tau)$ (where $Y$ is a $\mathbb{Q}HS^{3}$ and $\tau$ is an involution acting on it), modulo rational homology bordism. Naively, rational homology bordisms are rational homology cobordisms equipped with a diffeomorphism restricting to the boundary diffeomorphisms.  

The general bordism group $\Delta_{3}$ (without any homological restrictions on the cobordism) has been well-studied in the past by Kreck \cite{Kreck}, Melvin \cite{Melvin}, and Bonahon \cite{Bonahon}. The object of this group are pairs $(Y,\phi)$, where $Y$ is a $3$-manifold and $\phi$ is a diffeomorphism on it. Two such pairs $(Y_1,\phi_1)$ and $(Y_2,\phi_2)$ are equivalent if there exist a pair $(W,f)$ where $W$ is a cobordism between $Y_1$ and $Y_2$ and $f$ is a diffeomorphism that restrict to $\phi_1$ and $\phi_2$ in the boundary. The equivalence classes form a group under disjoint union. Interestingly, Melvin \cite{Melvin} showed that $\Delta_{3}$ is trivial. However, it is natural to expect a more complicated structure when we put homological restrictions on the cobordism $W$. This parallels the situation for the usual cobordism group in $3$-dimension, where by putting the extra homological restriction to the cobordism one obtains the group $\Theta^{3}_{\mathbb{Z}}$. Such a group,  $\Theta^{\tau}_{\mathbb{Z}}$ was defined in \cite{dai2020corks}, referred as the \textit{homology bordism group of involutions}. $\Theta^{\tau}_{\mathbb{Z}}$ can be thought of as a generalized version of integer homology cobordism group $\Theta^{3}_{\mathbb{Z}}$ by incorporating the involutions on the boundary. In this paper, we define $\G$, which is a generalized version of $\Theta^{spin}_{\mathbb{Q}}$, again by taking into account an involution on the boundary. We refer to $\G$ as the \textit{rational homology bordism group of involutions}. The object of this group are elements are tuples $(Y,\tau,\s)$ where $Y$ is a rational homology sphere, $\s$ is a spin-structure on $Y$ and $\tau$ is an involution on $Y$ which fixes $\s$. We identify two such tuples $(Y_1,\tau_1,\s_1)$ and $(Y_2, \tau_2, \s_2)$ if there exist a tuple  $(W, f, \s_{\W})$, where $W$ is a rational homology cobordism between $Y_1$ and $Y_2$, $f$ is a diffeomorphism of $W$ which restricts to the diffeomorphisms $\tau_1$ and $\tau_2$ on the boundary, and $\s_{\W}$ is a spin-structure on $W$ which restricts to $\s_1$ and $\s_2$ on the boundary. We also require that $f(\s_{\W})=\s_{\W}$. We refer to the equivalence classes of such pairs as the \textit{pseudo $\mathbb{Q}$-homology bordism class}. These classes form a group under disjoint union and we denote this group by $\G$. 

In \cite{dai2020corks} two correction terms were defined $\underline{d}_{\circ}$ and $\bar{d}_{\circ}$ (for $\circ \in \{ \tau, \iota \tau \}$), which may be interpreted as the \textit{equivariant involutive correction terms}. These correction terms are analogous to those defined by Hendricks and Manolescu \cite{HM}. The following Theorem is implicit in  \cite{dai2020corks}:
\begin{theorem}\label{dinvariant}
$\underline{d}_{\tau}$, $\bar{d}_{\tau}$, $\underline{d}_{\iota \tau}$, $\bar{d}_{\iota \tau}$ are invariants of pseudo $\mathbb{Q}$-homology bordism class. Hence, they induce maps
\[
\underline{d}_{\circ},\bar{d}_{\circ}: \G \rightarrow \mathbb{Q}.
\]
Where $\circ \in \{ \tau ,\iota \tau \}$.
\end{theorem}
\noindent
It follows from Theorem~\ref{t1} and Corollary~\ref{corollary}, that we can reduce the computation of the invariants $\underline{d}_{\circ}$ and $\bar{d}_{\circ}$ for rational homology spheres that are obtained via the surgery on symmetric knots $(K,\tau)$ to a computation of the action of $\tau_K$ on $CFK^{\infty}(K)$. Indeed, we show that for a vast collection of symmetric knots, we can compute the equivariant involutive correction terms, see Section~\ref{computation}. 

\subsection{Computations} As observed above, much of the usefulness of Theorem~\ref{t1} relies upon being able to compute the action of $\tau_K$ on $CFK^{\infty}(K)$. In general, it is quite hard to compute the action of the symmetry. Indeed to the best of the author's knowledge, prior to the present article, the basepoint moving automorphism (computed by Sarkar \cite{sarkar2015moving} and Zemke \cite{zemke2017quasistabilization}) and the periodic involution on the knot $K \# K$ (computed by Juh{\'a}sz-Zemke \cite[Theorem 8.1]{juhasz2018stabilization}, see Figure~\ref{KconnectKperiod}) were the only two known examples of computation of a mapping class group action for a class of knots. 

In general computations of this sort are involved because, in order to compute the actions from scratch, one usually needs an equivariant Heegaard diagram in order to identify the action of $\tau_K$ on the intersection points of $\alpha$ and $\beta$-curves. Typically, this Heegaard diagram has genus bigger than 1, which makes it harder to calculate the differential by identifying the holomorphic disks, see \cite[Proposition 6.27]{flexible_equivariant}. Another approach is given by starting with a grid diagram \cite{manolescucombinatorial} of the knot, so that identifying the holomorphic disks is easier but in that case number of intersection points increases drastically making it harder to keep track of the action of $\tau_K$ on the intersection points. 

We take a different route for computing the action of $\tau_K$. We show that $\tau_K$ comes with certain filtration and grading restrictions along with identities that it satisfies due to its order, (see Subsection~\ref{action_knot}) which enables us to compute the action for a vast class of equivariant knots.  This is reminiscent of the computation of $\iota_K$ action on the knot Floer complex, \cite{HM}. Although $\tau_K$ actions are more complicated in nature compared to $\iota_K$, as we will see later.
 
In \cite{Lens} Ozsv{\'a}th and Szab{\'o} introduced the L-space knots, i.e. knots which admit an L-space surgery. For example, these knots include the torus knots and the pretzel knots $P(-2,3,2k+1)$. Many of the L-space knots also admit symmetry. Such examples include the torus knots, with their unique strong involution \cite{Schreier_torus} and the $P(-2,3,2k+1)$ pretzel knots. In fact, it was conjectured by Watson by all L-space knots are strongly invertible \cite[Discussion surrounding Conjecture $30$]{watson2017khovanov}, which was later disproved by Baker and Luecke \cite{baker2020asymmetric} by providing a counterexample. Notably the simplest counterexample in \cite{baker2020asymmetric} is a knot of genus $119$. In Section~\ref{computation}, we compute the action of $\tau_K$ on $CFK^{\infty}(K)$ for any strongly invertible L-space knot $(K,\tau_K)$. In fact, it turns out, that the action is exactly the same as the action of $\iota_K$ on such knots. On the other hand, we show that if $(K, \tau_K)$ is a periodic L-space knot then the action of $\tau_K$ is chain homotopic to identity (see Proposition~\ref{lspacestrong} and Proposition~\ref{lspaceperiodic}). In light of Theorem~\ref{t1}, let us now observe the following. If $\tau_K \simeq \iota_K$ for a strongly invertible knot $(K,\tau)$, then the $\iota$-action on $(S^{3}_{p}(K),[0])$ can be identified with the action of the involution $\tau$ on $(S^{3}_{p}(K),[0])$, here $p \geq g(K)$. As noticed before, this involution $\tau$ is in fact a covering involution. First examples of the identification of a covering involution of a double-cover of $S^{3}$ with the $\iota$-action were shown by Alfieri-Kang-Stipsicz \cite[Theorem 5.3]{AKS}. Leveraging the computation  of $\tau_K$ on torus knots we at once obtain the following: 
 \begin{proposition}\label{iotaidentify}
Let $n \geq (p-1)(q-1)/2$ then
\[
CF^{-}(S^{3}_{n}(\overline{T}_{p,q}),\tau,[0]) \simeq CF^{-}(\Sigma_{2}(L_{n,p,q}), \tau_{n,p,q},[0])  \simeq CF^{-}(S^{3}_{n}(\overline{T}_{p,q}),\iota,[0]).
 \] 
\end{proposition}
\noindent
Where $L_{n,p,q}$ represents the branching locus of $\tau$-action, $\tau_{n,p,q}$ represents the covering involution  and `$\simeq$' denotes local equivalence (see Definition~\ref{localequi}).

\begin{figure}[h!]
\centering
\includegraphics[scale=.50]{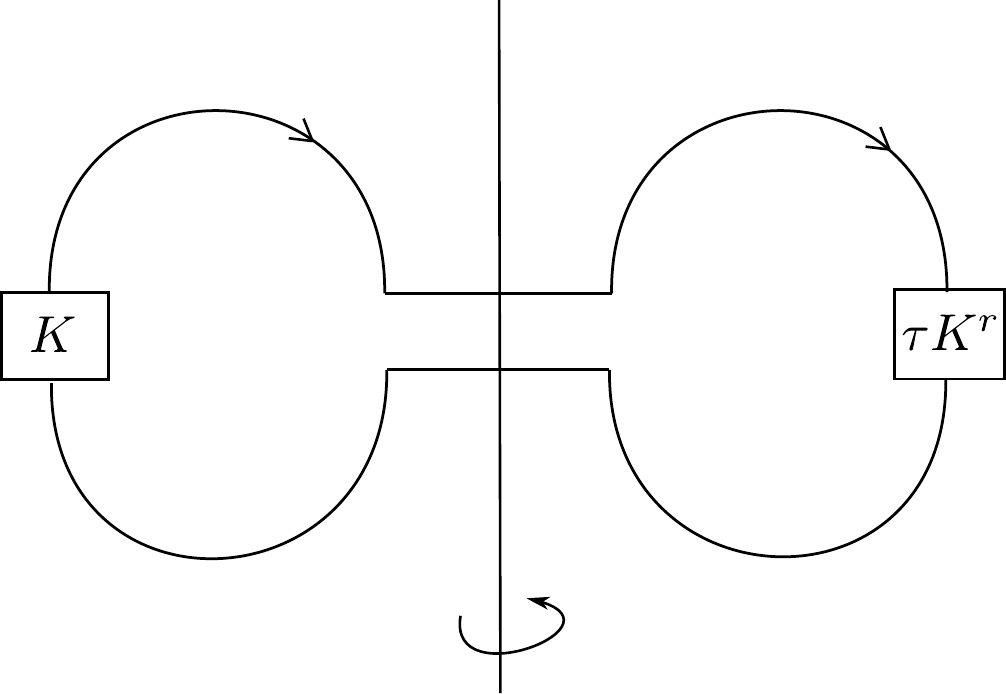} 
\caption
{The symmetry of $K \# \tau K^{r}$ of considered in Theorem~\ref{concordance}.}\label{KconnectK}
\end{figure} 

We also consider the strongly invertible symmetry $\tilde{\tau}_{sw}$ on the knot $K \# \tau K^{r}$. Here $K$ is any knot in $S^{3}$ and $K^{r}$ is the orientation reversed of $K$. $\tau$ is defined to be the symmetry depicted in Figure~\ref{KconnectK} induced by \textit{switching} the factors $K$ and $\tau K^{r}$. We denote the action of $\tau$ on the knot Floer complex of $K \# \tau K^{r}$ as $\tilde{\tau}_{sw}$. We show:

\begin{theorem}\label{KconnectKstrong}
Given a doubly-based knot $(K,w,z)$, and the knot $(\tau{K}^{r},\tau w, \tau z)$ as above. There is a filtered chain homotopy equivalence
\[
F:\mathcal{CFL}^{\infty}(K,w,z) \otimes \mathcal{CFL}^{\infty}(\tau {K}^{r}, \tau w, \tau z) \rightarrow \mathcal{CFL}^{\infty}(K \# \tau {K}^{r}, w^{\prime}, z^{\prime})
\]
which intertwines with the strong involution action $\tilde{\tau}_{sw}$ with the map $(\mathrm{id} \otimes \mathrm{id} + \Psi \otimes \Phi) \circ \tilde{\tau}_{exch}$. 
\end{theorem}

\noindent
Here $\mathcal{CFL}$ is the full link Floer complex considered in  \cite{Zemkelinkcobord} and $\tilde{\tau}_{exch}$ is an automorphism of $\mathcal{CFL}^{\infty}(K,w,z) \otimes \mathcal{CFL}^{\infty}(\tau {K}^{r}, \tau w, \tau z)$ induced from exchanging the two factors of the connect sum. We refer readers to Subsection~\ref{KK} for precise definitions.
A slightly different version of $\tilde{\tau}_{sw}$, $\tau_{sw}$ (owing to the relative placement of basepoints on the knots $K$ and $K^{r}$) was computed in \cite{abhishek2022equivariant} using a different approach. See Remark~\ref{compare}, for a comparison between $\tau_{sw}$ and $\tilde{\tau}_{sw}$.

In another direction, we consider the \textit{Floer homologically thin} knots. These knots are well-studied in the literature, see for example \cite{rasmussenthesis}, \cite{manolescuquasi}. By definition, thin knots are specifically those knots for which the knot Floer homology is supported in a single diagonal. These knots include the alternating knots and more the generally quasi-alternating knots \cite{manolescuquasi}. Typically for these knots, the Heegaard Floer theoretic invariants tend to be uniform in nature, for example, both knot concordance invariant $\tau$  and $\delta$ are determined by the knot signature \cite{FourBall, manolescuquasi}. Further evidence of this uniformity is demonstrated in \cite{HM}, where it was shown that that for a thin knot $K$, there is a unique action (up to change of basis) on $CFK^{\infty}(K)$ which is grading preserving and skew-filtered, which lead to the computation of $\iota_K$ action on those class of knots. 
\begin{figure}[h!]
\centering
\includegraphics[scale=.8]{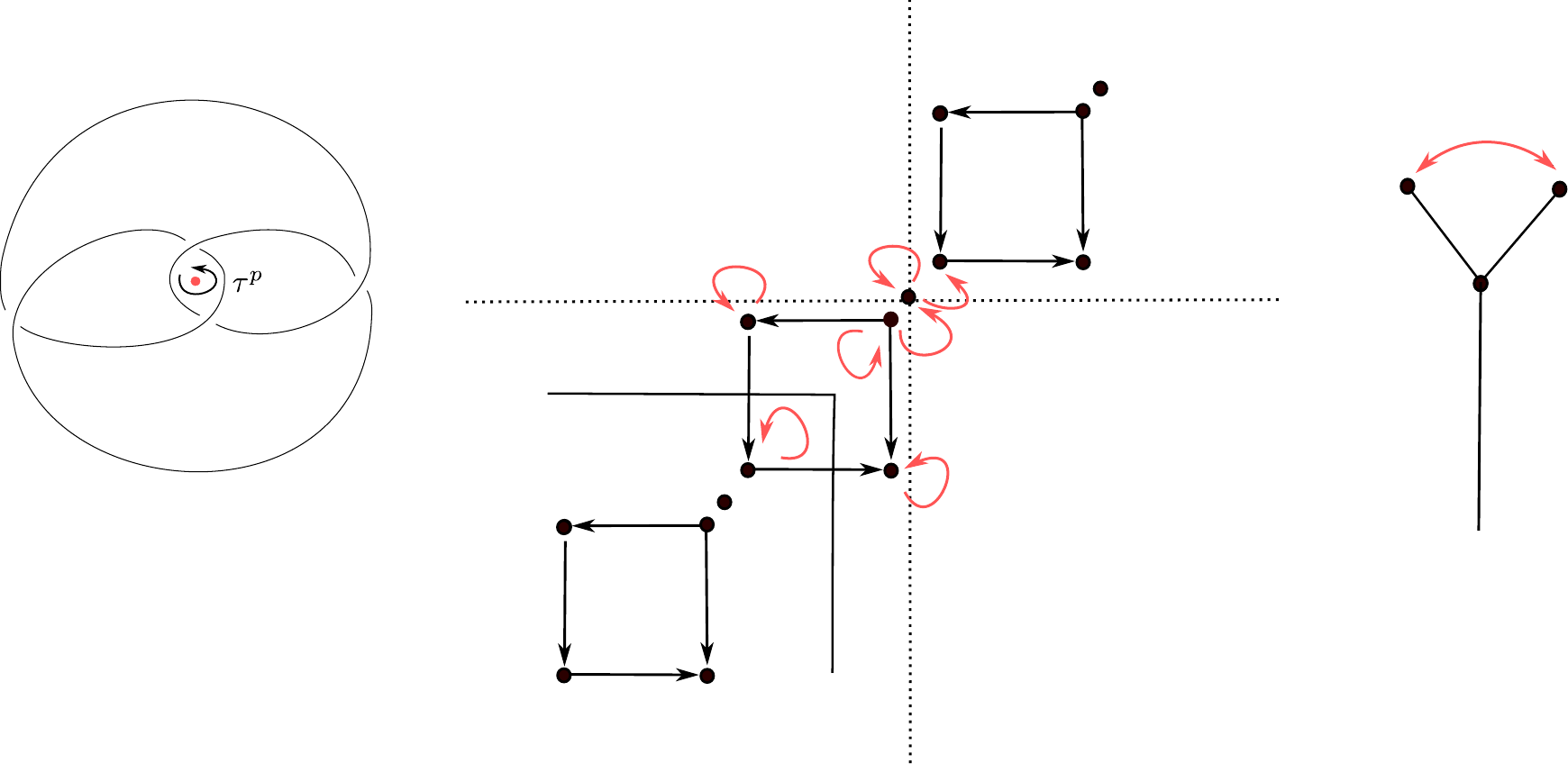} 
\caption
{{Left: Figure-eight knot $4_1$ with a periodic involution $\tau^{p}$. Middle: $CFK^{\infty}(4_1)$ with $\tau^{p}_K$ action, note that there are no boxes outside the main diagonal. Right: $HF^{-}(S_{+1}(4_1))$ with the $\tau^{p}$-action.}}\label{figure_eight_periodic}
\end{figure}
We show that the analogous situation for the action of $\tau_K$ for equivariant thin knots is quite different. For example for the figure-eight knot, one can have two different actions of strongly invertible symmetry, see discussion in Example~\ref{strong_figure_eight}. The parallel situation in the periodic case is slightly different from both the strongly invertible action and the $\iota_K$ action. Firstly, we show for a certain class of thin periodic knots the action is indeed unique. We refer to a thin knot as \textit{diagonally supported} if the \textit{model} Knot Floer chain complex of the knot (that arises from the work of Petkova \cite{Petkova}), does not have any boxes outside the main diagonal.  Where by a `box', we mean a part of the chain complex with $4$-generators, for which the differential forms a square shape. See Figure~\ref{figure_eight_periodic} for an example and Figure~\ref{square} for the definition. We show: 
\begin{theorem}\label{thin_periodic}
Let $(K,\tau)$ be a diagonally supported Floer homologically thin knot with a periodic symmetry $\tau$. Then we can explicitly identify this action on $CFK^{\infty}(K)$. In fact, the action of any other periodic symmetry $\tau^{\prime}$ of $K$ on $CFK^{\infty}(K)$ is equivalent to that of $\tau$. 
 \end{theorem}
\noindent
Where by equivalent, we mean that they are conjugate via an $U$-equivariant change of basis.
 Figure~\ref{figure_eight_periodic} demonstrates a periodic symmetry $\tau^{p}$ on the figure-eight knot, which is a diagonally supported thin knot. The computation of the periodic action $\tau^{p}_K$ on $CFK^{\infty}$ is shown in the middle. We also apply Theorem~\ref{t1} to compute the action of the induced symmetry $\tau^{p}$ on $HF^{-}(S_{+1}(4_1))=\Sigma(2,3,7)$.

We remark that diagonally supported thin knots with a periodic symmetry (i.e. satisfying the hypothesis of Theorem~\ref{thin_periodic}) exist in abundance. To see one such class of knots, recall the following from \cite{OSknots} 
\[
g(K)=max \{|s|: \widehat{HFK}(K,s) \neq 0 \}.
\]  
It can be checked that the above implies that any thin knot $K$ with genus $g(K) \leq 1$ must be diagonally supported. There are many periodic thin knots of genus 1, for example, the $3$-strand pretzel knots with an odd number of half-twists in each strand have such a property. Figure~\ref{pretzel} illustrates the periodic symmetry. Theorem~\ref{thin_periodic} then implies that we can compute the action of $\tau_K$ on such knots. However, it is not known to the author whether the action of a periodic symmetry on any thin knot is always unique (see Remark~\ref{thinremark}).

\begin{figure}[h!]
\centering
\includegraphics[scale=.47]{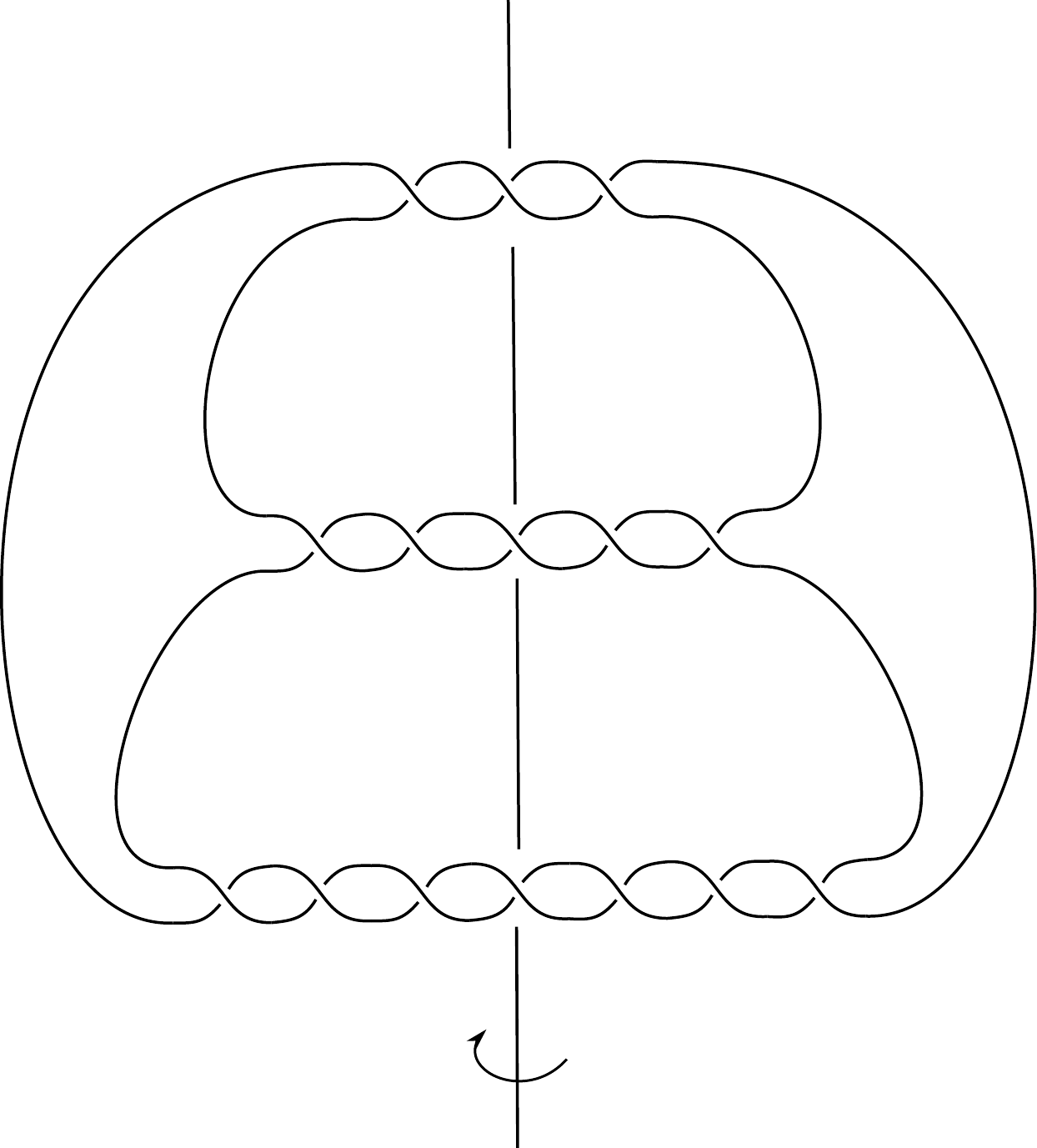} 
\caption
{A diagonally supported thin knot, pretzel $P(3,5,7)$, with a periodic symmetry.}\label{pretzel}
\end{figure}

\subsection{Concordance}
Using symmetry, we also define two knot concordance invariants. These are defined by considering strong symmetry on the knot $K \# \tau K^{r}$ as shown in Figure~\ref{KconnectK}. Using the $\underline{d}_{\circ}$ and $\bar{d}_{\circ}$ type invariants, we show the following 
\begin{theorem}\label{concordance}
 Let $K$ be a knot in $S^{3}$, $\tilde{\tau}_{sw}$ be the strong involution on the knot $K \# \tau K^{r}$ shown in Figure~\ref{KconnectK}. Then the equivariant correction terms $\underline{d}_{\tau}(A^{-}_{0}(K \# \tau K^{r}), \tilde{\tau}_{sw})$, $\bar{d}_{\tau}(A^{-}_{0}(K \# \tau K^{r}), \tilde{\tau}_{sw})$, $\underline{d}_{\iota \tau}(A^{-}_{0}(K \# \tau K^{r}), \tilde{\tau}_{sw})$, $\bar{d}_{\iota \tau}(A^{-}_{0}(K \# \tau K^{r}), \tilde{\tau}_{sw})$ are concordance invariants. 
\end{theorem}
\noindent
As mentioned before, the $\iota$-action in conjunction with the $\tau$-action has been useful in studying various questions in an equivariant setting. Indeed, this strategy is used extensively in \cite{dai2020corks} and \cite{abhishek2022equivariant}. However, here with $\underline{d}_{\iota \tau}$ and $\bar{d}_{\iota \tau}$, we use the $\iota$ and $\tau$-action together to study objects in the non-equivariant setting, namely knot concordance. The invariants from Theorem~\ref{concordance} can also be interpreted as the equivariant correction terms of double branched covers of certain cables of $K$, see discussion in Remark~\ref{knotconcordanceremark}.

Theorem~\ref{KconnectKstrong} coupled with Theorem~\ref{t1},  implies that we can compute the invariants above for any knot $K$. Leveraging this, we show that some of the invariants defined in Theorem~\ref{concordance} differ from their involutive Heegaard Floer counterpart, see Subsection~\ref{KconnectKcompute}. Although explicit examples illuminating the usefulness of these invariants  (where their involutive Heegaard Floer counterpart fail) is currently not known to the author (see Remark~\ref{questiontausw}).

\subsection{An invariant of equivariant knots}

The study of Equivariant knots up to conjugacy has received much attention in the literature. For example, they were studied using twisted-Alexander polynomial \cite{hillman2006twisted}, double branched covers \cite{naik1997new}, HOMFLYPT polynomials \cite{traczyk1991periodic}.
More recently, modern techniques were also used to study such knots. To study them, Hendricks used spectral sequence in link Floer homology \cite{hendricks2012localization}, Jabuka-Naik used $d$-invariant from Heegaard Floer homology \cite{jabukaperiodic}, Borodzik-Politarczyk used equivariant Khovanov homology \cite{borodzik2017khovanov}, Stoffregen-Zhang used annular Khovanov homology \cite{stoffregen2018localization} and Boyle-Issa used Donaldson's theorem \cite{boyle2021equivariant}. Using inputs from Khovanov homology, Watson \cite{watson2017khovanov} defined an invariant of strongly invertible knots which takes the form of a graded vector space (see also Lobb-Watson \cite{lobb2021refinement}).

In this article, we define an invariant of equivariant knots using the formalism of involutive Heegaard Floer homology.  The invariant manifests as a quasi-isomorphism class of $\mathbb{Z} \oplus \mathbb{Z}$-filtered complexes over $\mathbb{Z}_{2}[Q, U, U^{-1}]/(Q^{2})$. This invariant, although defined based on involutive techniques, defers from its involutive counterpart as we describe below.

In \cite{HM}, Hendricks and Manolescu defined an invariant $CI(Y,K,w,z)$ by taking into account the $\iota_K$ action on the knot. However, the invariant $CI(Y,K,w,z)$ lacks the $\mathbb{Z} \oplus \mathbb{Z}$ that is present in the usual knot Floer complex $CFK^{\infty}$, hence it was not as useful as its $3$-manifold counterpart. By considering the actions of $\tau_K$ for an equivariant knot, we define a similar invariant, which does admit a $\mathbb{Z} \oplus \mathbb{Z}$-filtration. 
\begin{proposition}\label{knotinvariant}
Let $(Y,K, \tau, w,z)$ be an equivariant doubly-based knot $(K,w,z)$ in an $\mathbb{Z}HS^{3}$, $Y$. If $\tau$ is a periodic action then the quasi-isomorphism class of $CI_{\tau}(Y,K,w,z)$, a $\mathbb{Z} \oplus \mathbb{Z}$-filtered chain complex over $\mathbb{Z}_{2}[Q, U, U^{-1}]/(Q^{2})$, is an invariant of the conjugacy class of $(Y,K, \tau, w,z)$. Likewise, if $\tau$ is a strongly invertible involution then $CI_{\iota \tau}(Y,K,w,z)$ is an invariant of the conjugacy class of $(Y,K, \tau, w,z)$.
\end{proposition}

We refer readers to Section~\ref{knotCI} for the definitions of these invariants. The proof of Proposition~\ref{knotinvariant} is rather straightforward and follows from the naturality results. We also make use of the computations of the $\tau_K$ actions on the knot Floer complex from Section~\ref{computation} together with the equivariant (and $\iota$-equivariant) surgery formula to obtain a few simple calculations of these invariants. For example, these invariants are trivial for equivariant L-space knots (and their mirrors) but are potentially non-trivial for equivariant thin knots. For example, $CF_{\tau}$ is non-trivial for figure-eight knots. In the interest of studying the equivariant knots up to its conjugacy class, one may hope to exploit the $\mathbb{Z} \oplus \mathbb{Z}$-filtrations present $CI_{\circ}$ further.

\subsection{Organization}
This document is organized as follows. In Section~\ref{involutionaction} we define the action of an involution on the Heegaard Floer chain complex of a $3$-manifold and on the Knot Floer complex of a knot. In Section~\ref{spingroup} we define the group $\G$. In Section~\ref{correction}, we prove invariance of $\underline{d}_{\circ}$ and $\bar{d}_{\circ}$. We then follow it up by proving the equivariant (and $\iota$-equivariant) large surgery formula in Section~\ref{proofsurgery}. We go on with the computations for strong involution on $K \# \tau K^{r}$ and the periodic involution on thin knots in Section~\ref{computesection}. In Section~\ref{concordanceinvariant}, we prove the concordance invariance. Section~\ref{computation} is devoted to the explicit computation of the action and the equivariant correction terms for many examples and finally, in Section~\ref{knotCI}, we observe that $CI_{\circ}$ are equivariant knot invariants.

\subsection{Acknowledgment}
The author wishes to thank Irving Dai, Kristen Hendricks, and Matthew Stoffregen for helpful conversations, especially regarding Section~\ref{computesection}. The author is extremely grateful to Matthew Hedden, whose input has improved this article in many ways. This project began when the author was a graduate student at Michigan State University and finished when the author was a postdoctoral fellow in the Max Planck Institute for Mathematics. The author wishes to thank both institutions for their support.


\section{Actions induced by the involution}\label{involutionaction}

In this section, we define the action of an involution on a $3$-manifold (or a knot) on the corresponding Heegaard Floer chain complex (or the knot Floer chain complex). The definitions are essentially a consequence of the naturality results in Heegaard Floer homology, which are due to various authors Ozsv{\'a}th-Szab{\'o} \cite{OS3manifolds1, OSknots}, Juh{\'a}sz-Thurston-Zemke \cite{JTZ}, Hendricks-Manolescu \cite{HM}.

\subsection{Action on the 3-manifold}\label{3manifold}

We start with a tuple $(Y,\tau,z,\s)$ consisting of a connected oriented rational homology sphere $Y$, with a basepoint $z$ and an involution $\tau$ acting on it. We also let $\s$ represent a $\spinc$-structure on $Y$.
Furthermore, we require that $\tau$ fixes the $\spinc$ structure $\s$. 
 
Recall that the input for Heegaard Floer homology is the Heegaard data, $\mathcal{H}=(\Sigma,\alphas,\betas,z,J)$ for $(Y,z)$. We can then define a push-forward of this data by $\tau$, denoted as $\tau \mathcal{H}=(\tau \alphas, \tau \betas, \tau z, \tau J)$. This induces a tautological chain isomorphism \footnote{Although we write all the maps for the $\infty$-version of the chain complex, all the discussions continue to hold for $\{ -, + \}$-versions as well.}
\[
 t:CF^{\infty}(\mathcal{H},\s) \rightarrow CF^{\infty}(\tau \mathcal{H},\s),
 \]
obtained by sending an intersection point to its image under the diffeomorphism. If $\tau$ fixes the basepoint $z$ then $\mathcal{H}$ and $\mathcal{\tau \mathcal{H}}$ represent the same based $3$-manifold $(Y,z)$. The work of Juh\'asz-Thruston-Zemke \cite{JTZ} and Hendricks-Manolescu \cite{HM} then implies that there is a canonical chain homotopy equivalence induced by the Heegaard moves:
 \[
 \Phi(\tau \mathcal{H}, \mathcal{H}): CF^{\infty}(\tau \mathcal{H},\s) \rightarrow CF^{\infty}(\mathcal{H},\s).
  \]  
The mapping class group action of $\tau$ on the Heegaard Floer chain complex is then defined to be the composition of the above two maps:
\[
  \tau=\Phi(\tau \mathcal{H}, \mathcal{H},\s) \circ t : CF^{\infty}(\mathcal{H}) \rightarrow CF^{\infty}(\mathcal{H},\s).
\]
In the case where $\tau$ does \textit{not} fix the basepoint $z$, we define the action as follows. We take a diffeomorphism $h^{\gamma}:Y \rightarrow Y$ induced by the isotopy, taking $z$ to $\tau z$ along a path $\gamma$ connecting them. The diffeomorphism $\tau_{\gamma}:= h^{\gamma} \circ \tau $ then fixes the basepoint, and we define the $\tau$ action to be the action of $\tau_{\gamma}$. To this end we have the following:

\begin{proposition}\cite[Lemma 4.1]{dai2020corks}
Let $Y$ be a $\mathbb{Q}HS^{3}$ with an involution $\tau$.  Then the action $\tau: CF^{\infty}(\mathcal{H},\s) \rightarrow CF^{\infty}(\mathcal{H},\s)$ is a well-defined up to chain homotopy and is a homotopy involution. 

\end{proposition}
    
\begin{proof}\label{prop:tau}
This is shown in \cite[Lemma 4.1]{dai2020corks}. The discussion in there is phrased in terms of integer homology spheres which continue to hold in the realm of rational homology spheres, after observing that the $\pi_1$-action on chain complex is still $U$-equivariantly homotopic to the identity, \cite[Theorem D]{Zemkegraph}. 
 \end{proof}
\noindent 
Since $\tau$ is independent of the choice of the Heegaard data, we will avoid specifying it and represent $\tau$ as: 
\[
\tau: CF^{\infty}(Y,z,\s) \rightarrow CF^{\infty}(Y,z,\s).
\] 
It can also be easily checked that the maps $h^{\gamma}$ and $t$ both commute (up to homotopy) with the chain homotopy equivalence $\Phi$ induced by the Heegaard moves, interpolating between two different Heegaard data representing the same based $3$-manifold. Hence both maps $h^{\gamma}$ and $t$ do not depend on the choice of Heegaard data as well. This allows us to write $\tau$ as a composition:
\[
CF^{\infty}(Y,z,\s_{0}) \xrightarrow{t} CF^{\infty}(Y,tz,\s) \xrightarrow{h^{\gamma}_1} CF^{\infty}(Y,z,\s_{0}).
\]
Before moving on, we will digress and show that up to a notion of equivalence the choice of basepoint in the definition of $\tau$ can be ignored. 
\noindent
We now recall the definition of $\iota$-complexes and local equivalence below, these were first introduced by Hendricks-Manolescu-Zemke \cite{HMZ}.

\begin{definition}\label{iota}
An $\iota$-\textit{complex} is a pair $(C,\iota)$ where 

\begin{itemize}
 
\item $C$ is a finitely-generated, free, $\mathbb{Z}$-graded chain complex over $\mathbb{Z}_{2}[U]$ such that 

\[
U^{-1}H_{*}(C) \cong \mathbb{Z}_{2}[U,U^{-1}] 
\]

Where $U$ has degree 2,

\item  and $\iota: C \rightarrow C$ is a grading preserving $U$-equivariant chain map such that $\iota ^{2} \simeq \mathrm{id}$.

\end{itemize}
\end{definition} 
\noindent
In \cite{HMZ}, the authors defined an equivalence relation on the set of such $\iota$-complexes.

\begin{definition}\label{localequi}
Two $\iota$-complex $(C_1,\iota_1)$ and $(C_2,\iota_2)$ are said to be \textit{locally equivalent} if 

\begin{itemize}

\item There are grading preserving $U$-equivariant chain maps $f:C_1 \rightarrow C_2$ and $g: C_2 \rightarrow C_1$ that induce isomorphism in homology after localizing with respect to $U$,

\item and $f \circ \iota_1 \simeq \iota_2 \circ f$, \; $g \circ \iota_1 \simeq \iota_2 \circ g$.
\end{itemize}
\end{definition}

\noindent
We now have the following: 
 \begin{proposition}\label{localbase}
 The local equivalence class of $(CF^{-}(Y,z,\s),\tau)$ is independent of the choice of basepoint $z$.
 \end{proposition}
\begin{proof}
Let $z_1$ and $z_2$ be two different basepoints. We start by taking a finger-moving diffeomorphism taking $z_1$ to $z_2$ along a path ${\tilde{\gamma}}$. We also take analogous finger moving diffeomorphism $\gamma_{i}$ joining $\tau z_{i}$ to $z_{i}$. Following the notation from the previous discussion, we have the following diagram which commutes up to chain homotopy.
\[\begin{tikzcd}
{CF}^{-}(Y,z_1) \arrow{r}{h^{\tilde{\gamma}}} \arrow{d}{t} & {CF}^{-}(Y,z_2) \arrow{d}{t} \\
{CF}^{-}(Y,\tau z_1) \arrow{r}{h^{\tau \tilde{\gamma}}} \arrow{d}{h^{\gamma_{1}}}  & {{CF}^{-}(Y,\tau z_2)} \arrow{d}{h^{\gamma_{2}}} \\
{CF}^{-}(Y,z_1) \arrow{r}{h^{\tilde{\gamma}}}  & {{CF}^{-}(Y, z_2)}
\end{tikzcd}
\]
Here the upper square commutes up to chain homotopy because of diffeomorphism invariance of cobordisms proved by \cite[Theorem A]{Zemkegraph} and the lower square commutes, as a consequence of triviality of $\pi_1$-action up to chain homotopy for rational homology spheres \cite[Theorem D]{Zemkegraph}. 
The two vertical compositions define the $\tau$ action on the chain complex $CF^{-}(Y,z_1)$ and $CF^{-}(Y,z_2)$ respectively. The map $h^{\tilde{\gamma}}$ defines a local map between the respective $CF$ complexes. One can now similarly define local maps going the other way, by considering the inverse of the path $\tilde{\gamma}$. 
 
\end{proof}

 \subsection{Action on the knot}\label{action_knot}
 
For this subsection, we will restrict ourselves to knots in integer homology spheres. Let $(Y,K,\tau,w,z)$ be a tuple, where $(Y,K,w,z)$ represents a doubly-based knot $(K,w,z)$ embedded in $Y$, and $\tau$ is an orientation preserving involution on $(Y,K)$. We now consider two separate families of such tuples depending on how they act on the knot.

\begin{definition}
Given $(Y,K,\tau,w,z)$ as above, we say that $K$ is \textit{2-periodic} if $\tau$ has no fixed points on $K$ and it preserves the orientation on $K$. On the other hand, we will say that $K$ is a \textit{strongly invertible} if $\tau$ has two fixed points when restricted to $K$. Moreover, it switches the orientation of $K$.  
\end{definition}
Since we are dealing only with involutions in this paper we will abbreviate 2-periodic knots as periodic. Both periodic and strong involutions induce actions on the knot Floer complex. Let us consider the periodic case first.
 As before, we start with a Heegaard data $\mathcal{H}_{K}:=(\Sigma,\alphas, \betas, w,z)$. There is a tautological chain isomorphism:\footnote{In order to distinguish the action of $\tau_K$ from the action of push-forward of $\tau$, we represent the push-forward map by $t_K$. A similar notation will also be implied for $3$-manifolds.}
 \[
 t_{K}: CFK^{\infty}(\mathcal{H}_{K}) \rightarrow CFK^{\infty}(\tau \mathcal{H}_{K}).
  \]
 Now note that $\tau \mathcal{H}_{K}$ represents the same knot inside $Y$ although the basepoints $(w,z)$ have moved to $(\tau w,\tau z)$. So we apply a diffeomorphism $\rho$, obtained by an isotopy taking  $\tau z$ and $\tau w$ back to $z$ and $w$ along an arc of the knot, following its orientation. We also require the isotopy to be identity outside a small neighborhood of the arc. $\rho \tau \mathcal{H}_K$ now represents the based knot $(Y,K,w,z)$. So by work of Juh{\'a}sz-Thurston-Zemke \cite{JTZ}, Hendricks-Manolescu \cite[Proposition 6.1]{HM}, there is a sequence of Heegaard moves relating the $\rho \tau \mathcal{H}_K$ and $ \mathcal{H}_K$ inducing a chain homotopy equivalence, 
 \[
 \Phi(\rho\tau \mathcal{H}_K, \mathcal{H}_K) : CFK^{\infty}(\rho\tau \mathcal{H}_K) \rightarrow CFK^{\infty}(\mathcal{H}_K).
 \]
 We now define the $\tau$ action to be: 
 \[
 \tau_K:= \Phi(\rho \tau \mathcal{H}_K, \mathcal{H}_K) \circ t_{K}:CFK^{\infty}(\mathcal{H}_K) \rightarrow CFK^{\infty}(\mathcal{H}_K).
 \]
The chain homotopy type of $\tau_K$ is independent of the choice of Heegaard data. This is again a consequence of the naturality results proved in \cite{HM}. So, by abusing notation we will write $\tau_K$ as:
 \[
\tau_{K}: CFK^{\infty}(K,w,z) \rightarrow CFK^{\infty}(K,w,z).
\]
Sarkar in \cite{sarkar2015moving} studied a specific action on the knot Floer complex obtained by moving the two basepoints once around the orientation of the knot, which amounts to applying a full Dehn twist along the orientation of the knot. We refer to this map as the Sarkar map $\varsigma$. This map $\varsigma$ is a filtered, grading-preserving chain map which was computed by \cite{sarkar2015moving} and later in full generality for links by Zemke \cite{zemke2017quasistabilization}. In particular, we have $\varsigma^{2} \simeq \textrm{id}$ \cite{sarkar2015moving, zemke2017quasistabilization}. Analogous to the case for the $3$-manifolds, one can enquire whether $\tau_{K}$ is a homotopy involution. It turns out that, it is not a homotopy involution in general but $\tau_{K}^{4} \simeq \textrm{id}$, as a consequence of the following Proposition.

 \begin{proposition}\label{prop:tauK}
Let $Y$ be a $\mathbb{Z}HS^3$ and $(K,\tau,w,z)$ be a doubly-based periodic knot on it. Then $\tau_{K}$ is a grading preserving, filtered map that is well-defined up to chain homotopy and $\tau^{2}_{K} \simeq \varsigma$.
 \end{proposition}
 \begin{proof}
The proof is similar to that of Hendricks-Manolescu \cite[Lemma 2.5]{HM} after some cosmetic changes, so we will omit the proof. The main idea is that since the definition $\tau_{K}$ involves the basepoint moving map taking $(w,z)$ to $(\tau w ,\tau z)$, $\tau_{K}^{2}$ results in moving the pair $(w,z)$ once around the knot $K$ along its orientation back to $(w,z)$ which is precisely the Sarkar map. $\tau_K$ is grading preserving and filtered since all the maps involved in its definition are. Finally, one can check that $\tau_K$ is well-defined by showing each map involved in the definition of $\tau_K$ commutes with homotopy equivalence $\Phi$ induced by the Heegaard moves that interpolates between two different Heegaard data for the same doubly-based knot.
 \end{proof}
We now define the action for strong involutions. Note that in this case, $\tau$ reverses the orientation of the knot $K$. Since knot Floer chain complex is an invariant (up to chain homotopy) of oriented knots, we do not \textit{a priori} have an automorphism of $K$. However, it is still possible to define an involution on the knot Floer complex induced by $\tau$.

As before, we start by taking a Heegaard data $\mathcal{H}_K=(\Sigma,\alphas,\betas,w,z)$ for $(Y,K,w,z)$. Recall that  the knot $K$ intersects $\Sigma$ positively at $z$ and negatively $w$. Note $\mathcal{H}_{K^{r}}=(\Sigma,\alphas,\betas,z,w)$ then represents $(Y,K^{r},z,w)$. \footnote{Here $K^{r}$ represent $K$ with its orientation reversed.} Furthermore there is an obvious correspondence between the intersection points of these two diagrams. In order to avoid confusion for an intersection point in $\mathcal{H}_{K}$, $\x \in \mathbb{T}_{\alpha} \cap  \mathbb{T}_{\beta}$ we will write the corresponding intersection point in $\mathcal{H}_{K^{r}}$ as $\x^{\prime}$. Now note that there is a tautological grading preserving \textit{skew-filtered} chain isomorphism:
\[ 
sw: CFK^{\infty}(\mathcal{H}_{K}) \rightarrow CFK^{\infty}(\mathcal{H}_{K^{r}}),
\]
obtained by \textit{switching} the order of the basepoints $z$ and $w$. More specifically, recall that $CFK^{\infty}(\mathcal{H}_{K})$ is $\mathbb{Z}\oplus \mathbb{Z}$-filtered chain complex, generated by triples $[\x,i,j]$. So we define, 
\[
sw[\x,i,j]=[\x^{\prime},j,i].
\]
This map is \textit{skew-filtered} in the sense that if we take the filtration on the range to be 
\[
\bar{\mathcal{F}}_{z,w}([\x^{\prime},i,j])=(j,i),
\]
then $sw$ is filtration preserving. 

\noindent
Let us now define the action of $\tau$ on the knot Floer complex. To begin with, we place the basepoints $(w,z)$ in such a way that they are switched by the involution $\tau$. This symmetric placement of the basepoints will be an important part of our construction of the action. Let us denote the map in knot Floer complex, obtained by pushing-forward $\mathcal{H}_K$ by $\tau$, as $t_{K}$.
\[
t_{K}: CFK^{\infty}(\mathcal{H}_{K}) \rightarrow CFK^{\infty}(\tau \mathcal{H}_{K}).
\]
Then by \cite[Proposition 6.1]{HM} there is a chain homotopy equivalence $\Phi$ induced by a sequence of Heegaard moves connecting $\mathcal{H}_{K^{r}}$ and $\tau \mathcal{H}_{K}$ since they both represent the same based knot $(Y,K^{r},z,w)$. Finally, we apply the $sw$ map to get back to the original knot Floer complex,
\[
sw: CFK^{\infty}(\H_{K^{r}}) \rightarrow CFK^{\infty}(\H_{K}).
\]
The action $\tau_K$ on the knot Floer complex is then defined to be the composition of the maps above, 
\[
CFK^{\infty}(\mathcal{H}_{K}) \xrightarrow{t_{K}} CFK^{\infty}(\tau \mathcal{H}_{K}) \xrightarrow{\Phi} CFK^{\infty}(\mathcal{H}_{K^{r}}) \xrightarrow{sw} CFK^{\infty}(\mathcal{H}_{K}).
\]
Note that all the maps involved in the definition above are well defined i.e. independent of the choice of the Heegaard diagram up to chain homotopy. For example, one can check that for any two different choices $\mathcal{H}_{K}$ and $\mathcal{H}^{\prime}_{K}$ of Heegaard data for $(Y,K,w,z)$, the $sw$ map commutes with the homotopy equivalence induced by the Heegaard moves $\Phi$. 
\[
\begin{tikzcd}[column sep =large, row sep =large]
CFK^{\infty}(\mathcal{H}_{K}) \arrow{r}{sw} \arrow{d}{\Phi} & CFK^{\infty}(\mathcal{H}_{K^{r}}) \arrow{d}{\overline{\Phi}} \\
CFK^{\infty}(\mathcal{H}^{\prime}_{K}) \arrow{r}{sw}  & CFK^{\infty}(\mathcal{H}^{\prime}_{K^{r}})
\end{tikzcd}
\]
\noindent Here $\overline{\Phi}$ represents the map induced from the moves used to define $\Phi$, after switching the basepoints. More precisely, the Heegaard moves constitute stabilizations(destabilizations), the action of based diffeomorphisms that are isotopic to identity, and isotopies and handleslides among $\alpha$, $\beta$-curves in the complement of the basepoint. Commutation is tautological for based stabilizations. For other types of Heegaard moves, the induced map $\Phi$ is given by counting pseudo-holomorphic triangles $\psi$, for which the commutation of the above diagram is again tautological.

Using this observation, we find it useful to express the $\tau_{K}$ as the composition of maps of transitive chain homotopy category, without referencing the underlying Heegaard diagrams used to define these maps
\[
CF^{\infty}(Y,K,w,z) \xrightarrow{t_{K}} CF^{\infty}(Y,K^{r},\tau w,\tau z) \xrightarrow{sw} CF^{\infty}(Y,K,w,z).
\]
In contrast to the case for periodic knots, we show that $\tau_{K}$ is a homotopy involution.

\begin{proposition}\label{strong}
Let $Y$ be a $\mathbb{Z}HS^3$ and $(K,\tau,w,z)$ be a doubly-based strongly invertible knot in it. Then the induced map $\tau_{K}$ a well defined map up to chain homotopy. Furthermore it is a grading preserving skew-filtered involution on $CFK^{\infty}(Y,K)$, hence we get $\tau_{K}^{2} \simeq \textnormal{id}$.
\end{proposition}

\begin{proof}
The discussion above shows that the definition of $\tau_K$ is independent of the choice of a Heegaard diagram up to homotopy equivalence. Since the map $sw$ is skew-filtered, $\tau_K$ is skew-filtered.
We now observe that push-forward map commute with the $sw$ map.
\[
sw \circ t_{K} \simeq t_{K} \circ sw
\]
Additionally, since $t_K^{2} \simeq \mathrm{id}$, $sw^{2}=\mathrm{id}$, we get $\tau^{2} \simeq \mathrm{id}$.

\end{proof}
\noindent
Readers may wonder if we had placed the basepoints $(w,z)$ asymmetrically, whether that would affect the Proposition~\ref{strong}. Indeed, there are several different options for the placement of the basepoints. The action $\tau_K$ and the proof of Proposition~\ref{strong} for all such cases are analogous to the one discussed. Instead of exhaustively considering all possible placement of basepoints for the action, let us only demonstrate the case where the basepoints $(w,z)$ are placed asymmetrically and close to each other as in Figure~\ref{strongfigure} (In Subsection~\ref{KK} we also define the action when the basepoint are fixed by $\tau$).
\begin{figure}[h!]
\centering
\includegraphics[scale=.8]{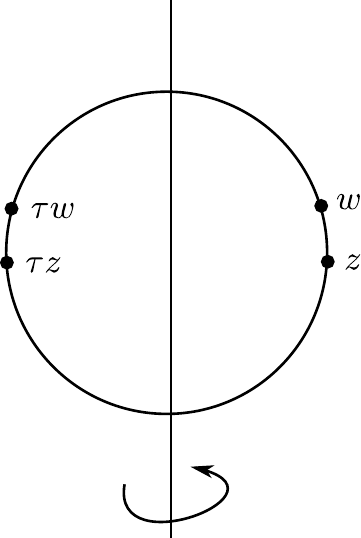} 
\caption
{Asymmetric placement of basepoints.} \label{strongfigure}
\end{figure}
For such a placement of basepoint, the action of $\tau_K$ is defined similarly for the symmetrically placed basepoint case, except now we need a basepoint moving map $\rho$. Specifically, the definition is given as the composition of the maps below
\[
CFK^{\infty}(\mathcal{H}_{K}) \xrightarrow{t_{K}} CFK^{\infty}(\tau \mathcal{H}_{K}) \xrightarrow{\rho} CFK^{\infty}(\rho \tau\mathcal{H}) \xrightarrow{\Phi} CFK^{\infty}(\mathcal{H}_{K^{r}}) \xrightarrow{sw} CFK^{\infty}(\mathcal{H}_{K})
\]
Here $\rho$ represents the basepoint moving map along the orientation of $K^{r}$ taking $(\tau w, \tau z)$ to $(z,w)$. Now since $\rho \tau\mathcal{H}$ and $\H_{K^{r}}$ both represent the same double based knot $(K^{r},z,w)$, there is a chain homotopy equivalence $\Phi$ between them induced by Heegaard moves relating the two Heegaard diagrams. As before, all of the maps in the composition above are independent of the choice of the Heegaard diagram used to define them up to chain homotopy. We now show that despite the appearance of basepoint moving map $\rho$, the action $\tau_K$ (defined in the way above) is a homotopy involution. This is in contrast to the situation with $\iota_K$ (the $\spinc$-conjugation action on the knot) or $\tau_K$ for the periodic symmetry action, both of those maps square to the Sarkar map $\varsigma$ (see Proposition~\ref{prop:tauK}).

\begin{proposition}\label{strongskewbasepoint}
Let $Y$ be a $\mathbb{Z}HS^3$ and $(K,\tau,w,z)$ be a doubly-based strongly invertible knot in it, Moreover the basepoints $(w,z)$ in $K$ are placed as in Figure~\ref{strongfigure}. Then the induced map $\tau_{K}$ a well defined map up to chain homotopy. Furthermore, it is a grading preserving skew-filtered involution on $CFK^{\infty}(Y,K)$, $\tau_{K}^{2} \simeq \textnormal{id}$.
\end{proposition}

\begin{proof}
Firstly we note that as before we have $t_{K}$ and $sw$ commute tautologically. 
\[
sw \circ t_{K} \simeq t_{K} \circ sw.
\]
Let us now define the basepoint moving maps $\rho$, $\overline{\rho}$ and $\underline{\rho}$ by pushing the basepoints along the orientated arcs as in Figure~\ref{rhomaps}. The arcs are a part of the knot, but they are drawn as push-offs. Note that the underlying unoriented arc is the same for all three maps. In fact, the maps $\rho$ and $\underline{\rho}$ are exactly the same diffeomorphism. Hence it follows that
\[
\rho \circ sw \simeq sw \circ \underline{\rho}.
\]
Moreover, since $\tau$ sends the arc $\rho$ to $\overline{\rho}$, we get,
\[
 t_{K} \circ \overline{\rho} \simeq {\rho} \circ t_{K}.
\]
\begin{figure}[h!]
\centering
\includegraphics[scale=0.7]{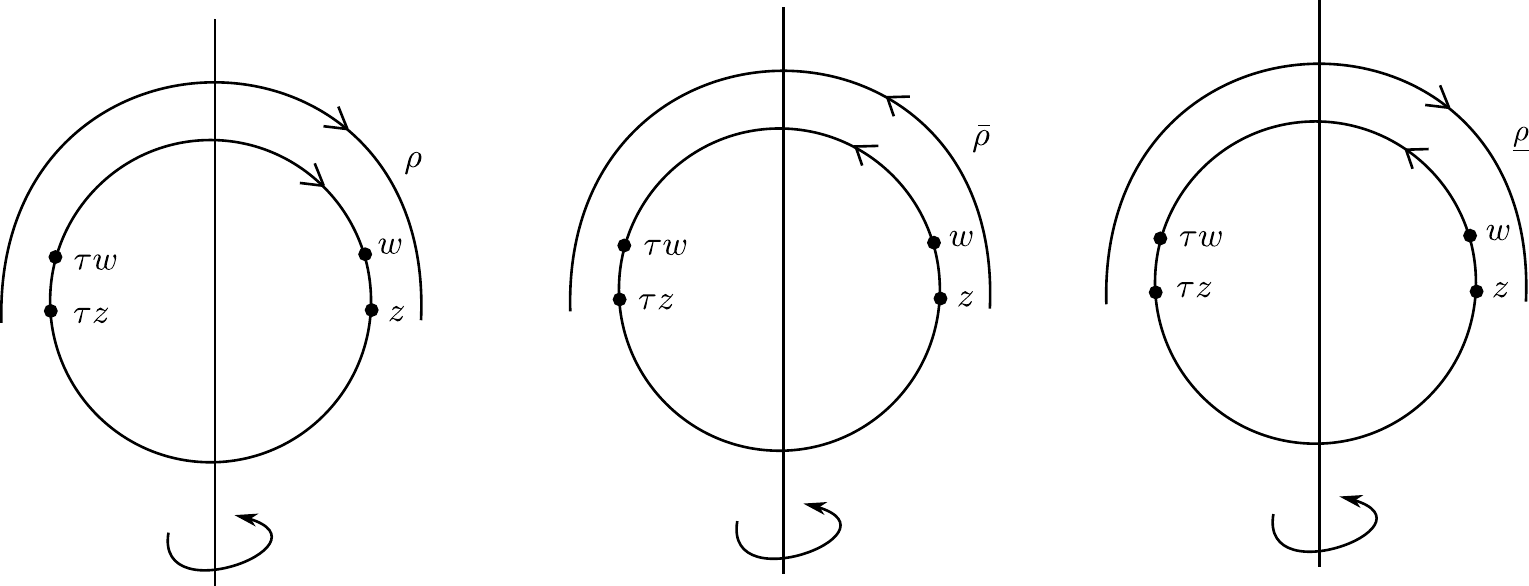} 
\caption{The orientation of the arcs used to define finger moving isotopy.}\label{rhomaps}
\end{figure} 


Using the above relations get the following sequence of chain homotopies
\begin{align*}
\tau^{2}_{K}&= sw \circ \rho \circ t_{K} \circ sw \circ \rho \circ t_{K} \\
&\simeq sw \circ \rho \circ t_{K} \circ \underline{\rho} \circ sw \circ t_{K} \\
&\simeq sw \circ t_{K}  \circ \overline{\rho}  \circ \underline{\rho} \circ sw \circ t_{K} \\
&\simeq  sw \circ t_{K} \circ sw  \circ t_{K}  \\
&\simeq \textrm{id}.
\end{align*}
Here we have used $\overline{\rho}  \circ \underline{\rho} \simeq \mathrm{id}$, since they are induced by the finger moving isotopy along the same arc but in opposite direction.
\end{proof}

\begin{remark}
Proposition \ref{strong} implies that $\tau_{K}$ is different from the involutive knot Floer action $\iota_{K}$ in the sense that although both are graded, skew-filtered maps, $\tau_{K}$ does not square to the Sarkar map. This is reflected in the Example~\ref{strong_figure_eight}, where as discussed, despite being a thin knot, there are two `different' strong involution actions on the figure-eight knot, compare \cite[Proposition 8.1]{HM}. On the other hand, the action of a periodic involution $\tau_K$ resembles $\iota_K$ in the sense that both are grading preserving maps that square to the Sarkar map, albeit $\tau_K$ is this case in filtered not skew-filtered. In the case of the figure-eight knot, the periodic action $\tau_K$ is different from the $\iota_K$ action \cite[Section 8.2]{HM}, see Figure~\ref{figure_eight_periodic}.
\end{remark}


\section{ Spin rational Homology bordism group of involutions: $\G$}\label{spingroup}

Two rational homology three-spheres $Y_1$ and $Y_2$ are said to be rational homology cobordant if there is a $4$-manifold $W$ such that inclusion $H_{*}(Y_i,\mathbb{Q}) \hookrightarrow H_{*}(W,\mathbb{Q})$ induce an isomorphism. The equivalence classes form a group under connected sum operation, called the rational homology cobordism group, $\Theta^{3}_{\mathbb{Q}}$. This group is well studied in the literature, see \cite{cassonrational, froyshov1996seiberg}. 
A slight variation of this group is the \textit{spin rational homology cobordism group}, $\Theta^{spin}_{\mathbb{Q}}$. The underlying set for this group are pairs $(Y,\s)$, where $Y$ a rational homology three-sphere equipped with a spin-structure $\s$. We identify two such pairs $(Y_1,\s_1)$ and $(Y_2,\s_2)$ if there exist a pair $(W,\mathfrak{t})$, where $W$ is a rational homology cobordism between $Y_1$ and $Y_2$ and $\mathfrak{t}$ is a spin structure on $W$ which restricts to the spin-structures $\s_i$ on the boundary.

In another direction, let $Y$ be a $3$-manifold equipped with a diffeomorphism $\phi$. One can relate any two Such pairs $(Y_{i},\phi_{i})$ if there is a pair $(W,f)$ where $W$ is a cobordism between $Y_1$ and $Y_2$ and $f$ is a diffeomorphism that restricts to the boundary diffeomorphisms $\phi_i$. The equivalence class of such pairs forms an abelian group under the operation of \textit{disjoint union}. This group is usually referred to as the \textit{$3$-dimensional bordism group of diffeomorphisms} $\Delta_{3}$. In \cite{Melvin}, Melvin showed that $\Delta_{3}=0$. This parallels the situation observed in the $3$-dimensional cobordism group. To this end, one can naturally ask to modify this group by putting homological restrictions on the cobordism. This parallels the situation in the traditional cobordism group, where after putting on homological restrictions on the cobordism one obtains the group of $3$-dimensional integer homology cobordism, $\Theta^{3}_{\mathbb{Z}}$ (or the rational homology cobordism group, $\Theta^{3}_{\mathbb{Q}}$), which has a lot more structure.

In \cite[Definition 2.2]{dai2020corks}, the authors generalized $\Delta_{3}$ and $\Theta^{3}_{\mathbb{Z}}$ to a group called the \textit{3-dimensional homology bordism group of involutions}, $\Theta^{\tau}_{\mathbb{Z}}$. Roughly, this was obtained by putting integer homology cobordism type homological restrictions on cobordisms of $\Delta_{3}$. Here analogously, we consider a generalization of $\Delta_3$ by putting homological restrictions similar to that in $\Theta^{spin}_{\mathbb{Q}}$.

\begin{definition}\cite[Definition 2.2]{dai2020corks} \label{spinclass} Let $(Y,\s,\tau)$ be a tuple such that $Y$ is a compact, disjoint union of oriented rational homology 3-spheres; $\tau$ is an involution which fixes each component of $Y$ set-wise and $\s$ is a collection of spin structures on each component of $Y$ each of which is fixed by $\tau$.


Given two such tuples $(Y_1,\tau_1,\s_1)$ and $(Y_2,\tau_2,\s_2)$ we say that they are \textit{pseudo spin $\mathbb{Q}$-homology bordant} if there exist a tuple $(W, \mathfrak{t}, f)$ where $W$ is compact oriented cobordism between $Y_1$ and $Y_2$ with $H_2(W;\mathbb{Q})=0$, $f$ is an orientation preserving diffeomorphsim on $W$ which extends the involutions $\tau_1$ and $\tau
_2$ on the boundary and $\mathfrak{t}$ is a spin structure on $W$ which restrict to $\s_i$ on $Y_i$. Furthermore, we require $f$ to satisfy 

\begin{enumerate}

\item $f$ acts as identity on $H_1(W,\partial W; \mathbb{Q})$;
\item $f$ fixes the spin-structure $\mathfrak{t}$. 

\end{enumerate}

\end{definition}

\noindent Readers may wonder about the motivation behind such a definition. Instead of belaboring the topic, we request interested readers to consult \cite[Section 2]{dai2020corks} for an extensive discussion. With the definition of pseudo spin $\mathbb{Q}$-homology bordism in mind, we define the group $\G$.

\begin{definition}\label{g}
The \textit{3-dimensional spin $\mathbb{Q}$-homology bordism group of involutions} $\G$, is an abelian group where underlying objects are pseudo $\mathbb{Q}$-homology bordism classes of tuples $(Y,\s,\tau)$ endowed with addition operation given by disjoint union. The identity is given by the empty set and the inverse is given by orientation reversal.
\end{definition}

Roughly, the readers may interpret the group $\G$ as the one obtained by decorating the boundaries of the cobordisms in $\Theta^{spin}_{\mathbb{Q}}$ with involutions that extend over the cobordism.
In the following Section~\ref{correction}, we will define and study invariants of this group.


\section{$\tau$-involutive correction terms}\label{correction}

Hendricks and Manolescu \cite{HM} studied the $\spinc$-conjugation action on the Heegaard Floer chain complex $CF^{-}(Y,\s)$, and using this action they also defined a mapping cone complex $CFI^{-}(Y,\s)$. Moreover, they showed that for a self-conjugate $\spinc$-structure $\s$, the quasi-isomorphism type of the chain complex complex $CFI^{-}(Y,\s)$  is an invariant of the $3$-manifold $Y$ \cite[Proposition 2.8]{HM}.
In a similar manner we define  
\begin{definition}
Given $(Y,\s, \tau)$ as in Section~\ref{3manifold}, we define a chain complex $CFI^{-}_{\tau}$ to be the mapping cone of 
\[
 CF^{-}(Y,\s) \xrightarrow{Q(1 + \tau)} Q.CF^{-}(Y,\s),
\] 
where $Q$ is a formal variable of degree $-1$ with $Q^2=0$.
\end{definition}

We refer to the homology of $CFI^{-}_{\tau}$ as $\tau$-\textit{involutive Heegaard Floer homology} $HFI^{-}_{\tau}$. Note that by construction,  $HFI^{-}_{\tau}$ is a module over $\mathbb{Z}_{2}[U,Q]/(Q^2)$. It follows that $HFI^{-}_{\tau}$ is an invariant of the pair $(Y,\s,\tau)$. More precisely, we say two such tuples $(Y_1,\tau_1,\s_1)$ and $(Y_2,\tau_2,\s_2)$ are equivalent to each other if there exist a diffeomorphism
\[
\phi: (Y_1,\s_1) \rightarrow (Y_2, \s_2),
\]
which intertwines with $\tau_1$ and $\tau_2$ up to homotopy. We then have the following Proposition

\begin{proposition}
$HFI^{-}_{\tau}$ is an invariant of the equivalence class of $[(Y,\tau,\s)]$.
\end{proposition}

\begin{proof}
Given two equivalent tuples  $(Y_1,\tau_1,\s_1)$ and $(Y_2,\tau_2,\s_2)$ , the diffeomorphism invariance of $3+1$-dimensional cobordisms \cite[Theorem A]{Zemkegraph} implies that we have the following diagram which commute up to chain homotopy. 
\[\begin{tikzcd}[column sep =large, row sep =large]
CF^{-}(Y_1,\s_1) \arrow{r}{\phi} \arrow{d}{\tau_1} & CF^{-}(Y_2,\s_2) \arrow{d}{\tau_2} \\
CF^{-}(Y_1,\s_1) \arrow{r}{\phi}  & CF^{-}(Y_2,\s_2)
\end{tikzcd}
\]
The result follows.
\end{proof}
\noindent
Hendricks, Manolescu, and Zemke studied \cite{HMZ} two invariants $\underline{d}$ and $\bar{d}$ of the 3-dimensional homology cobordism group $\Theta^{3}_{\mathbb{Z}}$ stemming from the involutive Heegaard Floer homology. By adapting their construction we define two invariants $\underline{d}_{\tau}$ and $\bar{d}_{\tau}$. The definition of these invariants is analogous to that of the involutive Floer counterpart (by replacing $\iota$ with $\tau$). Instead of repeating the definition, we refer the readers to \cite[Lemma 2.12]{HMZ} for a convenient description of the invariants. However, we need to specify the definition of $\underline{d}_{\tau}$ and $\bar{d}_{\tau}$ when the $Y$ is disconnected, as this is not considered in the involutive case. We define the following
\begin{definition}
Let $(Y_1 \sqcup Y_2, \s_1 \sqcup \s_2, \tau_1 \sqcup \tau_2)$ be a tuple so that $Y_i$ is a $\mathbb{Q}HS^{3}$, $\s_i$ is a spin structure on $Y_i$ and $\tau_i$ is an involution on $Y_i$ which fix $\s_i$. We define
\[
\underline{d}_{\tau}(Y_1 \sqcup Y_2, \s_1 \sqcup \s_2, \tau_1 \sqcup \tau_2):= \underline{d}_{\tau}(\otimes CF^{-}(Y_i,\s_i)[-2], \otimes \tau_i).
\]
here the tensor product for $CF^{-}$ is taken over $\mathbb{Z}_{2}[U]$. $\bar{d}_{\tau}$ is defined analogously.
\end{definition}
\noindent
This definition is reminiscent of the equivariant connected sum formula proved in \cite[Proposition 6.8]{dai2020corks}. Recall that in Definition~\ref{spinclass}, we defined pseudo spin $\mathbb{Q}$-homology bordant class. We now show:

\begin{lemma}\cite[Section 6]{dai2020corks}\label{dinvariantproof}
$\underline{d}_{\circ}$ and $\bar{d}_{\circ}$ are invariants of pseudo spin $\mathbb{Q}$-homology bordant class. Where $\circ \in \{ \tau, \iota \tau \}$.
\end{lemma}

\begin{proof}
The proof is essentially an adaptation of the proof of Theorem~1.2 from \cite{dai2020corks}. We give a brief overview of the proof here for the convenience of the readers. We will only consider the case for $\tau$. The argument for $\iota \tau$ is similar. Let $(Y_1,\tau_1,\s_1)$ and $(Y_2,\tau_2,\s_2)$ be in the same equivalence class. Which implies that there is a pseudo spin $\mathbb{Q}$-homology bordism $(W,f,\s_{\scaleto{W}{3pt}})$ between them. For simplicity we assume that $W$ is connected. Zemke in \cite{Zemkegraph} defined \textit{graph cobordism} maps for Heegaard Floer homology. These maps require an extra information (compared to the cobordism maps defined by Ozsv{\'a}th and Szab{\'o} \cite{OSsmooth4}) in the form of an embedded graph $\Gamma$ in the cobordism $W$ with ends in the boundary of cobordism. \footnote{the ends of $\Gamma$ are the basepoints for $Y_1$ and $Y_2$, which we have chosen to omit from our notation for simplicity.} Using this data Zemeke defined the graph cobordism maps
\[F_{W,\Gamma,\s_{\scaleto{W}{3pt}}}:CF^{\circ}(Y_1,\s_1) \rightarrow CF^{\circ}(Y_1,\s_1).
\] 
The diffeomorphism invariance of graph cobordism maps \cite[Theorem A]{Zemkegraph} then implies 
\[
 F_{W,\Gamma,\s_{\scaleto{W}{3pt}}} \circ \tau_2 \simeq \tau_1 \circ F_{W, f(\Gamma),\s_{\scaleto{W}{3pt}}}.
 \]
In \cite{dai2020corks} it was shown that if $(W,f)$ is a \textit{pseudo-homology bordism} \cite[Definition 2.2]{dai2020corks} between the two boundaries of $W$ (which are integer homology spheres) then $F_{W,f(\Gamma)}$ is $U$-equivariant chain homotopic to $F_{W,\Gamma}$. However the proof is  easily refined for the case in hand. That the boundaries were integer homology spheres was used to imply that for any closed loop $\gamma$, the \textit{graph-action} map satisfy $A_{\gamma} \simeq 0$ \cite[Lemma 5.6]{Zemkegraph}. This continue to hold for rational homology spheres as well, since $A_{\gamma}$ only depend on the class of $[\gamma] \in H_{1}(Y)/\textrm{Tors}$. Hence following Section 6 from \cite{dai2020corks}, we get the following diagram which commutes up to homotopy.  
 \[
 \begin{tikzcd}[column sep =large, row sep =large]
CF^{\circ}(Y_1,\s_1) \arrow[dr, dashrightarrow, "H"] \arrow{r}{F_{W,\Gamma,\s_{\scaleto{W}{3pt}}}} \arrow{d}{\tau_1} & CF^{\circ}(Y_2,\s_2) \arrow{d}{\tau_2} \\
CF^{\circ}(Y_1,\s_1) \arrow{r}{F_{W,\Gamma, \s_{\scaleto{W}{3pt}}}}  & CF^{\circ}(Y_2,\s_2)
\end{tikzcd}
\]
Since $F_{W,\Gamma, \s_{W}}$ induces an isomorphism on $CF^{\infty}$ (see  proof of \cite[Theorem 1.2]{dai2020corks}), the map on the mapping cones  
\[ F^{\tau,\infty}_{W,\Gamma,\s_{\scaleto{W}{3pt}}}: HFI^{\infty}_{\tau}(Y_1,\tau_1,\s_1) \rightarrow HFI^{\infty}_{\tau}(Y_2,\tau_2,\s_2),
\]
induced by the chain homotopy $H$ above is a quasi-isomorphism. Now the maps $F^{\tau,\infty}_{W.\Gamma, \s_{\scaleto{W}{3pt}}}$ admit a grading shift formula. In \cite[Lemma 4.12]{HM}  Hendricks-Manolescu computed similar grading shift formula in the case of cobordism maps in involutive Floer homology. It can be checked that the same grading formula holds for our case as well. 
 
In order to finish our argument, it will be important to consider the construction of the cobordism map $F_{W,\Gamma, \s_{\scaleto{W}{3pt}}}$. Roughly, the cobordism $W$ decomposes into 3 parts as shown in Figure~\ref{cobordism}. The $W_{a}$ and $W_b$ parts consist of attaching one and three handles respectively with the embedded graph as shown in the Figure.
\begin{figure}[h!]
\centering
\includegraphics[scale=.5]{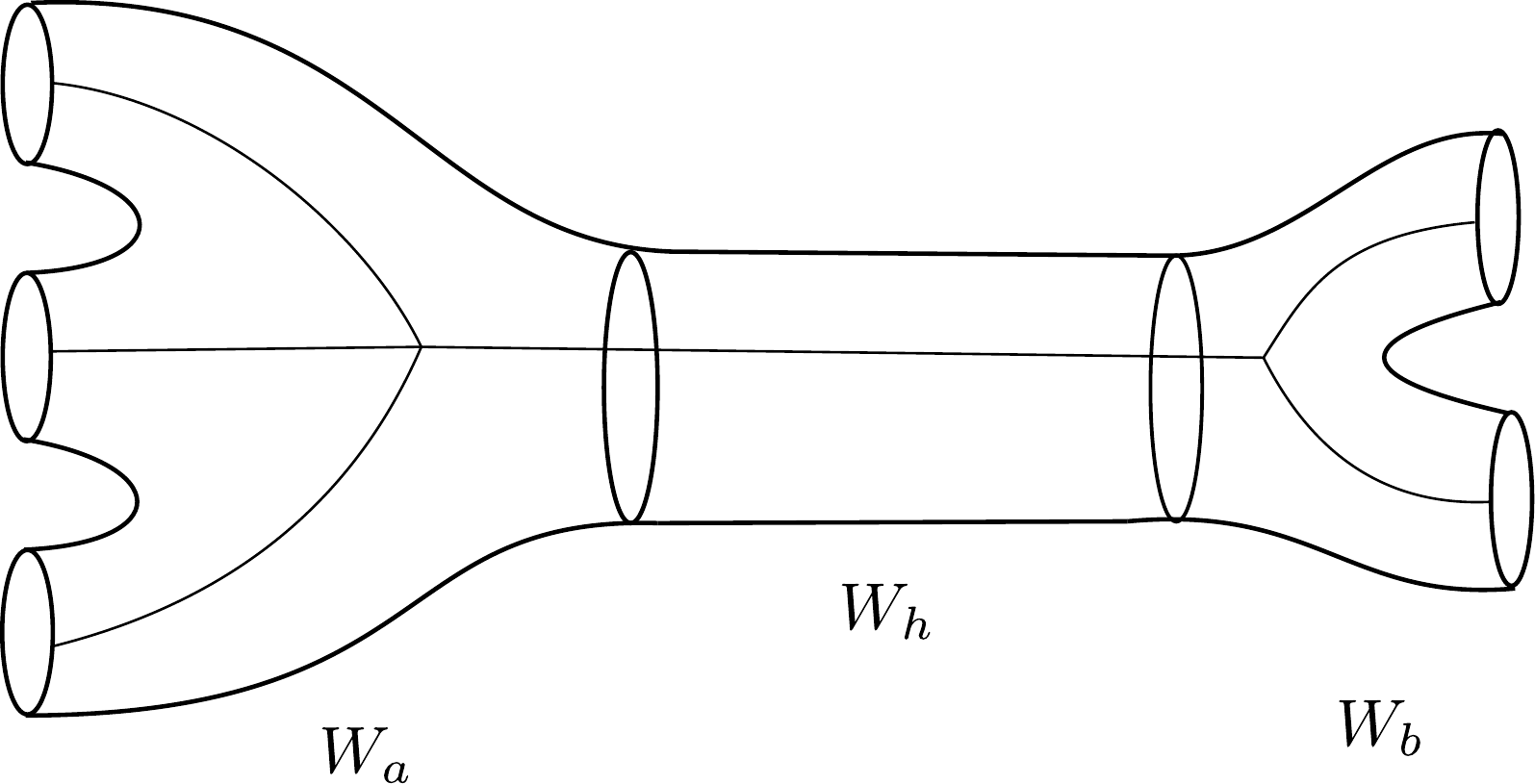} 
\caption{Decomposing $W$ in $3$-parts.}\label{cobordism}
\end{figure}
The part $W_h$ in the middle is then a rational homology cobordism, which is further equipped with a path. As discussed in \cite{dai2020corks} the cobordism maps associated with $W_{a}$ and $W_{b}$ are the ones inducing the connected sum cobordism. Hence the equivariant connected sum formula from \cite[Proposition 6.8]{dai2020corks} imply that the $\underline{d}_{\tau}$ and $\bar{d}_{\tau}$ invariants of the incoming boundary of $W_{a}$ and the outgoing boundary of $W_{a}$ are equal. Similar result holds $W_{b}$. Applying the grading shift formula to $W_h$, we get
\begin{gather*}
\underline{d}_{\tau}(Y_2,\tau_2,\s_2) - \underline{d}_{\tau}(Y_1,\tau_1,\s_1) \geq \frac{c_{1}^{2}(\s_{\scaleto{W|_{{\scaleto{W_{h}}{7pt}}}}{6pt}})-\chi({W_{h}})-\sigma(W_{h})}{4} \\
\overline{d}_{\tau}(Y_2,\tau_2,\s_2) - \overline{d}_{\tau}(Y_1,\tau_1,\s_1) \geq \frac{c_{1}^{2}(\s_{\scaleto{W|_{{\scaleto{W_{h}}{7pt}}}}{6pt}})-\chi({W_{h}})-\sigma(W_{h})}{4}
\end{gather*}
 
\noindent Since $W_{h}$ is a rational homology cobordism, it follows that  $\underline{d}_{\tau}(Y_2,\tau_2,\s_2) \geq \underline{d}_{\tau}(Y_1,\tau_1,\s_1)$ and 
$\overline{d}_{\tau}(Y_2,\tau_2,\s_2) \geq \overline{d}_{\tau}(Y_1,\tau_1,\s_1)$. The conclusion follows by turning the cobordism $W$ around and applying the same argument.

\end{proof}

\begin{proof}[Proof of Theorem~\ref{dinvariant}]
The proof is a follows from Lemma~\ref{dinvariantproof}.
\end{proof}


 \section{Equivariant surgery formula}\label{proofsurgery}
In this section, we prove the surgery formula outlined in Theorem~\ref{t1}. We start with some background on the large surgery formula for the Heegaard Floer homology.
\subsection{Large surgery formula in Heegaard Floer homology}
  
Let $K$ be a knot inside an integer homology sphere, $Y$ and $p > 0$ be an integer. Let $Y_{p}(K)$ denote the $3$-manifold obtained by $p$-surgery on $K$. Let us fix a Heegaard data $\mathcal{H}_{K}=(\Sigma,\alphas,\betas,w,z)$ for $(Y,K,w,z)$,  such that one of the $\beta$-curves among $\beta_1, \beta_2, \cdots \beta_g$, say $\beta_{g}$ (here $g$ is the genus of the Heegaard surface $\Sigma$) represent the meridian of the knot $K$. The basepoints $w$ and $z$ are placed on either side of $\beta_{g}$. Recall from \cite{OSknots}, that given such a Heegaard diagram for the knot, it is possible to construct a Heegaard diagram for the surgered manifold $Y_{p}(K)$. This is done as follows. We introduce a new set of curves $\gamma_1, \gamma_2, \cdots \gamma_g$ on $\Sigma$  such that $\gamma_{i}$ for $i \in \{ 1, \cdots, g-1 \}$ is a small Hamiltonian translate of the corresponding $\beta$ curve $\beta_{i}$. While the curve $\gamma_g$ is obtained by winding a knot longitude in a neighborhood of the curve $\beta_{g}$ in $\Sigma$ a number times so that it represents $p$-surgery. This neighborhood of $\beta_g$ where the winding takes place is usally referred to as the \textit{winding region}. It follows that $\mathcal{H}_{p}:=(\Sigma,\alphas,\gammas,w)$ is a Heegaard diagram for $(Y_{p}(K),w)$. We also recall that the elements of $CFK^{\infty}(\mathcal{H})$ are generated by $[\bold{x},i,j]$ for $i , j \in \mathbb{Z}$ where $\x$ is an intersection point between the $\alpha$ and the $\beta$-curves. We now define a quotient complex $A^{+}_{s}$ of $CFK^{\infty}(\mathcal{H})$ generated by elements $[\bold{x}, i, j]$ such that $i \geq 0, \; j \geq 0$. 
\[ 
A^{+}_{s}(\mathcal{H}):= \{ [\bold{x},i,j] \in CFK^{\infty}(\mathcal{H}) \; \textrm{with} \; i \geq 0, \; j \geq s \}.
\]
We will also use the subcomplex $A^{-}_{0}$. Which is defined as follows
\[
A^{-}_{s}(\mathcal{H}):= \{ [\bold{x},i,j] \in CFK^{\infty}(\mathcal{H}) \; \textrm{with} \; i < 0, \; j < s \}.
\]
The large surgery formula in Heegaard Floer homology is stated as:
\begin{theorem}\cite{OSknots, rasmussenthesis}\label{large}
There exist integer $N \geq 0$ such that for $p \geq N$ and for $s \in \mathbb{Z}$, with $|s| \leq \frac{p}{2}$, there is an isomorphism of relatively graded $\mathbb{Z}_{2}[U]$-module
\[
HF^{+}(Y_{p}(K),[s]) \cong H_{*}(A^{+}_{s}(K)).
\] 
\end{theorem}
\noindent
In fact, one can show that the stated isomorphism holds for   $N \geq g(K) + |s|$, see for example \cite{OSknots}, \cite[Proposition 6.9]{HM}. We will only be interested in the $[0]$-$\spinc$ structure in this article.

The isomorphism in the proof of Theorem~\ref{large} is given as follows. Recall that we obtain $Y_{p}(K)$ from $Y$ by attaching a $2$-handle, hence we get an associated cobordism $W_{p}(K)$ from $Y$ to $Y_{p}(K)$. Let $W^{\prime}_{p}(K)$ denote the result of turning the reverse cobordism $-W_{p}(K)$ around, resulting in  a cobordism from $Y_{p}(K)$ to $Y$. Ozsv{\'a}th and Szab{\'o} then use $W^{\prime}_{p}(K)$ to define a triangle counting map 
\[
\Gamma^{\infty}_{p,0}([\bold{x},i])= \sum\limits_{\bold{y} \in \mathbb{T}_{\alpha} \cap \mathbb{T}_{\beta}} {\sum\limits_{\substack{\psi \in \pi_{2}(\bold{x},\Theta, \bold{y}) \\ \mu (\psi)=0 ,  n_{z}(\psi)=n_{w}(\psi) }} } \# \mathcal{M}(\psi) \cdot [\bold{y}, i- n_{w}(\psi), i- n_{z}(\psi)].
\]
Here $\Theta \in \mathbb{T}_{\beta} \cap \mathbb{T}_{\gamma}$ represents the top degree generator in homology of $HF^{-}(\#^{g} S^{1} \times S^{2},\s_{0})$ where $\s_{0}$ is the torsion $\spinc$-structure. The map $\Gamma^{\infty}$ sends the $CF^{-}$-subcomplex to the $A^{-}_{0}$-subcomplex, hence induces a map on the quotient complexes 
\[
\Gamma^{+}_{p,0}: CF^{+}(\mathcal{H}_{p}) \rightarrow A^{+}_{0}(\mathcal{H}).
\]
It turns out that $\Gamma^{+}_{p,0}$ is chain isomorphism if $p$ is large enough. The main idea of the proof is to show that there is a small triangle counting map that induces an isomorphism between the two complexes (though it is not necessarily a chain map) and to use an area filtration argument to conclude $\Gamma^{+}_{p,0}$ is indeed a chain isomorphism. 

In \cite{HM}, Hendricks and Manolescu showed that the $\spinc$-conjugation map $\iota$ acting on $CF^{+}(S^{3}_{p}(K),[0])$ is related to the analogous $\spinc$-conjugation map $\iota_K$ acting on $CFK^{\infty}(K)$ in the following sense. We briefly discuss the result below. Recall that $HFI^{+}(Y,[0])$ is the mapping cone complex 
\[
CF^{+}(Y,\s) \xrightarrow{Q(\mathrm{id} + \iota)}  Q.CF^{+}(Y,\s)[-1].
\]
Here $Q$ is a formal variable such that $Q^{2}=0$ and $[-1]$ denotes a shift in grading. Similarly, let $AI^{+}_{0}$ represent the mapping cone chain complex of the map 
 \[
 A^{+}_{0} \xrightarrow{Q(\mathrm{id} + \iota_K)} Q.A^{+}_{0}[-1].
\] 
Here, by abuse of notation, we use $\iota_K$ to represent the induced action of $\iota_K$ on the quotient complex $A^{+}_{0}$. Hendricks and Manolescu showed that:

\begin{theorem}\cite[Theorem 1.5]{HM}
Let $K$ be a knot in $S^{3}$, then for all $p\geq g(K)$ then there is  a relatively $\mathbb{Z}_{2}[U,Q]/(Q^{2})$-graded isomorphism
\begin{center}
$HFI^{+}(S^{3}_{p}(K),[0]) \cong H_{*}(AI_{{0}}^{+}).$
\end{center}
\end{theorem} 
\noindent
The underlying idea for their proof was to construct a sequence of Heegaard moves relating the Heegaard data for the surgered manifold to the $\spinc$-conjugated Heegaard data so that the sequence simultaneously induces a sequence relating the Heegaard data of the knot with the $\spinc$-conjugated Heegaard data for the knot. 

In these subsequent subsections, we prove an analogous formula for equivariant surgeries on symmetric knots. While we essentially follow the same underlying philosophy of constructing a sequence of Heegaard moves relating two Heegaard triples, our proof of equivariant surgery formula defers from the involutive one. We begin by discussing the topology of equivariant surgeries.

Recall that a knot $K$ with either a strong or a periodic involution, induce an involution on the 3-manifold $Y_{p}(K)$ obtained by doing $p$-surgery on it, where $p$ is any integer\footnote{Infact, this is true for any rational number.} (see \cite{Montesinos1975}, \cite[Section 5]{dai2020corks}). Moreover, by examining the extension of the involution to the surgery, we see that it fixes the orientation of the meridian (of the surgery solid torus) for a periodic involution and reverses the orientation for a strong involution (see \cite[Lemma 5.2]{dai2020corks}). Hence, after identifying the $\spinc$-structures on $Y_{p}(K)$ with $\mathbb{Z}/ p\mathbb{Z}$, we get that the $\spinc$ structure corresponding to $[0] \in \mathbb{Z}/p \mathbb{Z}$ is fixed by the involution regardless of the type of symmetry (strong or periodic) on the knot $K$.

 \subsection{Large surgery formula for Periodic Knots}\label{periodic}
We would now like to prove an analog of the large surgery formula in the context of periodic knots. Let us recall that given a doubly-based periodic knot $(K, \tau, w, z)$ there is an induced involution on the knot Floer homology as defined in Subsection~\ref{action_knot},
\[
\tau_{K}: CFK^{\infty}(K) \rightarrow CFK^{\infty}(K).
\]
Furthermore, we note that since this map is filtered, it maps the subcomplex $A^{-}_{0}(K)$ to itself, and hence induces an involution on the quotient complex, 
\[
\tau_{K}: A^{+}_{0}(K) \rightarrow A^{+}_{0}(K).
\]
We then define $AI_{0}^{+,\tau}(K)$ to be the mapping cone of the following map,
\[
Q(\textrm{id} + \tau_{K}): A^{+}_{0}(K) \rightarrow Q.A^{+}_{0}(K)[-1].
\]
On the other hand, from Subsection~\ref{3manifold}, we get the action of the induced involution (from the periodic symmetry on the knot) on the surgered manifold $S^{3}_{p}(K)$
\[
\tau: CF^{+}(S^{3}_{p}(K), [0]) \rightarrow CF^{+}(S^{3}_{p}(K), [0]).
\]
Hence we can define the invariant $HFI^{+}_{\tau}(S^{3}_{p}(K),[0])$ as in Section~\ref{correction}. We now prove the equivariant surgery formula for periodic knots.
\begin{lemma}\label{periodicproof}
Let $K \subset S^{3}$ be a periodic knot with periodic involution $\tau$. Then for all $p\geq g(K)$ as a relatively graded $\mathbb{Z}_{2}[Q,U]/(Q^{2})$-modules, we the following isomorphism
\begin{center}
$HFI^{+}_{\tau}(S^{3}_{p}(K),[0]) \cong H_{*}(AI_{0}^{+,\tau})$.
\end{center}
\end{lemma} 

\begin{proof}
 It is enough to show that the following diagram commutes up to chain homotopy $R^{+}_{0}$. Where ${CF}^{+}(S^{3}_{p}(K),[0])$ and $A^{+}_{0}$ are computed from a certain Heegaard diagram and $\Gamma^{+}_{p,0}$ is the chain isomorphism defined in the large surgery formula. 
   \[
 \begin{tikzcd}[column sep =large, row sep =large]
{{CF}^{+}(S^{3}_{p}(K),[0])} \arrow[dr, dashrightarrow, "{R^{+}_{p}}"] \arrow{r}{Q(\textrm{id} + \tau)} \arrow{d}{\Gamma^{+}_{p,0}} & {{CF}^{+}(S^{3}_{p}(K),[0])} \arrow{d}{\Gamma^{+}_{p,0}} \\
{A^{+}_{0}} \arrow{r}{Q(\textrm{id} + \tau_{K})}  &{A^{+}_{0}}
\end{tikzcd}
\]
\noindent 
The chain homotopy $R^{+}_{0}$ will then induce a chain map between the respective mapping cone complexes. Since the vertical maps are chain isomorphisms when $p$ is large, the induced map on mapping cones is a quasi-isomorphism by standard homological algebra arguments (for example, see \cite[Lemma 2.1]{integersurgeryos}).

Hence it suffices to construct a homotopy $R^{+}_{p}$ and the rest of the proof is devoted to showing the existence of such a homotopy. Firstly, we begin by constructing a specific Heegaard diagram suitable for our argument.

By conjugating $\tau$ if necessary, we can assume that $\tau$ is given by rotating $180$-degrees about an axis, \cite{waldhauseninvolutionen}. Following \cite{OSalternating}, we now construct an equivariant Heegaard diagram for the knot $K$. Readers familiar with the Heegaard diagram construction from \cite{OSalternating} may skip the following construction. We assign an oriented $4$-valent graph $G$ to the knot projection, whose vertices correspond to double points of the projection and the edges correspond to parts of the knots between the crossings. To each vertex, we assign a 4-punctured $S^{2}$ and decorate it by a $\beta$ curve and $4$ `arcs' of the $\alpha$ curves as in Figure~\ref{tube}.

\begin{figure}[h!]
\centering
\includegraphics[scale=.7]{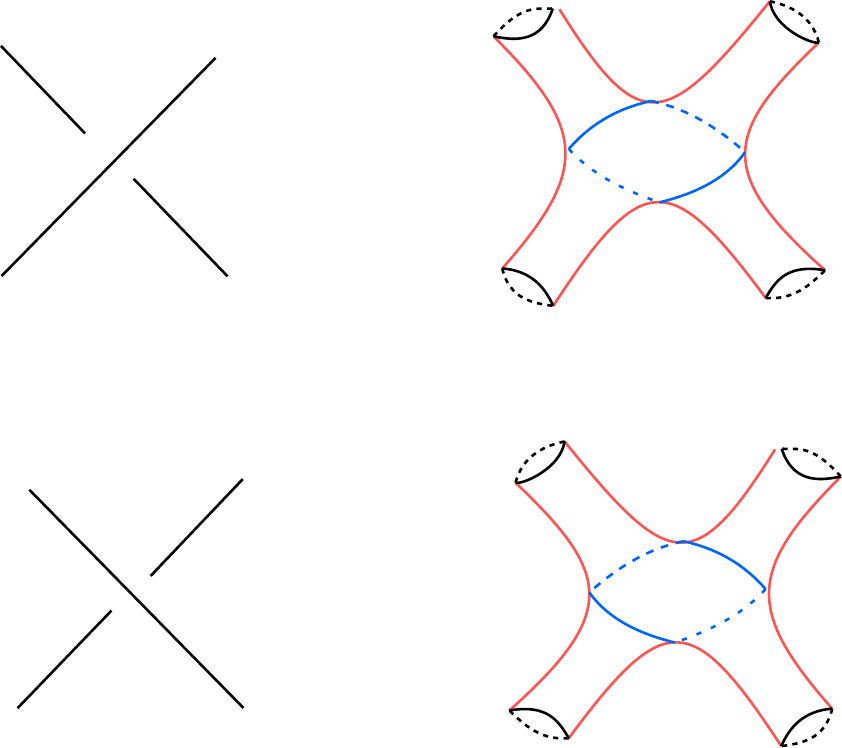} 

\caption
{Decoration of the punctured-$S^{2}$ with $\alpha$-curves in red and $\beta$-curves in blue.}\label{tube}
\end{figure}
\noindent
Next, we glue in tubes along the punctures in a manner that is compatible with the knot diagram. Moreover, we extend the arcs of the $\alpha$ curves to the tubes compatibly. Note that, there is always an outermost part of the graph that is homeomorphic to $S^{1}$, which occurs as the boundary of the unbounded region of the plane. We call this curve $c$. For all the vertices that appear in $c$, there are two edges going out of them, which are part of $c$. We omit drawing the $\alpha$-curves for those edges in the corresponding punctured spheres. Finally, we add a $\beta$-curve indexing it $\beta_{g}$, as the meridian of the tube corresponding to $c$. We then place the two basepoints $z$ and $w$ in either side of it $\beta_g$ so that joining $w$ to $z$ by a small arc in the tube corresponding to $c$, we recover the orientation of the corresponding edge. By \cite{OSalternating}, the resulting Heegaard diagram $\mathcal{H}_K :=(\Sigma, \alphas, \betas, w,z)$  is a doubly pointed Heegaard diagram for the knot $(K, w, z)$. We can visualize the knot lying in the surface by joining $w$ to $z$ in the complement of the $\alpha$-curves and $z$ to $w$ in the complement of the $\beta$-curves. Moreover note that, due to our construction $\tau \Sigma=\Sigma$. See Figure~\ref{trefoil_periodic} for an illustration of this construction. 

We now add the $\gamma$ curves to $\Sigma$. We let $\gamma_{i}$ be a small Hamiltonian isotopy of $\beta_{i}$ for $i \neq g $. We let $\gamma_{g}$ to be knot longitude that lies on $\Sigma$, except we wind it around $\beta_{g}$ in the winding region so that it represents the $p$-surgery framing.  Figure~\ref{winding} represent the winding region. We let $\mathcal{H}_{p} :=(\Sigma, \alphas, \gammas, w)$ represent the based surgered manifold $(S^{3}_{p}(K),w)$ and denote the resulting Heegaard triple by $\H_{\Delta}:= (\Sigma, \alphas, \betas, \gammas, w,z)$. 
\begin{figure}[h!]
\centering
\includegraphics[scale=.9]{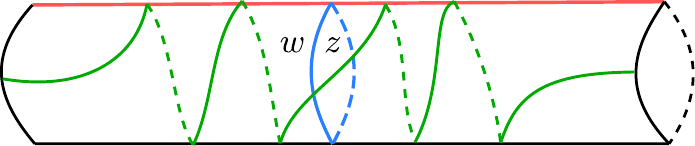} 

\caption
{The winding region.}\label{winding}
\end{figure}

\noindent Before moving on, let us recall the action of $\tau$ on $CFK(K)$ and on $CF(S^{3}_{p}(K),[0])$. Let $\rho$ represent the diffeomorphism induced by a finger moving isotopy along an arc of the knot which sends the pair $(\tau w, \tau z)$ to $(w,z)$. By abusing notation, let us denote the map in the chain complex induced by taking the push forward of the Heegaard data $\H_{K}$ (and $\H_{p}$) by the above isotopy as $\rho$. As shown in Section~\ref{involutionaction}, the action $\tau_K$ on $CFK(K)$ is defined as, 
\[
\tau_K= \Phi(\rho \tau\H_K, \H_K) \circ \rho \circ t_K,
\]
and the action of $\tau$ on $CF(S^{3}_p(K),[0])$ is defined as,
\[
\tau= \Phi^{\prime}(\rho \tau\H_p, \H_p) \circ \rho \circ t.
\] 
For the purpose of the argument, we find it useful to use an alternative description of these actions. Note that there is a finger moving map $\bar{\rho}$ obtained by performing finger moves along the same path defining $\rho$, but in reverse, i.e. taking $(w,z)$ to $(\tau w, \tau z)$. We claim that $\tau_K$ can be alternatively represented as the following composition,
\[
CFK(\mathcal{H}_{K}) \xrightarrow{\Phi_1} CFK(\tau \bar{\rho} \mathcal{H}_{K}) \xrightarrow{t_K} CFK(\bar{\rho} \mathcal{H}_{K}) \xrightarrow{\rho} CFK(\mathcal{H}_{K}).
\]
Here, $\Phi_{1}: CFK(\H_K) \rightarrow CFK(\tau \bar{\rho} \H_K)$ is a chain homotopy equivalence induced by Heegaard moves relating $\H_K$ and $\tau \bar{\rho} \H_K$ as both diagrams represent the same doubly-based knot $(K,w,z)$. The commutation (up to chain homotopy) of the diagram below  shows that the alternative definition of $\tau_K$ is chain homotopic to the original one.
\[
\begin{tikzcd}
	{CFK(\mathcal{H}_{K})} && {CFK(\tau \mathcal{H}_{K})} && {CFK( \rho \tau \mathcal{H}_{K})} \\
	\\
	{CFK(\tau \bar{\rho} \mathcal{H}_{K})} && {CFK(\bar{\rho} \mathcal{H}_{K})} && {CFK(\mathcal{H}_{K})}
	\arrow["\rho", from=1-3, to=1-5]
	\arrow["{\Phi_{2}}"', from=1-3, to=3-3]
	\arrow["\Phi", from=1-5, to=3-5]
	\arrow["\rho"', from=3-3, to=3-5]
	\arrow["{t_K}", from=1-1, to=1-3]
	\arrow["{\Phi_{1}}"', from=1-1, to=3-1]
	\arrow["{t_K}"', from=3-1, to=3-3]
\end{tikzcd}\]

\noindent The map $\Phi_{2}: CFK(\tau \H_K) \rightarrow CFK(\bar{\rho} \H_K)$  represents the chain homotopy equivalence induced by conjugating the Heegaard moves via $t_K$. In a similar manner, let $\Phi_{3}: CFK(\rho \tau \H_K) \rightarrow CFK(\H_K)$ denote the chain homotopy equivalence induced by conjugating the Heegaard moves involved in $\Phi_{2}$ by $\rho$. Since both $\Phi_3$ and $\Phi$ are maps induced by Heegaard moves between two Heegaard diagrams, it follows from \cite[Proposition 6.1]{HM} that $\Phi_{3}$ is chain homotopic to $\Phi$. Hence, the square on the right-hand side also commutes up to chain homotopy. We also consider the similar description of the $\tau$ action on $\H_p$, as $\tau=  \rho  \circ t \circ \Phi^{\prime}_{1}$, where $\Phi^{\prime}_{1}$ is obtained by commuting the $\Phi^{\prime}$ past $t$ and $\rho$.

We now observe that in order to prove the Lemma, it suffices to show that all the three squares below commute up to chain homotopy.
\[
\begin{tikzcd}
	{CF^{+}(\mathcal{H}_p,[0])} && {CF^{+}(\tau \bar{\rho}\mathcal{H}_p,[0])} && {CF^{+}(\bar{\rho}\mathcal{H}_p,[0])} && {CF^{+}(\mathcal{H}_p,[0])} \\
	\\
	{A^{+}_{0}(\mathcal{H}_K)} && {A^{+}_{0}(\tau \bar{\rho}\mathcal{H}_K)} && {A^{+}_{0}(\bar{\rho}\mathcal{H}_K)} && {A^{+}_{0}(\mathcal{H}_K)}
	\arrow["{\Phi^{'}_1}", from=1-1, to=1-3]
	\arrow["{\Gamma^{+}_{(w,z)}}"', from=1-1, to=3-1]
	\arrow["{\Phi_1}"', from=3-1, to=3-3]
	\arrow["t", from=1-3, to=1-5]
	\arrow["{t_K}"', from=3-3, to=3-5]
	\arrow["{\Gamma^{+}_{(\tau w,\tau z)}}", from=1-5, to=3-5]
	\arrow["\rho", from=1-5, to=1-7]
	\arrow["\rho"', from=3-5, to=3-7]
	\arrow["{\Gamma^{+}_{(w,z)}}", from=1-7, to=3-7]
	\arrow["{\Gamma^{+}_{(w,z)}}", from=1-3, to=3-3]
\end{tikzcd}
\]
Where $\Gamma^{+}$-maps are the large surgery isomorphisms, \cite{OSknots} and the subscript denotes the pair of basepoints that is used to define the map. For example $\Gamma^{+}_{(\tau w, \tau z)}$ is defined as follows,
\[
\Gamma^{+}_{(\tau w, \tau z)}([\bold{x},i])= \sum\limits_{\bold{y} \in \mathbb{T}_{\alpha} \cap \mathbb{T}_{\beta}} {\sum\limits_{\substack{\psi \in \pi_{2}(\bold{x},\Theta, \bold{y}) \\ \mu (\psi)=0 ,  n_{\tau z}(\psi)=n_{\tau w}(\psi) }} } \# \mathcal{M}(\psi) \cdot [\bold{y}, i- n_{\tau w}(\psi), i- n_{\tau z}(\psi)].
\] 
Let us note that the squares on the middle and on the right commute tautologically. For example, in the case of the middle square, this is seen by the observation that there is an one-one correspondence between the pseudo-holomorphic triangles $\psi$ with boundary among $\alpha$, $\beta$ and $\gamma$ curves (with $n_{z}(\psi)=n_{w}(\psi)$) with pseudo-holomorphic triangles $\psi^{\prime}$ with boundary among $\tau \alpha$, $\tau \beta$ and $\tau \gamma$ (with $n_{z}(\psi ^{\prime})=n_{w}(\psi^{\prime})$), given by taking $\psi^{\prime}$ to be $\tau \psi$.  The count of the former is used in definition of $\Gamma^{+}_{(w,z)}$ while the count of the latter is used in defining $\Gamma^{+}_{(\tau w,\tau z)}$.


\noindent Now, we focus on the commutation of the left square. To show the commutation, we will need to construct a sequence of Heegaard moves taking the Heegaard triple $\mathcal{H}_{\Delta}$ to the Heegaard triple $\tau \bar{\rho} \mathcal{H}_{\Delta}:=(\tau \bar{\rho} \Sigma, \tau \bar{\rho} \alphas, \tau \bar{\rho} \gammas, \tau \bar{\rho} \betas,  w, z)$. This will in turn induce maps $\Phi^{'}_{1}$ and $\Phi_{1}$. Before going forward, we describe a useful interpretation of the map $\rho$. We note that due to our constriction of the Heegaard surface $\Sigma$, it can be seen that there is a knot longitude $\gamma_{0}$ lying in the surface $\Sigma$ which is fixed by the involution $\tau$. We refer to this equivariant knot longitude as $\gamma_{0}$. For example, in Figure~\ref{trefoil_periodic} the green curve represents $\gamma_{0}$. 
\begin{figure}[h!]
\centering
\includegraphics[scale=.60]{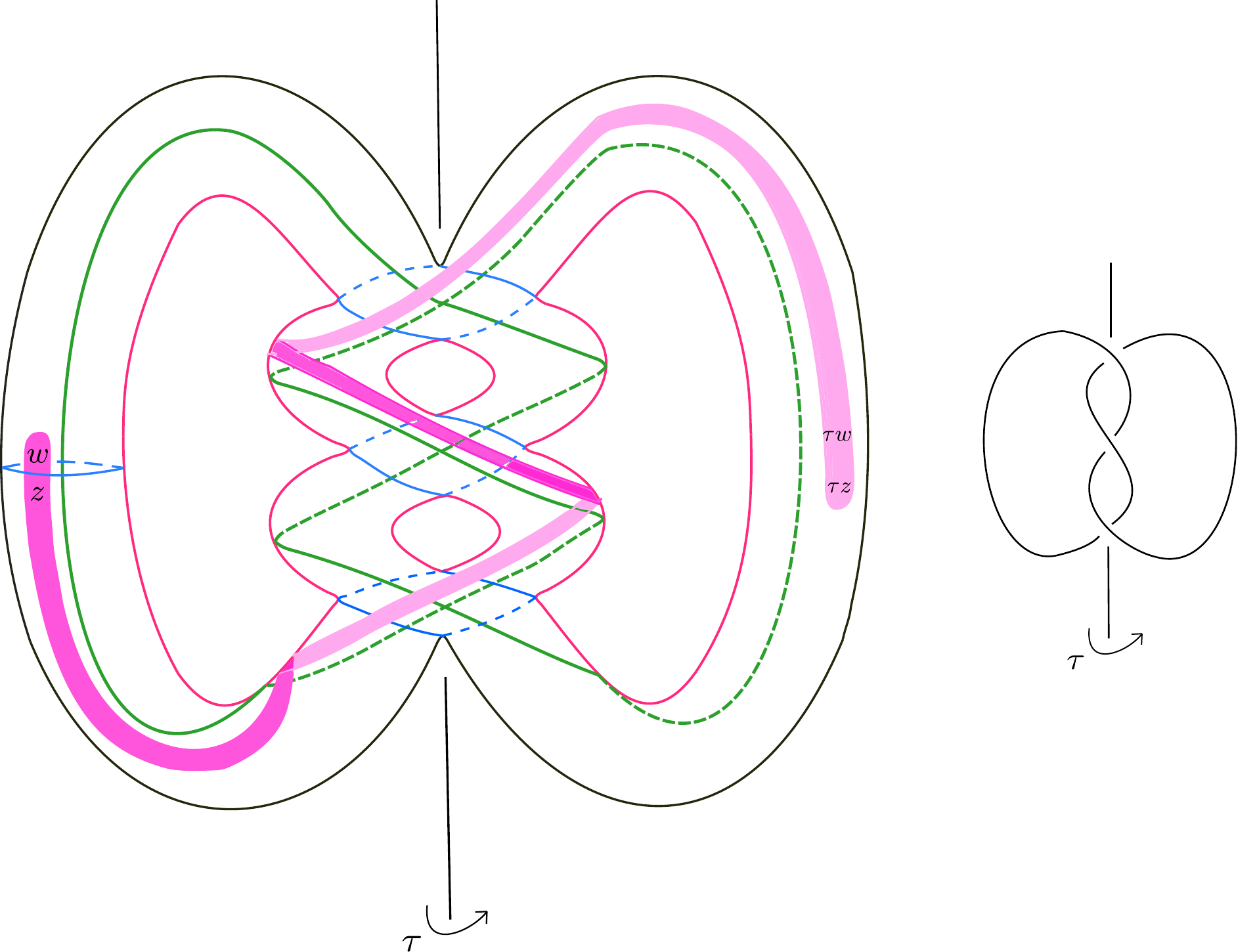} 
\caption
{Equivariant Heegaard diagram for the left-handed trefoil with the indicated periodic symmetry on the right. the $\alpha$ and $\beta$-curves are in red and blue respectively. The green curve is the equivariant longitude $\gamma_0$. Other $\gamma$-curves differ from the blue $\alpha$-curves by a small isotopy. The pink region depicts the finger-moving isotopy region running parallel to $\gamma_0$. }\label{trefoil_periodic}
\end{figure}

In particular, the surgery longitude $\gamma_{g}$ can be obtained by winding $\gamma_{0}$ around $\beta_{g}$ so that it represents the $p$-surgery curve. Let us join $w$ to $z$ in the complement of the $\alpha$-curves and $z$ to $w$ in the complement of the $\beta$-curve to obtain the knot equivariant $K$. Let us now take a small annular neighborhood of an arc of $K$ on $\Sigma$ which is disjoint from $\gamma_0$ but runs parallel to it, taking $(w,z)$ to $(\tau w, \tau z)$. We take the finger moving isotopy which is used to define $\rho$ to be supported inside the annular region and be identity outside it. In Figure~\ref{trefoil_periodic}, the pink region depicts the isotopy region on $\Sigma$. We now describe the Heegaard moves that will induce the maps $\Phi$ and $\Phi^{'}$.

\begin{enumerate}
 \item Firstly we apply Dehn twists in the winding region so that $\gamma_{g}$ becomes $\gamma_{0}$ at the cost of winding the curve $\alpha_{g}$ accordingly. The winding region after the twists is depicted in Figure~\ref{dehn}. By abusing notation, we will now refer to the $\gamma_0$-curve as $\gamma_{g}$ and will continue to refer to the new $\alpha_g$ curve (after the Dehn-twist) as $\alpha_g$.

\item \label{moves} Before going on to the next set of moves, let us first understand the effect of $\rho$ in the Heegaard diagram. Since the annular region supporting the isotopy runs parallel to $\gamma_g$, the isotopy does not change $\gamma_g$. Moreover since $\gamma_g$ is chosen to be invariant with respect to $t_K$, we have $t_K \bar{\rho} (\gamma_g)=\gamma_g$. Since $\H_K$ and $\tau \bar{\rho} \H_K$ both represent the doubly-based knot $(K,w,z)$, by \cite{OSknots} there is a sequence of Heegaard moves relating diagrams $\H_K$ and $\tau \bar{\rho} \H_K$. More explicitly, these moves are

\begin{itemize}
\item Handleslides and isotopies among the $\alpha$-curves,

\item Handleslides and isotopies among the $\beta$-curves $\{ \beta_1,\beta_2, \cdots \beta_{g-1} \}$,

\item Handleslides and isotopies of $\beta_{g}$ across one of $\beta$-curves $\beta_{i}$ for $i \in 2, \cdots g-1$,

\item Stabilizations and de-stabilizations of the Heegaard diagram.

\end{itemize}

\noindent Here all the isotopies and handleslides are supported in the complement of the basepoints $w$ and $z$.

\begin{figure}[h!]
\centering
\includegraphics[scale=.9]{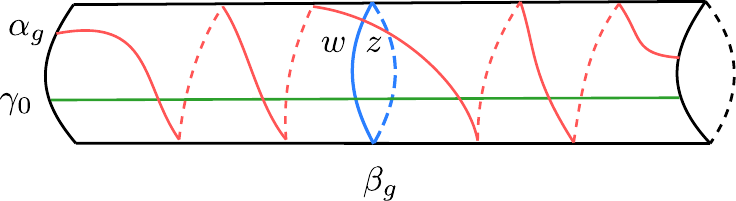} 

\caption
{After applying Dehn twist in the winding region, the $\alpha$-curve twists while the longitude $\gamma_g$ becomes the equivariant longitude $\gamma_0$.}\label{dehn}
\end{figure}

\item We would now like to argue that there is a sequence of doubly-based Heegaard moves relating the diagrams $\H_p$ and $\tau \bar{\rho} \H_p$. Note that the $\alpha$-moves applied in Step~\ref{moves} already take $\alphas$ to $\tau \bar{\rho} \alphas$. Since the curve $\gamma_{g}$ satisfies $\tau \bar{\rho} \gamma_{g}=\gamma_{g}$, we only need to define moves taking the $\gamma$-curves $\gamma_{1}, \gamma_2, \cdots \gamma_{g-1}$ to $\gamma$-curves $\tau \bar{\rho}  \gamma_{1}, \tau \bar{\rho} \gamma_2, \cdots \tau \bar{\rho} \gamma_{g-1}$. However, while defining the moves, we also need to make sure that we do not change $\gamma_g$. But these $\gamma_{i}$ differ from the corresponding $\beta_{i}$-curve (for $i \neq g$) by a small Hamiltonian isotopy. So on $\gamma_i$ we apply the isotopies and handleslides which are induced from the associated $\beta$-curves. Importantly, these $\beta$-moves do not involve $\beta_g$, hence $\gamma_g$ is unchanged during this process. 

 \end{enumerate} 
 
\noindent The moves described in the steps above induce the maps $\Phi_1$ and $\Phi^{'}_1$. We are now in a position to show that the following diagram commute up to chain homotopy.

 \[
 \begin{tikzcd}[row sep =large]
{CF}^{+}(\mathcal{H}_{p}) \arrow{r}{\Phi} \arrow{d}{\Gamma^{+}_{(w,z)}} & {CF}^{+}(\tau \bar{\rho} \mathcal{H}_{p}) \arrow{d}{\Gamma^{+}_{(w,z)}} \\
{A^{+}_{0}(\mathcal{H}_{K})} \arrow{r}{\Phi}  & {A^{+}_{0}(\tau \bar {\rho} \mathcal{H}_{K})} 
\end{tikzcd}
\]

\noindent Let $\mathcal{H}^{(i)}_{p}$ and $\mathcal{H}^{(i)}_{K}$  represent the Heegaard data of the surgered manifold and the knot respectively after applying in the $i$-th move, among the moves listed above. We then have maps 
\[
\Phi(\mathcal{H}^{(i)}_{p},\mathcal{H}^{(i+1)}_{p}): CF^{+}(\mathcal{H}^{(i)}_{p}) \longrightarrow CF^{+}(\mathcal{H}^{(i+1)}_{p}),
\]
and 
\[
\Phi(\mathcal{H}^{(i)}_{p},\mathcal{H}^{(i+1)}_{K}): CF^{+}(\mathcal{H}^{(i)}_{K}) \longrightarrow CF^{+}(\mathcal{H}^{(i+1)}_{K}).
\] 
Furthermore, at each step we have the triangle counting maps from the large surgery isomorphism that we described earlier
\[
\Gamma^{+}_{i}: CF^{+}(\mathcal{H}^{(i)}_{p}) \longrightarrow A^{+}_{0}(\mathcal{H}^{(i)}_{K}).
\]
We now claim that we can construct homotopies $R^{+}_{i}$ which make each of the diagrams below commute.
\[
 \begin{tikzcd}[column sep =huge, row sep =large]
{CF}^{+}(\mathcal{H}^{(i)}_{p}) \arrow{r}{\Phi(\mathcal{H}^{(i)}_{p},\mathcal{H}^{(i+1)}_{p})} \arrow[dr, dashrightarrow, "R^{+}_{i}"] \arrow{d}{\Gamma_{i}^{+}} & {CF}^{+}(\mathcal{H}^{(i+1)}_{p}) \arrow{d}{\Gamma_{i+1}^{+}} \\
{A^{+}_{0}}(\mathcal{H}^{(i)}_{K}) \arrow{r}{\Phi(\mathcal{H}^{(i)}_{K},\mathcal{H}^{(i+1)}_{K})} & {A^{+}_{0}}(\mathcal{H}^{(i+1)}_{K}) 
\end{tikzcd}
\]
The argument is similar to that in the proof of large surgery formula in involutive Floer homology \cite[Theorem 1.5]{HM}. Although the Heegaard moves defining the maps $\Phi(\mathcal{H}^{(i)}_{p},\mathcal{H}^{(i+1)}_{p})$ and $\Phi(\mathcal{H}^{(i)}_{p},\mathcal{H}^{(i+1)}_{K})$ are different, the underlying argument is still the same (in fact, the step where Hendricks and Manolescu had to define a chain homotopy where underlying Heegaard triple changed from right-subordinate to left-subordinate is absent here, since all the triples are right-subordinate).

Recall that the vertical maps are obtained by counting appropriate pseudo-holomorphic triangles $\psi$ in the homotopy class $n_{z}(\psi)=n_{w}(\psi)$, we can construct a chain homotopy $R^{+}_{i}$ for the $i$-th step by counting pseudo-holomorphic quadrilaterals $\Box$ again in homotopy class $n_{z}(\Box)=n_{w}(\Box)$. This follows from the proof of the independence of the $2$-handle maps from the underlying Heegaard triple subordinate to a bouquet \cite[Theorem 4.4]{OSsmooth4}, see also \cite[Theorem 6.9]{Juhasz}. The proof of independence, in turn, follows from the associativity of the holomorphic triangle maps \cite[Theorem 2.5]{OSsmooth4}, where the count of holomorphic quadrilaterals yields the homotopy. In our case, the exact same argument holds, except we need to restrict to the homotopy class of the quadrilaterals with $n_{z}(\Box)=n_{w}(\Box)$. This finishes the proof.

\end{proof}

\subsection{Large surgeries on strongly invertible knots}

We now move towards proving a similar formula for the strongly invertible knots. The underlying idea is essentially similar to that in the proof of large surgery formula for periodic knots. 

\begin{lemma}\label{strongproof}
Let $K \subset S^{3}$ be a strongly invertible knot with an involution $\tau$. Then for all $p\geq g(K)$ as a relatively graded $\mathbb{Z}_{2}[Q,U]/(Q^{2})$-modules, we the following isomorphism
\begin{center}
$HFI^{+}_{\tau}(S^{3}_{p}(K),[0]) \cong H_{*}(AI_{0}^{+,\tau})$.
\end{center}
\end{lemma}

\begin{proof}
As before, we represent $\tau$ by $180$-degree rotation around the z-axis. For simplicity, we only consider the case where basepoints $(w,z)$ on the knot $K$ are switched by $\tau$. Following, Lemma~\ref{periodic} we construct an Heegaard triple $\mathcal{H}_{\Delta} :=(\Sigma, \alphas, \gammas, \betas, w, z)$, where the Heegaard surface $\Sigma$ is fixed by $\tau$ setwise. However, we will now modify the construction slightly, so that it is compatible with the present argument. Recall from the previous construction that, we had placed the meridian $\beta_g$ on the tube corresponding to one of the edges of $c$ (i.e. the boundary of the unbounded region.). Now since $\tau$ is a strong involution on $K$, the axis intersects knot $K$ in two points. If one of those points lies in the arc of the knot corresponding to $c$, we place the meridian $\beta_g$ along the axis so that the winding region is fixed by $\tau$ setwise. We also base the basepoints $w$ and $z$ so that they are switched by $\tau$. Figure~\ref{strongwinding} depicts the equivariant winding region. 

\begin{figure}[h!]
\centering
\includegraphics[scale=.9]{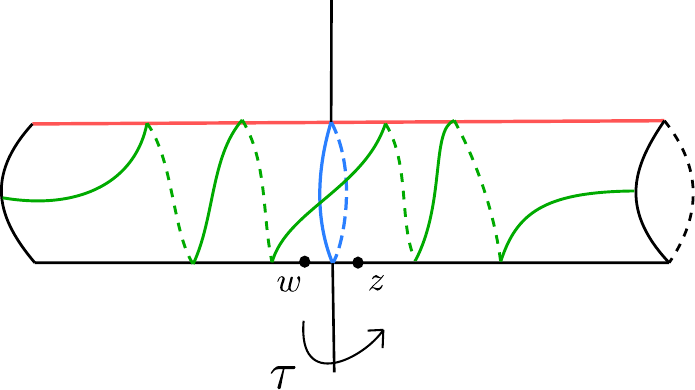} 

\caption
{Placement of the basepoints in the winding region so that they are switched.}\label{strongwinding}
\end{figure}
\noindent
However, it is possible that none of the two fixed points on $K$ lie in $c$, see Figure~\ref{badgraph} for an example. Then let $e$ be an internal edge so that the arc in the knot $K$ corresponding to $e$ has a fixed point, as in Figure~\ref{badgraph}. We now delete one of the $\alpha$ curves lying the tube corresponding to $e$ and add a $\alpha$ curve in the tube corresponding to $c$. Finally, we place the meridian $\beta_g$ along the axis in tube corresponding to $e$. This step is illustrated in Figure~\ref{switchbeta}.

 \begin{figure}[h!]
\centering
\includegraphics[scale=.6]{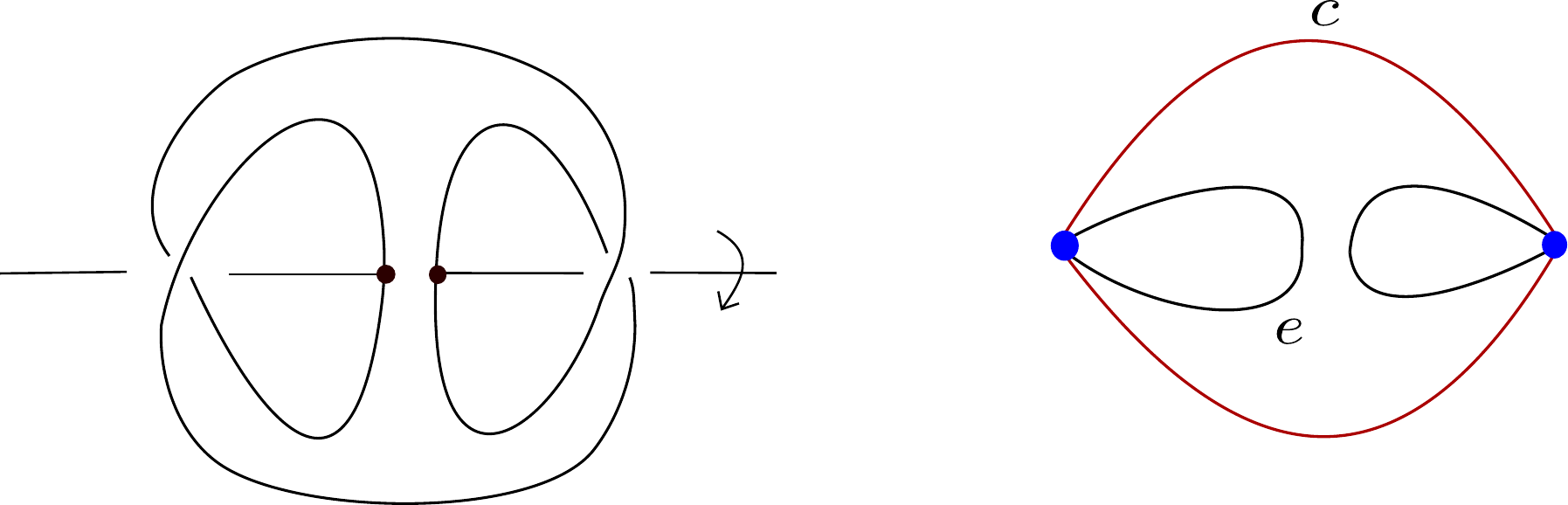} 
\caption
{An illustration of a graph, where the fixed points do not intersect the edge $c$.}\label{badgraph}
\end{figure}

\noindent
Note that our choice of placing the basepoints dictates that in order to define an automorphism on the surgered $3$-manifold $\tau: CF(\H_{p}) \rightarrow CF(\H_p)$, we need to choose a finger-moving isotopy taking the basepoint $\tau w =z$ to $w$ along some path. However, as seen in Proposition~\ref{localbase} the definition of $\tau$ is independent of this choice of isotopy up to local equivalence. Here we will choose a small arc connecting $w$ and $z$ basepoints in the winding region, i.e. the arc $\rho$ which intersects the curve $\beta_g$ once.

 \begin{figure}[h!]
\centering
\includegraphics[scale=.7]{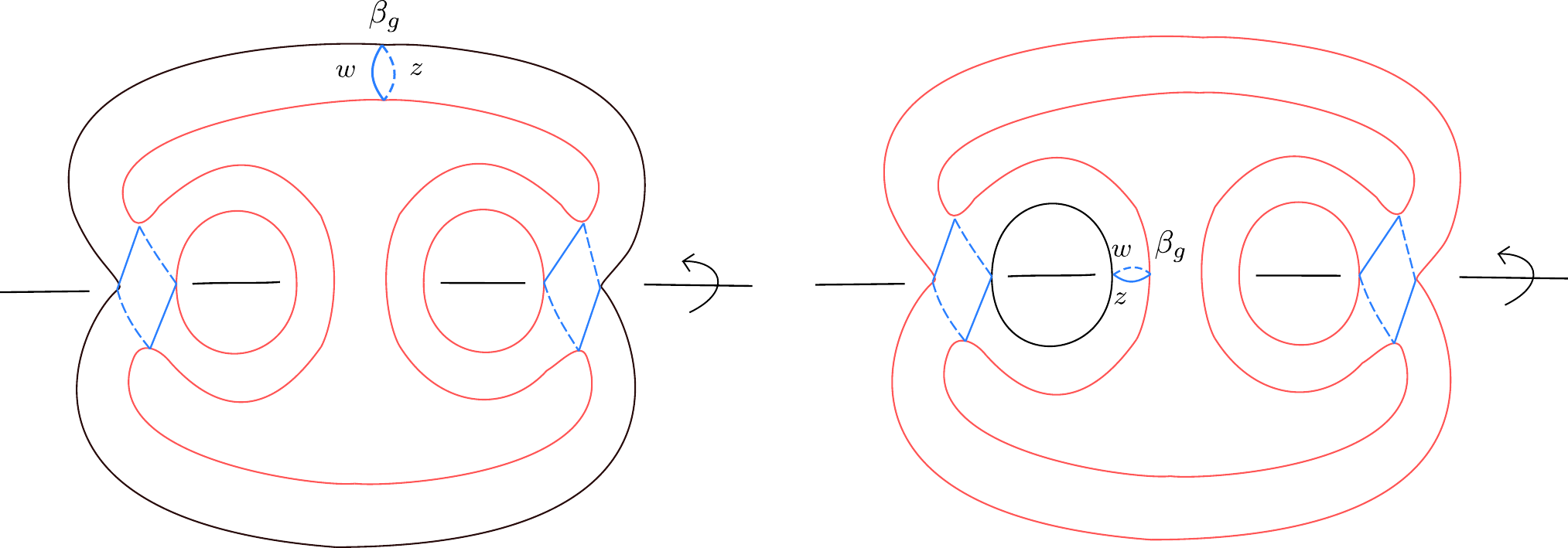} 
\caption
{Adjusting the meridian so that the basepoints are switched.}\label{switchbeta}
\end{figure}

We now wish to show the three squares in the diagram below commute up to chain homotopy. The composition of the maps in the top row represent $\tau: CF(\H_{p}) \rightarrow CF(\H_p)$, while the bottom row represents $\tau_K: CFK(\H_K) \rightarrow CFK(\H_K)$. For the maps $\Gamma^{+}$ we follow notation from the proof of Lemma~\ref{periodic}. The basepoint in the superscript of $\H_p$ indicates which basepoint we use for the surgered manifold $CF(S^{3}_{p}(K),[0])$ while defining maps on it.
\[
\begin{tikzcd}[column sep =large,row sep=large ]
CF^{+}(\mathcal{H}^{w}_{p},[0]) \arrow{r}{t} \arrow{d}{\Gamma^{+}_{(w,z)}} &
  CF^{+}(\tau \mathcal{H}^{z}_{p},[0])    \arrow{r}{\Phi^{'}} \arrow{d}{\Gamma^{+}_{(z,w)}} &
 CF^{+}( \mathcal{H}^{z}_{p},[0]) \arrow{r}{\rho} \arrow{d}{\Gamma^{+}_{(z,w)}} &
  CF^{+}(\mathcal{H}^{w}_{p},[0]) \arrow{d}{\Gamma^{+}_{(w,z)}}
\\
A^{+}_{0}(\mathcal{H}_{K}) \arrow{r}{t_K} &
 A^{+}_{0}(\tau \mathcal{H}_{K}) \arrow{r}{\Phi} &
  A^{+}_{0}(\mathcal{H}_{K^{r}}) \arrow{r}{sw} & A^{+}_{0}(\mathcal{H}_{K})
\end{tikzcd}
\]

\begin{figure}[h!]
\centering
\includegraphics[scale=.65]{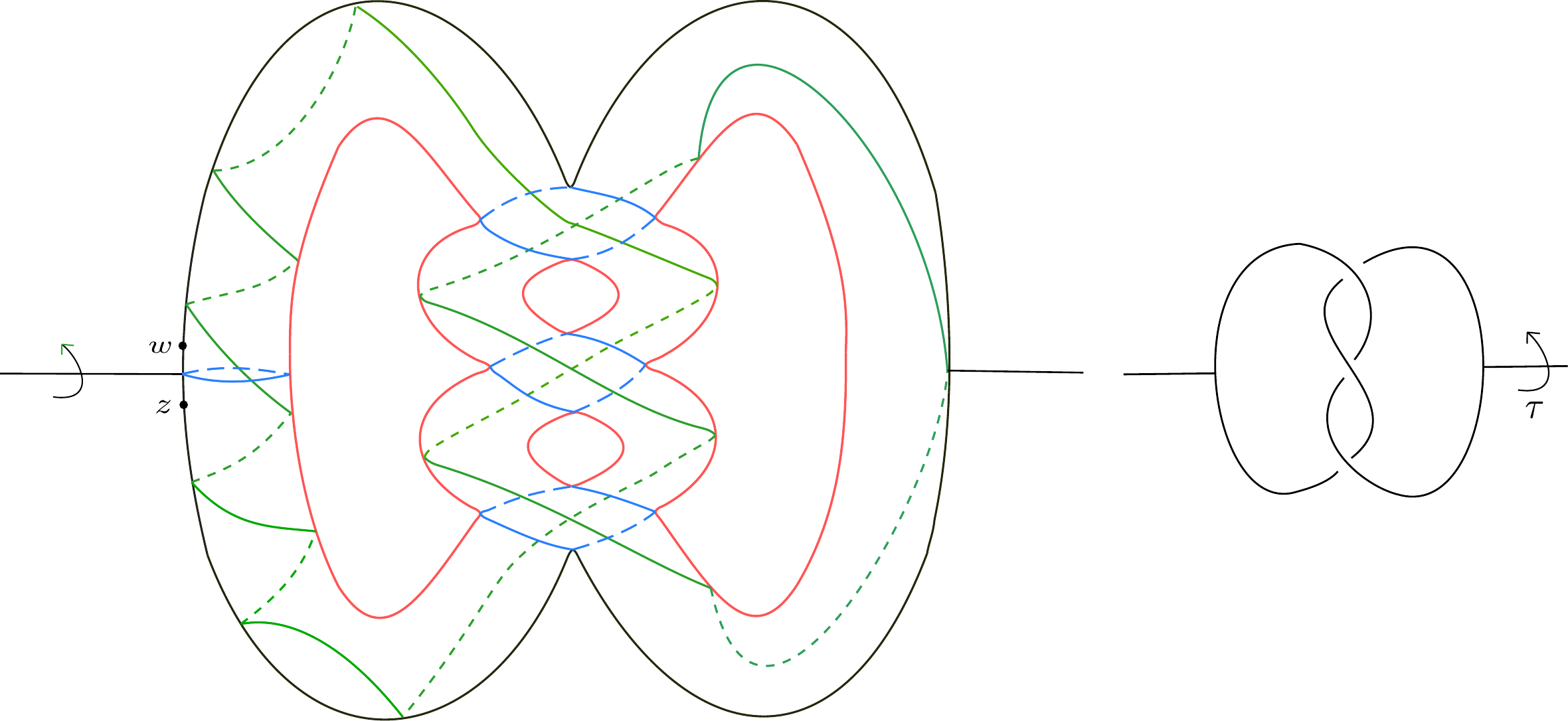} 

\caption
{Equivariant Heegaard diagram for the left-handed trefoil with an equivariant longitude $\gamma_g$ in green.}\label{trefoil_strong_surgery}
\end{figure}

The argument for commutation of the left and the middle squares are verbatim to the latter half of the proof of large surgery formula for periodic knots, so we will be terse. Firstly, we notice that the square on the left commutes tautologically. For the commutation of the middle square, we use a similar strategy as in the periodic case. We notice that we can place the $\gamma_g$ on the Heegaard diagram so that $\tau \gamma_g$ can be isotoped to $\gamma_g$ in the complement of the basepoints $w$ and $z$, see Figure~\ref{isotopy} (and also Figure~\ref{trefoil_strong_surgery}).
\begin{figure}[h!]
\centering
\includegraphics[scale=.9]{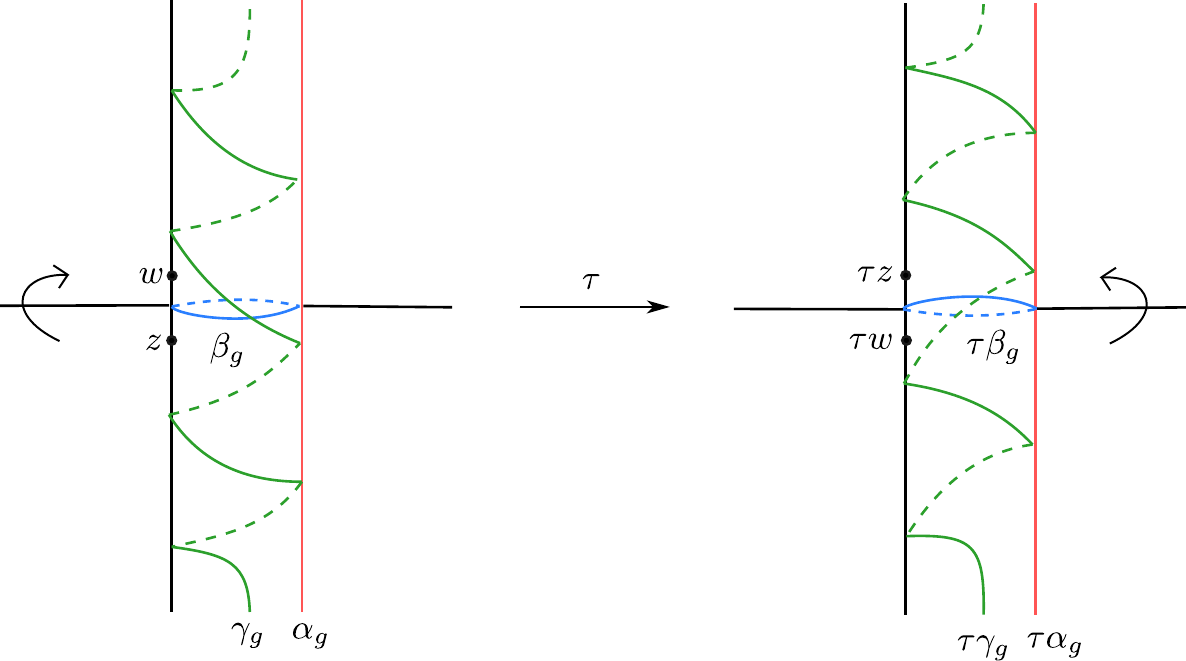} 
\caption
{The effect of taking the push-forward of the winding region by the involution $\tau$. Note that $\gamma_g$ is chosen in a way such that $\tau \gamma_g$ can be isotoped to $\gamma_g$ in the complement of the basepoints $w$ and $z$. }\label{isotopy}
\end{figure}
There are Heegaard moves in the complement of basepoints taking $\tau \H_K$ to $\H_K^{r}$ since they both represent the knot $(K^{r},z,w)$. This defines the map $\Phi$. Now towards defining $\Phi^{'}$, we already have the $\alpha$-moves taking $\tau \alphas$ to $\alphas$. The $\gamma$-moves taking $\tau \gammas$ to $\gammas$ while being supported in the complement of the basepoint is defined analogously as in the periodic case. Firstly we apply the isotopy taking $\tau \gamma_g$ to $\gamma_g$ then we apply replicate the moves required take $\tau \beta_i$ to $\beta_i$ (for $i \neq g$) to take $\tau \gamma_i$ to $\gamma_i$. The commutation of the middle square then follows from a similar argument as in the periodic case.

The square on the right is somewhat different from the ones we have encountered so far. Since for this square, we are applying maps $\rho$ and $sw$ on $CF^{+}$ and $A^{+}_0$ respectively, which are \textit{a priori} different in nature. While the map $\rho$ is induced by pushforward defined by the finger moving isotopy along the path $\rho$, the map $sw$ is not; rather it has no such geometric interpretation. Nevertheless, we claim that the diagram commutes up to chain homotopy. Note that the path $\rho$ does not intersect either the $\alpha$ or the $\gamma$ curves. It follows that the map $\rho$ homotopic to a map that is tautological on the intersection points, i.e. $\rho[\bold{x},i]=[\bold{x},i]$. To see this, note that $\rho$ is explicitly defined as
\[
\rho:= \Phi_{\rho_{*}(J) \rightarrow J}(\H^{w}_p, \H^{w}_{p}) \circ \ t_{\rho},
\]
where $t_{\rho}$ represents the push forward of the Heegaard data $\H^{z}_p$ under the image of a finger-moving isotopy. However, this changes the almost complex structure $J$. We then post compose $t_{\rho}$ by the continuation map $\Phi_{\rho_{*}(J) \rightarrow J}$ taking the almost complex structure $t_{{\rho}_{*}}(J)$ back to $J$. In the proof of \cite[Theorem 14.11]{Zemkegraph} by Zemke, it was shown that both the maps involved in the definition are tautological. The commutation of the right square then follows tautologically from the definition of the maps involved. This completes the proof.

\end{proof}

\begin{proof}[Proof of Theorem~\ref{t1}]
We showed this for a certain choice of Heegaard diagram in Lemma~\ref{periodicproof} and Lemma~\ref{strongproof}. The proof is finished after observing that actions $\tau$ and $\tau_K$ are independent of the choice of the Heegaard diagram up to chain homotopy.
\end{proof}
Let us now also discuss the proof of Corollary~\ref{corollary}. The $\iota \circ \tau$ action on $CF^{+}$ was studied in \cite{dai2020corks}. It was shown that the $\iota \circ \tau$ action is a well-defined involution on $CF^{+}$ \cite[Lemma 4.4]{dai2020corks} (this was shown for integer homology spheres, but the proof carries through to rational homology spheres while specifying that $\tau$ fixes the $[0]$-$\spinc$ structure). In a similar vein, one can study the $\iota_K \circ \tau_K$ action on knots in integer homology spheres, which again induce an action $\iota_{K} \circ \tau_{K}$ on the $A^{+}_{0}$-complex. We show that we can identify the action of  $\iota \circ \tau$ for large equivariant surgeries on symmetric knots with the action of $\iota_{K} \circ \tau_{K}$ on the $A^{+}_{0}$-complex.

\begin{proof}[Proof of Corollary~\ref{corollary}]
In the proof of the involutive large surgery formula, Hendricks-Manolescu  \cite[Theorem 1.5]{HM}, showed,
\[
\iota_{K} \circ \Gamma^{+}_{(p,0)} \simeq \Gamma^{+}_{(p,0)} \circ \iota.
\]
In the proof of Theorem~\ref{t1}, we showed,
\[
\tau_{K} \circ \Gamma^{+}_{(p,0)} \simeq \Gamma^{+}_{(p,0)} \circ \tau.
\]
Notably, the map $\Gamma^{+}_{(p,0)}$ is the same for both surgery formulas. Hence the conclusion follows.
\end{proof}



\section{Computations}\label{computesection}
As discussed in the introduction, computation of the action of symmetry on the knot Floer complex has been limited. In this section, we compute the action of $\tau_K$ on the knot Floer complex for a class of strongly-invertible knots and a class of thin knots. The tools we use for these computations are functoriality of link Floer cobordisms maps \cite{Zemkelinkcobord} and the grading and filtration restrictions stemming from the various properties of the knot action detailed in Section~\ref{action_knot}. These computations in turn with the equivariant surgery formula, yields computations of the induced action $\tau$ on the $3$-manifolds obtained by large surgery on $K$.

\subsection{Strong involution on $K \# K^{r}$}\label{KK}

We now compute the action for a certain class of strongly invertible knots. Let $(K,w,z) \subset \mathbb{R}^{3} \bigcup \{ \infty \}$ be a doubly-based oriented knot which is placed in the first quadrant of $\mathbb{R}^{3}$. Let us place another copy of $K$ in $\mathbb{R}^3$, obtained by rotating the original $K$, by $180$-degrees along the $z$-axis. We will represent the rotation as $\tau$ and the image of $K$ as $\tau K$. We then connected sum $(K,w,z)$ and $(\tau {K}^{r}, \tau w, \tau z)$ by a trivial band that intersect the $z$-axis in two points, see Figure~\ref{band}.
\begin{figure}[h!]
\centering
\includegraphics[scale=.6]{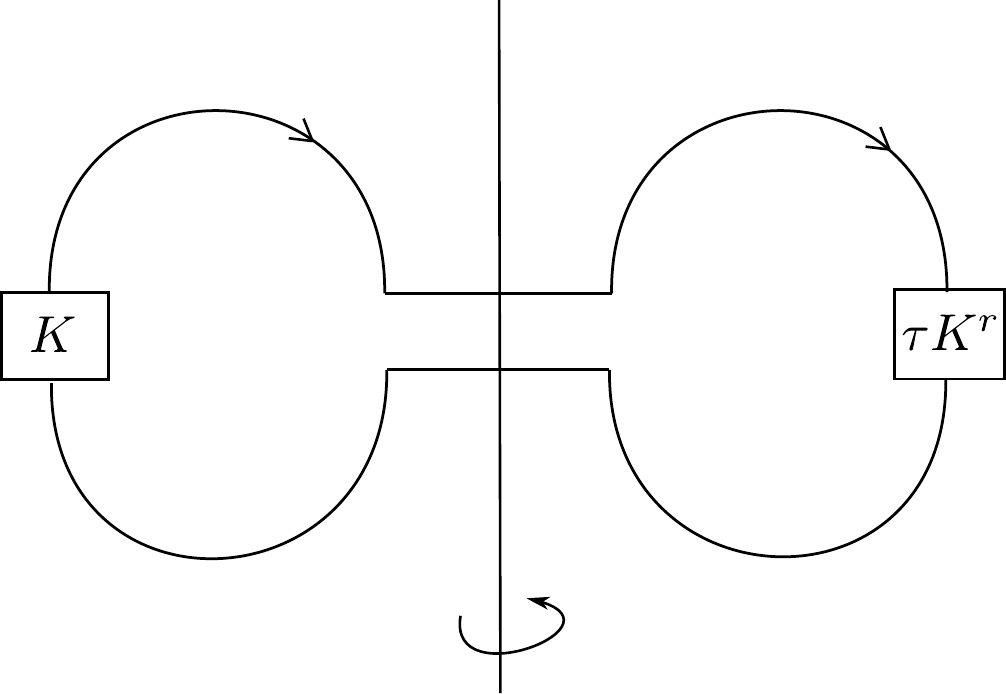} 
\caption
{The knot $K \# \tau K^{r}$ with the strong involution $\tilde{\tau}_{sw}$.}\label{band}
\end{figure}
There is an obvious strong involution induced by $\tau$ on the connected sum knot $K\# \tau {K}^{r}$, as in Figure~\ref{band}. Moreover, we place two basepoints $(w^{\prime},z^{\prime})$ on $K\# \tau {K}^{r}$ so that they are switched by $\tau$. Let us denote the action of $\tau$ on $CFK^{\infty}(K\# \tau {K}^{r},w^{\prime}, z^{\prime})$ as $\tilde{\tau}_{sw}$.

\noindent We begin by defining a map $\tilde{\tau}_{exch}$, which may be interpreted as being induced from \textit{exchanging} the two knots $(K,w,z)$ and $(\tau K^{r}, \tau w, \tau z)$ via $\tau$:
\[
\tilde{\tau}_{exch} : CFK^{\infty}(K,w,z) \otimes CFK^{\infty}(\tau{K}^{r}, \tau w, \tau z) \longrightarrow   CFK^{\infty}(K,w,z) \otimes CFK^{\infty}(\tau{K}^{r}, \tau w, \tau z).
\]
$\tilde{\tau}_{exch}$ defined as a composition of several maps. We start by taking the half-Dehn twist along the orientation of the knots $K$ and $\tau K^{r}$ mapping the $w$-type basepoints to the $z$-type basepoints and vice versa.
\[
\rho_1 \otimes \rho_2: CFK^{\infty}(K,w,z) \otimes CFK^{\infty}( \tau {K}^{r},\tau w, \tau z) \longrightarrow CFK^{\infty}(K,z,w) \otimes CFK^{\infty}( \tau {K}^{r},\tau z, \tau w).
\]
We then apply the push forward map associated with the diffeomorphism $\tau$, which switches the two factors:
\[
t: CFK^{\infty}(K,z,w) \otimes CFK^{\infty}( \tau{K}^{r},\tau z, \tau w) \longrightarrow   CFK^{\infty}(K^{r},z,w) \otimes CFK^{\infty}(\tau K, \tau z, \tau w).
\]
Finally, we apply the basepoint switching map from Subsection~\ref{action_knot}  to each component:
\[
sw \otimes sw: CFK^{\infty}(K^{r},z,w) \otimes CFK^{\infty}(\tau K,\tau z,\tau w) \rightarrow CFK^{\infty}(K,w,z) \otimes CFK^{\infty}( t {K}^{r},\tau w, \tau z).
\]
$\tilde{\tau}_{exch}$ is then to be the composition
\[
\tilde{\tau}_{exch} :=  sw \otimes sw \circ t \circ \rho_1 \otimes \rho_2.
\]
Note that, the order in which the maps appear in the above composition in the definition of $\tilde{\tau}_{exch}$ can be interchanged up to homotopy, although we will not need any such description.

Another useful piece of data for us will be the definition of the action of a strongly invertible symmetry on a knot when the basepoints are fixed by the involution. Recall that in Subsection~\ref{action_knot}, we defined the action of a strongly invertible symmetry on the knot Floer complex, under the assumption that the symmetry switches the basepoints. That definition is easily modified to define the action when the basepoints are fixed.

\begin{definition}\label{taufix}
 Let $(K,w,z)$ be a doubly-based knot with a strong involution $\tau_K$ that fixes the basepoints. The action of $\tau_K$ is defined as the composition:
\[
CFK^{\infty}(K,w,z) \xrightarrow{\rho} CFK^{\infty}(K,z,w) \xrightarrow{t_K} CFK^{\infty}(K^{r},z,w) \xrightarrow{sw} CFK^{\infty}(K,w,z).
\]
Here, as before $\rho$ represents the map induced by a half Dehn twist along the orientation of $K$ switching the $w$ and the $z$ basepoints. We will denote the action when the basepoint is fixed as $\tau^{\prime}_{K}$ (in order to distinguish from the previous definition, where the basepoints were switched).
\end{definition}
It can be checked that $\tau^{\prime}_{K}$ also satisfies the properties of a strongly invertible action as detailed in Proposition~\ref{strong}. Moreover, the earlier defined $\tau_{K}$ action is related to $\tau^{\prime}_{K}$ in the following way.
\begin{proposition}\label{taufixproof}
Let $K \subset S^{3}$ be any strongly invertible knot. We place $2$ pairs of basepoints $(w^{\prime}, z^{\prime})$ and $(w,z)$, on $K$ in Figure~\ref{rhoarc}, so that $(w,z)$ is switched under the involution and $(w^{\prime}, z^{\prime})$ is fixed. Let $\tau_K$ and $\tau^{\prime}_{K}$ denote the action of the symmetry on $CFK^{\infty}(K , w, z)$ and $CFK^{\infty}(K, w^{\prime}, z^{\prime})$ respectively. Then there exist a chain homotopy equivalence $\tilde{\rho}$ which intertwines with $\tau_{K}$ and $\tau^{\prime}_{K}$ up to homotopy:
\[
\tilde{\rho} \circ {\tau}^{\prime}_{K}  \simeq  {\tau}_{K} \circ \tilde{\rho}.
\] 
\end{proposition}

\begin{proof}
Let us define the basepoint moving isotopy maps $\rho_1$, $\rho_2$ and $\rho_3$ along the oriented arcs shown in Figure~\ref{rhoarc}. We push the basepoints along those arc following its orientation.  Note that the arc used to define $\rho_3$ is the image of $\rho_2$ under the involution.
\begin{figure}[h!]
\centering
\includegraphics[scale=1.2]{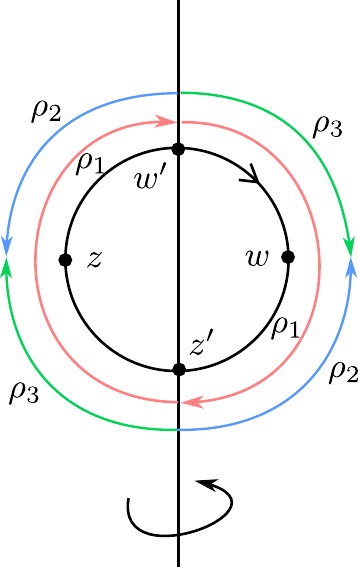} 
\caption
{The oriented arcs of $K$ used to define the finger moving iotopies. The arcs are shown as push offs. $\rho_1$, $\rho_2$ and $\rho_3$ are in red, blue and green respectively.}\label{rhoarc}
\end{figure}
The claim follows from the chain of commuting diagrams depicted below.
\[
\begin{tikzcd}[column sep =large,row sep=large ]
CFK(K,w^{\prime},z^{\prime}) \arrow{r}{\rho_1} \arrow{d}{\rho_3} &
  {CFK}(K,z^{\prime},w^{\prime})    \arrow{r}{t_K} \arrow{d}{\rho_2} &
 {CFK}(K^{r},z^{\prime},w^{\prime}) \arrow{r}{sw} \arrow{d}{{\rho}_3} &
  {CFK}(K^{r},w^{\prime},z^{\prime}) \arrow{d}{\rho_3}
\\
{CFK}(K,w,z) \arrow{r}{\mathrm{id}} &
 {CFK}(K,w,z) \arrow{r}{t_K} &
  {CFK}(K^{r},z,w) \arrow{r}{sw} & {CFK}(K,w,z)
\end{tikzcd}
\]
\noindent
The commutation for the left square follows from Figure~\ref{rhoarc}. The commutation for the middle and square and the right square follows tautologically. Hence we may take $\tilde{\rho}=\rho_3$.
\end{proof}

In \cite[Theorem 7.1]{OSknots} Ozsv{\'a}th and Szab{\'o} proved a connected sum formula for knot. In \cite[Proposition 5.1]{Zemkeconnected} Zemke re-proved the connected sum formula by considering \textit{decorated link cobordisms}. The latter formalism will be useful for us in the present computation. Although we have worked with the \textit{standard} knot Floer complex $CFK$ in this article, all the constructions readily generalize to the $\mathcal{CFK}$ version used by Zemke \cite{Zemkelinkcobord}. Both $CFK$ and $\mathcal{CFK}$ complexes essentially contain the same information. Instead of introducing $\mathcal{CFK}$, $\mathcal{CFL}$ and the link cobordism maps, we refer readers to \cite{Zemkelinkcobord} for an overview.

We are now in place to prove Theorem~\ref{KconnectKstrong}.

\begin{proof}[Proof of  Theorem~\ref{KconnectKstrong}]
We start with $K \sqcup \tau K^{r} \subset  S^{3} \sqcup S^{3}$. We then attach a $(1,1)$-handle, a copy of $(D^1 \times D^3, D^{1} \times  D^{1})$ to $(S^{3} \sqcup S^{3}, K \sqcup \tau K^{r}) \times I$ so that outgoing boundary of the trace of the handle attachment is $(S^{3}, K \# \tau K^{r})$. We refer to this cobordism from $S^{3} \sqcup S^{3}$ to $S^{3}$ as $W$ and the pair-of-pants surface embedded inside, obtained from the $1$-dimensional $1$-handle attachment as $\Sigma$. Moreover, we place the basepoints $(w,z)$ and $(\tau w, \tau z)$ in $K$ and $\tau K^{r}$ respectively. Then we decorate $\Sigma$ by $3$-arcs and color the region bounded by them as in Figure~\ref{decoration}. We also place two basepoints $w^{\prime}$ and $z^{\prime}$ on the outgoing end.  We denote this decorated cobordism as $\mathcal{F}$. Associated to $(W, \mathcal{F})$ there is a link cobordism map $G$ defined in \cite[Proposition 5.1]{Zemkeconnected}:
\[
G_{W,\mathcal{F}}: \mathcal{CFL}^{\infty}(K) \otimes \mathcal{CFL}^{\infty}(K^{r}) \rightarrow \mathcal{CFL}(K \# \tau {K}^{r})
\]

\begin{figure}[h!]
\centering
\includegraphics[scale=1]{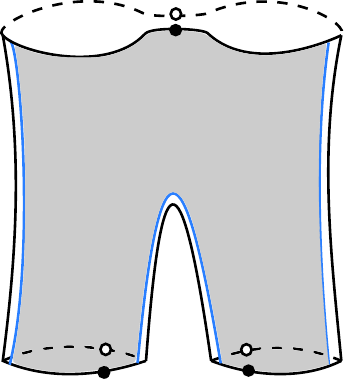} 

\caption
{The pair of pants cobordism with an equivariant decoration. The blue arcs represent the decorations. Black and white dots represent the $w$ and $z$-type basepoints respectively.}\label{decoration}
\end{figure}
\noindent
Now note that there is an obvious extension $f$ of $\tau$ on $(W,\Sigma)$. Schematically, this is given by reflecting along the vertical plane in the middle of the surface $\Sigma$. Note that although $f$ restricts to $\tilde{\tau}_{sw}$ on the outgoing end, it does not switch the basepoints $w^{\prime}$ and $z^{\prime}$ but it fixes them. In fact, for the choice of our decoration $\mathcal{F}$ on $\Sigma$, it is impossible to place any pair of basepoint on $K \# \tau {K}^{r}$ so that they are switched by $\tilde{\tau}_{sw}$. Hence for now, we will consider the action $\tilde{\tau}^{\prime}_{sw}$ on  $(K \# \tau {K}^{r}, w^{\prime}, z^{\prime})$ as defined earlier in Definition~\ref{taufix}.
We now claim that 
\begin{gather}\label{eq1}
G_{W,\overline{\mathcal{F}}} \circ \rho_1 \otimes \rho_2 + \rho_{\#} \circ G_{W,\mathcal{F}} \simeq G_{W, \overline{\mathcal{F}}} \circ (\Psi \rho_1 \otimes \Phi \rho_2).
\end{gather}
Here $\rho_1$ and $\rho_2$ and $\rho_{\#}$ denote the half-Dehn twist map along the orientation of the knot $K$, $\tau K^{r}$ and $K \# \tau K^{r}$ respectively. $\overline{\mathcal{F}}$ denotes the decoration conjugate to $\mathcal{F}$, i.e. obtained from switching the $w$ and $z$ regions of $\mathcal{F}$ and switching the basepoints $w$ and $z$-type basepoints. The proof of Relation~\ref{eq1} follows from the proof of \cite[Proposition 5.1]{Zemkeconnected} by applying the bypass relation to the disk shown in Figure~\ref{bypass}.
\begin{figure}[h!]
\centering
\includegraphics[scale=.7]{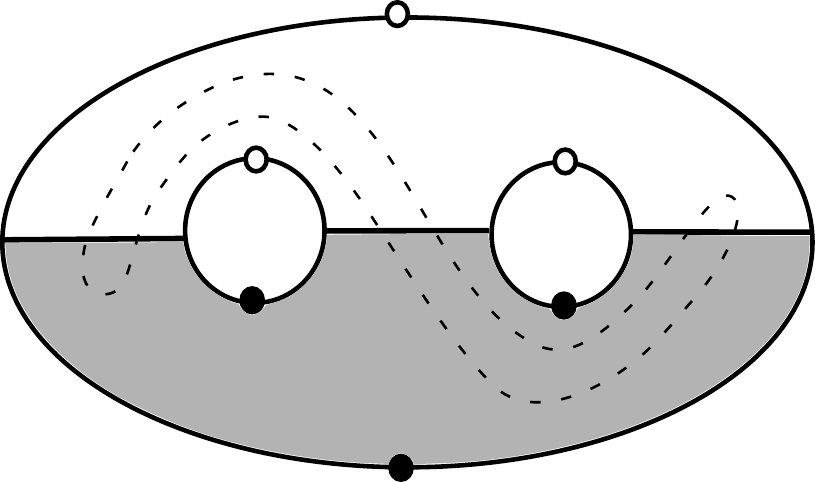} 
\caption
{Bypass disk used for Relation~\ref{eq1}.}\label{bypass}
\end{figure}
By post-composing Relation~\ref{eq1} with $sw \circ t$, we get
\begin{gather}\label{eq2}
sw \circ  t \circ G_{W,\overline{\mathcal{F}}} \circ \rho_1 \otimes \rho_2 +sw \circ  t \circ  \rho_{\#} \circ G_{W,\mathcal{F}} \simeq sw \circ  t \circ  G_{W, \overline{\mathcal{F}}} \circ (\Psi \rho_1 \otimes \Phi \rho_2)
\end{gather}
Let us now look at the following diagram which we claim commutes up to chain homotopy
\[\begin{tikzcd}
	{\mathcal{CFL}(K,z,w) \otimes \mathcal{CFL}( \tau K^{r},\tau z,\tau w)} && {\mathcal{CFL}(K \# \tau {K}^{r},z^{\prime},w^{\prime})} \\
	\\
	{ \mathcal{CFL}({K}^{r},z,w) \otimes \mathcal{CFL}(\tau K, \tau z, \tau w)} && {\mathcal{CFL}(K^{r} \# \tau {K},z^{\prime},w^{\prime} )} \\
	\\
	{\mathcal{CFL}(K, w,z) \otimes \mathcal{CFL}( \tau {K}^{r},\tau w, \tau z)} && {\mathcal{CFL}(K \# \tau {K}^{r}, w^{\prime}, z^{\prime})}
	\arrow["{G_{W,\overline{\mathcal{F}}}}", from=1-1, to=1-3]
	\arrow["t"', from=1-1, to=3-1]
	\arrow["{t}", from=1-3, to=3-3]
	\arrow["{sw \otimes sw}"', from=3-1, to=5-1]
	\arrow["sw", from=3-3, to=5-3]
	\arrow["G_{W,\mathcal{F}}"', from=5-1, to=5-3]
	\arrow["{G_{W,f(\overline{\mathcal{F}})}}"', from=3-1, to=3-3]
\end{tikzcd}\]
The diffeomorphism invariance of link cobordisms \cite[Theorem A]{Zemkelinkcobord} implies that the top square commutes up to chain homotopy. Now observe that $f$ fixes the $w$ and $z$-regions of $\mathcal{F}$ (due to the equivariant choice of the decoration). Hence the commutation of the lower square is tautological after switching the $w$ and $z$-type regions of $\overline{\mathcal{F}}$. Applying this in Relation~\ref{eq2}, we get 
\begin{gather}\label{eq3}
G_{W,{\mathcal{F}}} \circ sw \otimes sw \circ  t  \circ \rho_1 \otimes \rho_2  + sw \circ  t \circ  \rho_{\#} \circ G_{W,\mathcal{F}} \simeq G_{W,{\mathcal{F}}} \circ sw \otimes sw \circ  t  \circ (\Psi \rho_1 \otimes \Phi \rho_2)
\end{gather}
Next, we observe the following relations, stemming from the naturality of link cobordism maps:
\[
t \circ \Psi \otimes \Phi \simeq \Phi \otimes \Psi \circ t.
\]
The above relation follows since $t$ switches the two factors of the connected sum. Moreover, since the $sw$ map switches the basepoints $w$ and $z$, we have 
\[
sw \otimes sw \circ {\Phi \otimes \Psi} \simeq {\Psi \otimes \Phi} \circ {sw \otimes sw}.
\]
Applying these relations on the right-hand side of Relation~\ref{eq3}, we get
\begin{gather*}\label{eq4}
G_{W,{\mathcal{F}}} \circ sw \otimes sw \circ  t  \circ \rho_1 \otimes \rho_2  + sw \otimes sw \circ  t \circ  \rho_{\#} \circ G_{W,\mathcal{F}} \simeq G_{W,{\mathcal{F}}} \circ \Psi \otimes \Phi \circ sw \otimes sw \circ  t  \circ \rho_1 \otimes \rho_2 
\end{gather*}
Applying the definitions of the maps $\tilde{\tau}^{\prime}_{sw}$ and $\tilde{\tau}_{exch}$ from Subsection~\ref{KK} to Relation above, we get
\begin{gather}\label{eq5}
\tilde{\tau}^{\prime}_{sw} \circ  G_{W,\mathcal{F}}  \simeq  G_{W,\mathcal{F}} \circ 
(\mathrm{id} \otimes \mathrm{id} + \Psi \otimes \Phi) \circ \tilde{\tau}_{exch}.
\end{gather}
In Proposition~\ref{taufixproof}, we showed that there is chain homotopy equivalence $\tilde{\rho}$ which satisfies 
\[
\tilde{\rho} \circ \tilde{\tau}^{\prime}_{sw}  \simeq  \tilde{\tau}_{sw} \circ \tilde{\rho}
\] 
Hence, it follows that $\tilde{\rho} \circ G_{W, \mathcal{F}}$ intertwines with $\tilde{\tau}_{sw}$ and $(\mathrm{id} \otimes \mathrm{id} + \Psi \otimes \Phi) \circ \tilde{\tau}_{exch}$. Lastly, recall in \cite[Proposition 5.1]{Zemkeconnected}, it was shown that $G_{W, \mathcal{F}}$ is a homotopy equivalence. Since $\tilde{\rho}$ is a basepoint moving map, it admits a homotopy inverse. Hence $\tilde{\rho} \circ G_{W, \mathcal{F}}$ is also a homotopy equivalence. Hence, we chose $F=\tilde{\rho} \circ G_{W, \mathcal{F}}$, which completes the proof. 
\end{proof} 

\begin{remark}\label{compare}
A slightly different version of $\tilde{\tau}_{sw}$, was computed by Dai-Stoffregen and the author in \cite{abhishek2022equivariant}. Specifically, there the authors considered the connected sum $(K,w,z) \# (\tau K^{r}, \tau z, \tau w)$, while we do not flip the basepoints on $\tau K^{r}$, i.e. we consider $(K,w,z) \# (\tau K^{r}, \tau w,  \tau z)$. This implies that the involution on the disconnected end namely $\tilde{\tau}_{exch}$ was defined differently in \cite{abhishek2022equivariant}. The authors also compute of the $(\iota_K, \tau_K)$-connected sum formula for the switching involution \cite{abhishek2022equivariant}. For the present article, the behavior of the $\iota_K$ connected sum is rather regulated. Since $G_{W,\mathcal{F}}$ is exactly the one used in \cite[Proposition 5.1]{Zemkeconnected}, we at once obtain that
\[
\iota_{K \# \tau K^{r}} \circ G_{W,\mathcal{F}} \simeq  G_{W,\mathcal{F}} \circ 
(\mathrm{id} \otimes \mathrm{id} + \Phi \otimes \Psi) \circ \iota_K \otimes \iota_{\tau K^{r}}.
\]
by composing the above relation with $\tilde{\rho}$ on the left and observing that the push-forward map stemming from a finger moving diffeomorphism commute with $\iota_{K \# \tau K^{r}}$ (due to naturality), we get
\begin{gather}\label{iotaconnect}
\iota_{K \# \tau K^{r}} \circ F \simeq  F \circ 
(\mathrm{id} \otimes \mathrm{id} + \Phi \otimes \Psi) \circ \iota_K \otimes \iota_{\tau K^{r}}.
\end{gather}
\end{remark}

\begin{remark}
The appearance of half Dehn-twist map in $\tilde{\tau}_{exch}$ may give us the impression that $(\mathrm{id} \otimes \mathrm{id} + \Psi \otimes \Phi) \circ \tilde{\tau}_{exch}$ is not a homotopy involution but of order $4$ (following the intuition for $\iota_K$ maps), which would would apparently contradict the conclusion of Theorem~\ref{KconnectKstrong} (in light of Proposition~\ref{strong}). However, this is not the case. We leave it to the readers to verify that  $(\mathrm{id} \otimes \mathrm{id} + \Psi \otimes \Phi) \circ \tilde{\tau}_{exch}$ is indeed a homotopy involution. This is reminiscent of Proposition~\ref{strongskewbasepoint}.
\end{remark}

\subsection{Periodic involution on Thin knots} We now move on to compute periodic actions on knots that are \textit{Floer homologically thin}. As mentioned in the Introduction, there is a vast collection of thin knots that admit a periodic involution. In \cite{Petkova}, Petkova characterized the $CFK^{\infty}$ complexes of thin knots (up to homotopy equivalence). These complexes consist of direct sum of  \textit{squares} and a single \textit{staircase} of step length $1$. We will refer to such complexes for thin knots as \textit{model complexes}. Examples of such model $CFK^{\infty}$ complexes are shown in Figure~\ref{thinmodels1} and Figure~\ref{thinmodels2}.  


In \cite{HM}, Hendricks and Manolescu showed that it is possible to uniquely (up to grading preserving change of basis) determine the action of the $\spinc$-conjugation $\iota_{K}$, on for any thin knot $K$. As we will see in a moment the analogous situation for the periodic knots is more complicated. See Remark~\ref{thinremark}. We start with the following definition:

\begin{definition}
We will refer to an action $\tau_{0}$ on a knot Floer complex $CFK^{\infty}(K)$ as a \textit{periodic type} action if $\tau_0$ is a grading preserving, $\mathbb{Z} \oplus \mathbb{Z}$-filtered automorphism that squares to the Sarkar map $\varsigma$. 
\end{definition}
\noindent
We make a few observations for periodic type actions. In \cite{OStau}, the authors defined the concordance invariant $\tau$.\footnote{Not to be confused with an involution, for which we have been using the same notation.} This is relevant for us since for thin knots, the Maslov grading of a generator $[\x, i, j]$ is $i+j-\tau$. Now assume that, $\tau_{0}$ is a periodic type automorphism on $CFK^{\infty}(K)$. Hence $\tau_{0}$ is filtered, this implies $[\bold{x},i,j]$ is send to $[\tau_{0}(\bold{x}),i^{\prime}, j^{\prime}]$ where $i^{\prime} \leq i$ and $j^{\prime} \leq j$. Coupled with the fact that $\tau_{0}$ preserves the Maslov grading, we get $\tau_{0}(\bold{x})$ lies in the same diagonal line as $\bold{x}$ which implies $i^{\prime} = i$ and $j^{\prime} = j$.

On the other hand, we have $\tau^{2}_{0} \simeq \varsigma$. The Sarkar map $\varsigma$ was computed in \cite{sarkar2015moving}, in particular the map known to be identity for staircases. For square complexes $B_{i}=\{ a_i,b_i,c_i, U e_i \}$ as shown in Figure~\ref{square}, the map takes the form indicated below,
\[
\varsigma(a)= a+e, \;
\varsigma(b)=b, \;
\varsigma(c)=c, \;
\varsigma(e)=e. 
\]
\begin{figure}[h!]
\centering
\includegraphics[scale=1.3]{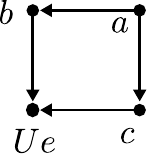} 

\caption
{A box complex.}\label{square}
\end{figure}
We typically refer to the generator $a$ as the \textit{initial vertex}/\textit{corner} of a box. 
\noindent
We will now see that these constraints on $\tau_{0}$ will be enough to determine it uniquely for a certain class of knots. We specify the class below,

\begin{definition}
We refer to a thin knot as \textit{diagonally supported} if the model complex of the knot does \textit{not} have any boxes outside of the main diagonal. 
\end{definition}
\noindent
For example, the figure-eight knot and the Stevedore knot are both diagonally supported, see Figure~\ref{figureeightcomplex}. 

\begin{figure}[h!]
\centering
\includegraphics[scale=.9]{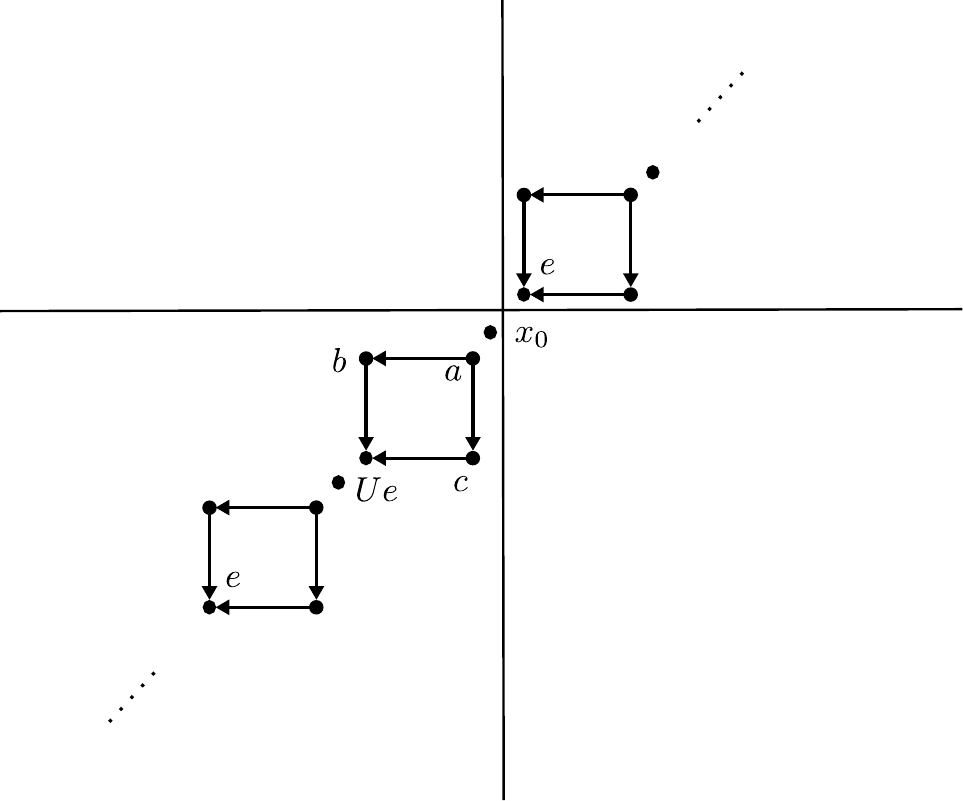} 

\caption
{$CFK^{\infty}(4_1)$, a model complex of a diagonally supported thin knot.}\label{figureeightcomplex}
\end{figure}
\noindent
We now state the following computational result, which is similar to its involutive counterpart:
\begin{lemma} \label{thin}
Let $K$ be a diagonally supported thin knot, then up to grading preserving change of basis, there is at most one periodic type automorphism of $CFK^{\infty}(K)$. 
\end{lemma}

\begin{proof}

Let $\tau_0$ be one such periodic type action. By $\partial_{\textrm{horz}}$ and $\partial_{\textrm{vert}}$ we denote the horizontal and vertical component of the differential $\partial$. It is clear that for any $\tau_{0}$ satisfying the hypothesis we have $\partial_{\textrm{horz}} \tau_{0} = \tau_{0} \partial_{\textrm{horz}}$ and $\partial_{\textrm{vert}} \tau_{0} = \tau_{0} \partial_{\textrm{vert}}$. This also implies $\partial_{\textrm{vert}} \partial_{\textrm{horz}} \tau_{0} = \tau_{0} \partial_{\textrm{vert}} \partial_{\textrm{horz}}$. Let  $\{B_{s}=\{ a_s,b_s,c_s, U e_s \} \}$ be the collection boxes that appear in the  main diagonal. Firstly, note that the for an initial vertex $a_s$, the image $\tau_{0} (a_s)$ must contain an initial vertex from one such box. This is because initial vertices $a_s$ are the only generator $y$ in the chain complex with the property that $\partial_{\textrm{vert}} \partial_{\textrm{horz}}(y) \neq 0$. Let $\{B_{s_{i}} \}$ be the subcollection of boxes from $\{ B_{s} \}$, for which $\tau (a_{s_{i}})$ contains at least one other initial vertex $a_{s_{j}}$ of a different box ($i \neq j$). We will now argue that all the squares in $\{ B_{s_{i}} \}$ can be paired so that they are related by $\tau_{0}$, allowing them to split off from the chain complex. Explicitly, let $\{a,b,c, Ue \}$ be a box in $\{ B_{s_{i}} \}$. Let $a^{\prime}$ be one of the initial corners that appears in $\tau(a)$. We will now change basis by sending $a^{\prime}$ to $\tau_{0} (a)$, $\partial_{\textrm{vert}}(a^{\prime})$ to $\tau_{0} (c)$ , $\partial_{\textrm{horz}}(a^{\prime})$ to $\tau_{0} (b)$ and $\partial_{\textrm{vert}} \partial_{\textrm{horz}}(a^{\prime})$ to $\tau_{0} (U e)$. Now let us denote the direct sum of the two boxes $\{a,b,c, Ue \}$ and $\{\tau_{0}(a),\tau_{0} (b), \tau_{0} (c), \tau_{0} (Ue) \}$ as $C_1 \otimes \mathbb{Z}_{2}[U,U^{-1}]$. We see that $\tau_{0}$ is an automorphism of the direct sum. For example, we have,
\[
\tau_0(\tau_{0}(a))=a + e, \; \tau_{0} (\tau_{0} (b))=b, \; \tau_{0} (\tau_{0} (c))=c, \; \tau_{0} (\tau_{0} (e))=e.
\]
Here we have used the definition of the Sarkar map on the squares. We would now like to split off  $C_1 \otimes \mathbb{Z}_{2}[U,U^{-1}]$ from $CFK^{\infty}(K)$. As we already showed that $\tau_{0}$ preserves such direct sum, we are left to show that $\tau_{0}$-image of no other generator of $CFK^{\infty}(K)$ (lying outside $C_1 \otimes \mathbb{Z}_{2}[U,U^{-1}]$) lands in $C_1 \otimes \mathbb{Z}_{2}[U,U^{-1}]$. 

To achieve this, we first show that if there is a generator $x$ outside $C_1$ such that $\tau_{0}(x)$ contains either $a$ or $\tau_{0} (a)$ (or both), then it is possible to change basis so that $\tau_{0} (x)$ no longer contains either of those elements.  We show this by considering separate cases. 

\begin{itemize}

\item Firstly, let $\tau_{0}(x) = a + m$, where $m$ is a sum of basis elements none of which are $a$. Now if none of the basis elements in $m$ are $\tau_{0}(a)$, we change basis by sending $x$ to $x+ \tau_{0}(a)$. Then $\tau_{0}(x)$ will not contain either $a$ or $\tau_{0}(a)$.

\item Now suppose, $\tau_{0}(x) = \tau_{0}(a) + m$, where $m$ as a sum of basis elements does not contain $\tau_{0}(a)$. Now note that $\tau^{2}_{0}(x)$ does not have $\tau_{0}(a)$ (when written as a sum of basis elements). This is because it equals to the Sarkar map $\varsigma$ and $\varsigma(x)= x \; \mathrm{or} \; x + U{e_{k}}$, for some $e_{k}$. So there is an even number of basis elements $z_i$ in $m= z_1 + z_2 + \cdots + z_n$ such that $\tau_{0}(z_i)$ contains $\tau_{0}(a)$. We now  simultaneously change any basis elements $y_i$ (outside $C_1$) to $y_i + a$, if $\tau_{0}(y_i)$ contained $\tau_{0}(a)$. In particular, $x$ is changed to $x+a$. Hence, $\tau_{0}(x)$ no longer contains $\tau_{0}(a)$. Moreover, $\tau_{0}(x)$ also can not contain $a$. Since if one of the $z_i$ was $a$, then there would be an odd number of $z_i$ (without counting $a$), which would go through the change of basis, thereby eliminating $a$ in the new basis. If none of $z_i$ was $a$, then after change of basis there is still no $a$ in $m$.

\end{itemize}
\noindent
By repeated application of the above steps, we can ensure that there are no basis element $x$ outside $C_1$, such that $\tau_{0}(x)$ contain $b$ or $\tau_{0}(b)$ and so on. Hence, we can split off  $C_1 \otimes \mathbb{Z}_{2}[U,U^{-1}]$ from the chain complex, and be left with a chain complex where number of boxes  $\{ B_{s_{i}} \}$ is precisely smaller by 2. After repeating this process, we can ensure that $\{ B_{s_{i}} \}$ is empty. So we are now left with boxes in $\{ B_{s} \}$ with initial corner $a_s$ such that $\tau_{0} (a_s)$ only contains the initial corner $a_s$. 

We now show that there can be at most one such box. We argue similarly as in \cite{HM}. By multiplying with $U$ if necessary, we can assume all the initial corners of the boxes left are in the grading $(0,0)$. Let the boxes be enumerated as $B_1$, $B_2 \cdots$, $B_p$. This implies $\{a_1, a_2, \cdots , a_p, e_1 , e_2, \cdots, e_p, x_0 \}$ are the only generators in $(0,0)$. Note that $\tau_{0}(e_i)=e_i$. We now claim that $\tau_{0}(a_j)$ must contain $x_{0}$ non-trivially, for any $j \in \{ 1, 2, \cdots, \}$. If not, then we can write 
\[
\tau_{0}(a_j) = a_j + \sum_{i=1}^{p} n_{i} e_{i},
\]
where $n_{i} \in \{ 0, 1 \}$. This implies $\tau_{0}^{2}(a_j)=a_j$, contradicting that $\tau_{0}^{2}$ is equal to the Sarkar map. Hence we get,
\[
\tau_{0}(a_j) = a_j + \sum_{i=1}^{p} n_{i} e_{i} + x_{0}.
\]
Which implies that
\begin{align*}
\tau_{0}(x_0)&= a_j + \sum_{i=1}^{p} n_{i} e_{i} +  x_{0} +  \sum_{i=1}^{p} n_{i} \tau_{0} (e_{i}) + a_{j} + e_{j} \\
&= x_{0} + e_j.
\end{align*}
This is a contradiction when $p \geq 2$, since $j \in \{ 1, 2,\cdots, p \}$ was arbitrary.
We are now in a case where we have at most one box in the main diagonal along with the staircase. In the case where there is only a staircase with no boxes, the $\tau_{0}$ action is identity, see Proposition~\ref{lspaceperiodic}. In the case where there is one box with an initial corner at $(0,0)$ with a staircase, we will need to show that there can be at most one $\tau_{0}$ action, up to change of basis. We now consider 4-cases depending on the nature of the stair and the placement of the box, see Figure~\ref{thinmodels1} and Figure~\ref{thinmodels2}. Note that irrespective of the case, using the constraints from Sarkar map we can show $\tau_{0}$ must take one of these two possible values,
\[
\tau_{0}(a)=a+x_0, \; \tau_{0}(x_0)=x_0 + e , \; \tau_{0}(e)=e.
\]
\[
\tau_{0}(a)= a+ x_0 + e, \; \tau_{0}(x_0)=x_0 + e , \; \tau_{0}(e)=e.
\]
However, upon changing the basis by sending $x_0$ to $x_0 + e$ one can see that these two apparently different actions are equivalent. Finally, using the various constraints again, it is easily checked that for Case-I and II, $\tau_{0}$ acts as 
\[
\tau_{0} (b)= b + x^{1}_1, \; \tau_{0} (c)= c + x^{2}_{1}, \; \tau_{0}(x^{i}_{s})=x^{i}_{s}.
\]
And for Case-III and IV $\tau_{0}$ acts as 
\[
\tau_{0} (x^{2}_{1}) = x^{2}_{1} + U^{-1} c , \; \tau_{0} (x^{1}_{1}) = x^{1}_{1} + U^{-1} b , \; \tau_{0} (b)=b, \; \tau_{0}(c)=c, \; \tau_{0} (x^{1}_{s}) = x^{1}_{s}, \; \tau_{0} (x^{2}_{s}) = x^{2}_{s} \; \textrm{for} \; s > 1.
\]
This completes the proof.
\begin{figure}[h!]
\centering
\includegraphics[scale=1]{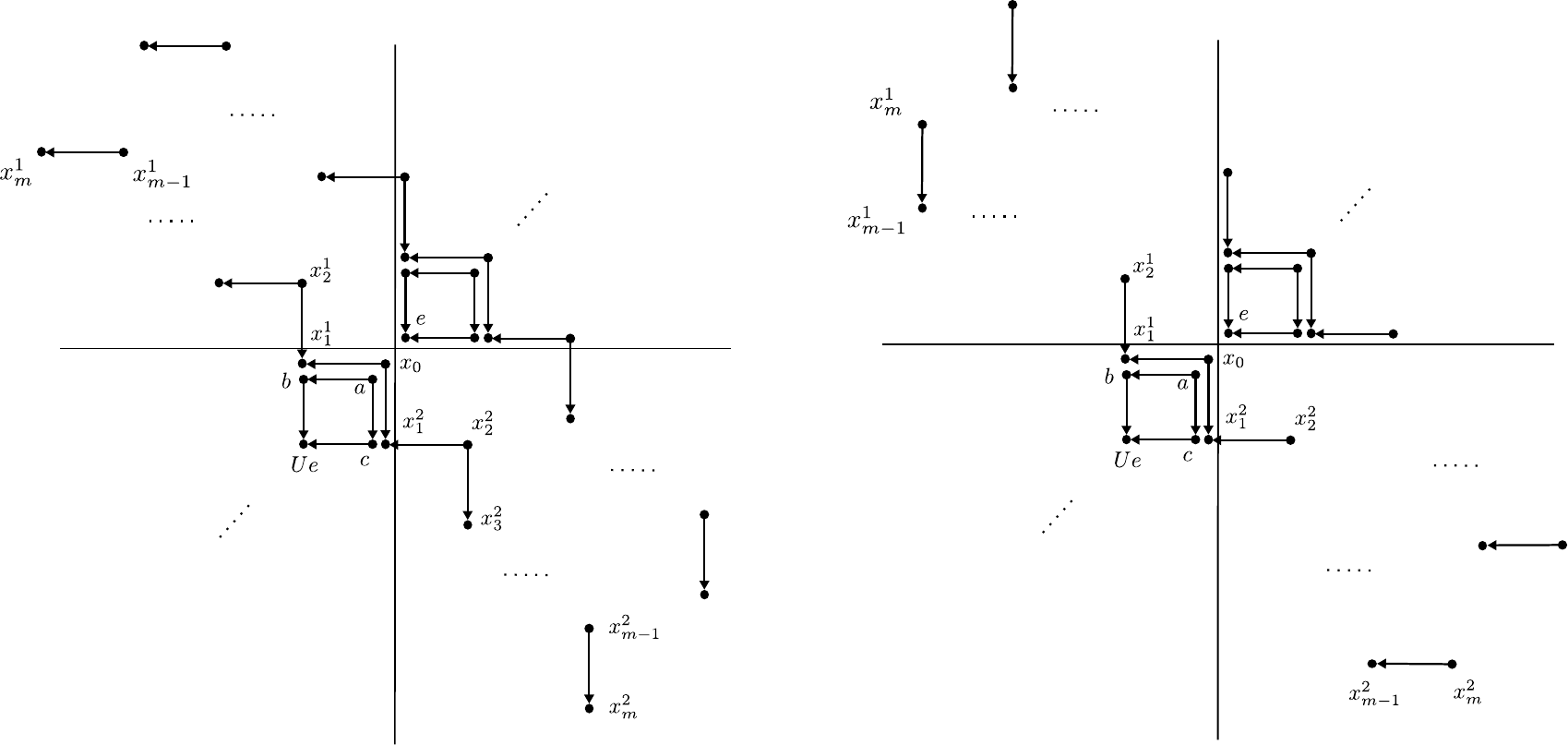} 

\caption
{Thin knot models complex, left: Case-I, right: Case-II.}\label{thinmodels1}
\end{figure}

\begin{figure}[h!]
\centering
\includegraphics[scale=1]{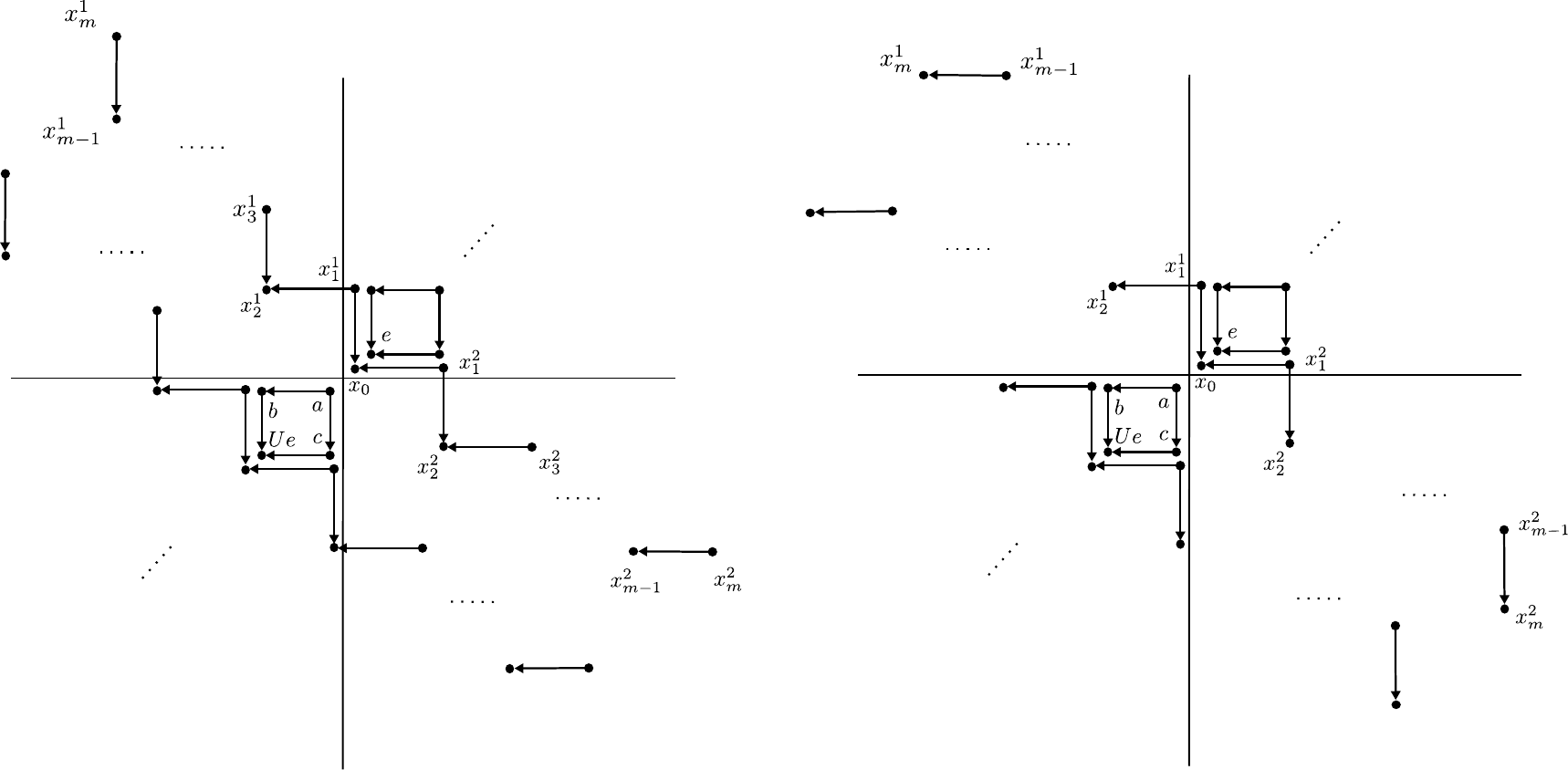} 

\caption
{Thin knot models complex, left: Case-III, right: Case-IV}\label{thinmodels2}
\end{figure}

 \end{proof}

\begin{remark}\label{thinremark}
Note that, despite being similar in nature to the $\iota_K$ action, we are unable to determine whether periodic type action is unique for all thin knots. This is because if there is a box $B$ outside of the main diagonal in grading $(i,j)$, $\iota_K$ action pairs $B$ with a box in grading $(j,i)$, and then the two boxes are split off from the chain complex. This process is repeated until there are no boxes left outside of the main diagonal. However, in the case of periodic type action, $\tau_{0}$ maps the initial corner of $a$ of $B$ to generators in the $(i,j)$ grading. In particular, $\tau_{0}(a)$ may not contain an initial corner of a different box (which is guaranteed for $\iota_K(a)$), making it harder to split off from the complex. Motivated from above, We pose the following question.
\begin{question}
Do there exist periodic thin knots with two periodic actions on $CFK^{\infty}$, such that the actions are not conjugate to each other via a change of basis? 
\end{question}
\noindent
Currently, the author does not have any conjectured example for the above. Although, one may expect to use equivariant cobordism techniques used in \cite{dai2020corks} to distinguish two periodic involutions by looking at the local equivalence class of large surgeries.
\end{remark}

\begin{proof}[Proof of Theorem~\ref{thin_periodic}]
Follows from Lemma~\ref{thin}.
\end{proof}

\section{Concordance invariants}\label{concordanceinvariant}

In this section we will define two concordance invariants arising from the strong symmetry $\tilde{\tau}_{sw}$ on the knot $K \# K^{r}$ defined in Subsection~\ref{KK}.\footnote{Here and in the subsequent sections, we abbreviate $\tau K^{r}$ to $K^{r}$.} Note that in Theorem~\ref{KconnectKstrong}, we already computed the strong involution $\tilde{\tau}_{sw}$. Furthermore, recall that in Section~\ref{correction} we defined two quantities $\underline{d}_{\tau}$, $\bar{d}_{\tau}$. We now show that these invarinats can also be used to produce concordance invariants.

\begin{proof}[Proof of Theorem~\ref{concordance}]
Let $K_1$ and $K_2$ be two knots in $S^{3},$ that are concordant to each other. Let $p$ be an odd integer such that $p \geq \textrm{max} \{2g(K_1), 2g(K_2) \}$ (here $g(K_i)$ represents the $3$-genus of the knot $K_i$). By Theorem~\ref{t1} and Corollary~\ref{corollary}, we then have $HFI^{+}_{\tau}(S^{3}_{p}(K),[0]) \cong H_{*}(AI_{0}^{+,\tau})$.
 and $HFI^{+}_{\iota \tau}(S^{3}_{p}(K),[0]) \cong H_{*}(AI_{0}^{+,\iota \tau})$. We now claim that by Montesinos trick  $S^{3}_{p}(K_i \# K^{r}_i)$ is a double branched cover of a cable $C(K_{i})$ of the knot $K_i$, where the covering involution correspondence to the strong involution on $K_i \# K^{r}_i$; see for Example \cite{Montesinos1975}, \cite[Section 1.1.12]{savelievinvariants}. Now since $K_{1}$ and $K_{2}$ are concordant, so is $C(K_{1})$ and $C(K_{2})$. Which in turn implies that their double branched covers are $\mathbb{Z}_{2}$-homology cobordant and the covering involution extend over the cobordism (since the cobordism is obtained by taking the double branched cover of the concordance). The conclusion follows from Theorem~\ref{dinvariant}. 

\end{proof}

\begin{remark}\label{knotconcordanceremark}
These invariants can be considered as a part of a more general set of invariants in the following sense: Given a knot $K$, consider the double branch cover $\Sigma_{2}(P(K))$ of a satellite knot $P(K)$ with companion $K$. Then  $\underline{d}_{\tau}(\Sigma_{2}(P(K))$ and $\bar{d}_{\tau}(\Sigma_{2}(P(K))$ are invariants of the concordance class of $K$. However, computing these invariants might be challenging. For the present case, we were able to identify $\Sigma_{2}(P(K))$ to surgery on a strongly invertible knot and make use of the equivariant surgery formula in conjunction with the computation of the action of the symmetry on $CFK^{\infty}$. See Example~\ref{KconnectKcompute} for a computation.
\end{remark}
\begin{figure}[h!]
\centering
\includegraphics[scale=1]{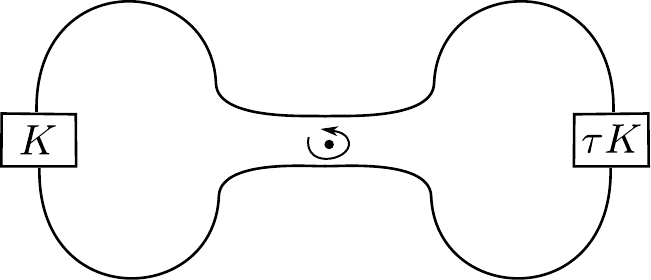} 

\caption
{The periodic involution on the knot $K \# \tau K$.}\label{KconnectKperiod}
\end{figure}
\begin{remark}
Instead of considering the strong involution on $K \# K^{r}$, one might consider the obvious periodic involution $\tau^{p}_{0}$ on $K \# K$, shown in Figure~\ref{KconnectKperiod}. This action was already computed by Juh{\'a}sz,  and Zemke in \cite[Theorem 8.1]{juhasz2018stabilization}. Using a slightly different argument it is possible to show that $\underline{d}_{\tau}(S^{3}_{p}(K \# K), \tau^{p}_{0})$, $\bar{d}_{\tau}(S^{3}_{p}(K \# K), \tau^{p}_{0})$ are also concordance invariants. 
\end{remark}

\section{Examples}\label{computation}
In this Section, we give explicit computation for the action of symmetry on $CFK^{\infty}$ for two classes of symmetric knots. This in turn leads to computations of $\underline{d}_{\tau}$ and $\bar{d}_{\tau}$ invariants.  But before moving ahead, for ease of computation, we would like to show that we can use any chain complex filtered homotopic to that coming from $CFK^{\infty}(\mathcal{H}_{K})$ for a given $\H_K$. This is similar to the observation made in \cite{HM}.

\begin{definition}\cite{HM}
A tuple $\mathfrak{C}_{p}=(C,\mathcal{F},M, \partial, \tau_{p})$ is called $CFKI_{\tau_{p}}$ data if $C$ is a free, finitely generated $\mathbb{Z}_2[U,U^{-1}]$ module with a $\mathbb{Z}$ grading $M$, and a $\mathbb{Z} \oplus \mathbb{Z}$-filtration $\mathcal{F}(i,j)=(i,j)$. The module $C$ is further equipped with a differential $\partial$ which preserves the filtration $\mathcal{F}$ and decreases the grading $M$ by $1$. Finally, $\tau_{p}$ is a grading preserving filtered quasi-isomorphism on $C$. Similarly one can define a tuple $\mathfrak{C}_{s}=(C,\mathcal{F},M, \partial, \tau_{s})$ with exact same properties except in this case $\tau_{s}$ is a \textit{skew}-filtered.
 \end{definition}
With this in mind one can define a notion of \textit{quasi-isomorphism} between two such  $CFKI_{\tau_{p}} $-data (respectively $CFKI_{\tau_{s}}$) by following \cite{HM}, leading to the following Proposition:

\begin{proposition}\cite[Lemma 6.5]{HM}\label{transfer}
Given a $CFKI_{\tau_{p}}$ data $\mathfrak{C_{p}}=(C,\mathcal{F},M, \partial, \tau_{p})$-data, let $(C^{\prime},\mathcal{F}^{\prime},M^{\prime},\partial^{\prime})$ be another chain complex but without the quasi-isomorphism $\tau^{\prime}_{p}$, which is filtered homotopy equivalent to $(C,\mathcal{F},M, \partial)$. Then there is a grading preserving filtered quasi-isomorphism $\tau^{\prime}_{p}$ on $C^{\prime}$ such that $\mathfrak{C}_{p}$ and $(C^{\prime},\mathcal{F}^{\prime},M^{\prime},\partial^{\prime},\tau^{\prime}_{p})$ are quasi-isomorphic as $CFKI_{\tau_{p}}$-data. 
Analogous conclusions also holds for the $CFKI_{\tau_{s}}$-data.
\end{proposition}

We also note that chain isomorphism in Theorem~\ref{t1} can be made absolutely graded after a grading shift. This in turn results in the following relation as given below, see \cite[Section 6.7]{HM} for a discussion.

\begin{equation}\label{gradingshift}
\begin{aligned}
  \underline{d}_{\circ}(S^{3}_{p}(K),[0]) &= \frac{p-1}{4} + \underline{d}_{\circ}(A^{-}_{0},[0]), \\
  \bar{d}_{\circ}(S^{3}_{p}(K),[0]) &= \frac{p-1}{4} + \bar{d}_{\circ}(A^{-}_{0},[0]).
\end{aligned}
\end{equation}
where $\circ \in \{ \tau , \iota \tau \}$. We now state the examples. 

\subsection{Strongly invertible L-space knots and their mirrors}\label{strongKconenctK}

As discussed in the Introduction, many L-space knots admit a strong involution. Here we show that we can explicitly compute the involution for those knots and their mirrors. 

Ozsv{\'a}th and Szab{\'o} showed that the knot Floer homology of these knots is determined by their Alexander polynomial \cite{Lens}. Specifically, the knot Floer homology $CFK^{\infty}(K)$ of an L-space knot is chain homotopic to $C \otimes \mathbb{Z}_{2}[U,U^{-1}]$, where $C$ is a staircase complex as shown in Figure~\ref{stair}. In a similar fashion, the knot Floer homology of the mirror of an L-space knot can be taken to be copies of staircases shown in Figure~\ref{stair}.
\begin{figure}[h!]
\centering
\includegraphics[scale=1]{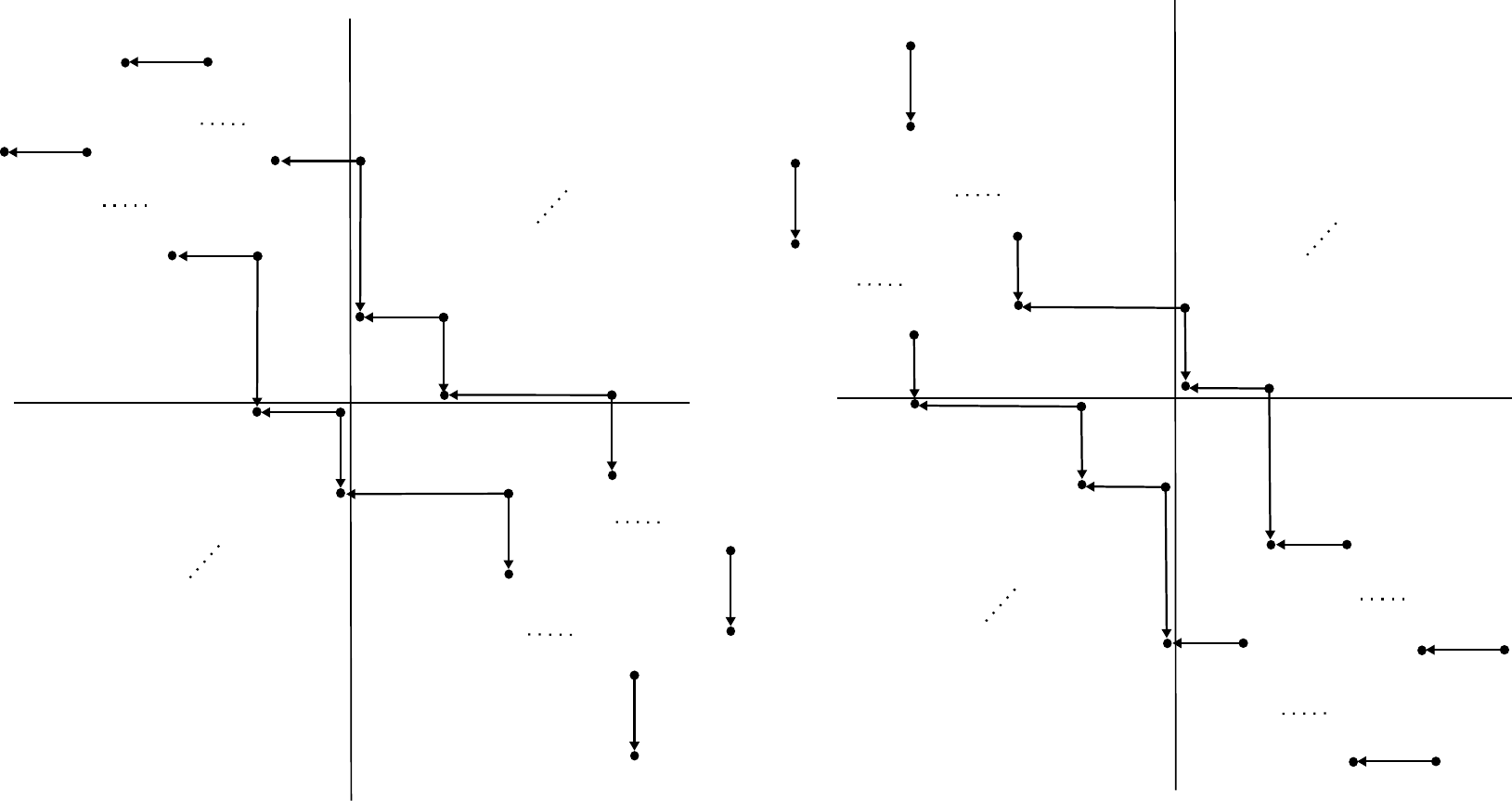} 

\caption
{Examples of model $CFK^{\infty}$ for L-space knots (on the left), and mirror of L-space knots (on the right).}\label{stair}
\end{figure}
\noindent

We define the quantity $n(K)$ as follows. Let $\Delta_{K}(t)$ denote the Alexander polynomial for an L-space knot, written as
\[
\Delta_{K}(t)= (-1)^{m} + \sum_{i=1}^{m} (-1)^{m-i} (t^{n_{i}} + t^{-n_{i}})
\]
where $n_{i}$ are a strictly increasing sequence of positive integers. Then $n(K)$ defined as 
\[
n_{K}= n_m - n_{m-1} + \cdots + (-1)^{m-2} n_2 + (-1)^{m-1} n_1.
\]

\noindent
Recall that in \cite[Section 7]{HM}, Hendricks and Manolescu computed the $\spinc$-conjugation action $\iota_K$ for L-space knots and its mirrors. Analogously, we have the following computation. 
\begin{proposition}\label{lspacestrong}
Let $K$ be an L-space knot (or the mirror of an L-space knot) that is strongly invertible with the strong-involution $\tau_{K}$. We have 
\[
\iota_{K} \simeq \tau_{K}.
\]
In particular, we get 

\begin{enumerate}

\item $\underline{d}_{\tau}(S^{3}_{p}(K),[0])=\bar{d}_{\tau}(S^{3}_{p}(K),[0])= \frac{p-1}{4} - 2n(K)$, \;  when $K$ is an L-space knot, 
\vspace{1mm}
\item $\underline{d}^{\tau}(S^{3}_{p}(K),[0])=\frac{p-1}{4}$, $ \bar{d}^{\tau}(S^{3}_{p}(K),[0])=\frac{p-1}{4} + 2n(K)$, \; when $K$ is mirror of an L-space knot.

\end{enumerate}

Here $p$ is an integer such that $p \geq g(K)$.
\end{proposition}

\begin{proof}
Recall that if $\tau_{K}$ is a strong involution on a knot $K$, the induced map on $CFK^{\infty}$ is a skew filtered map that squares to the Sarkar map (see Proposition~\ref{strong}). Furthermore, the Sarkar map is the identity for staircase complexes, \cite{sarkar2015moving}. Firstly, by Proposition~\ref{transfer} we transfer both $\iota_K$ and $\tau_{K}$ to the model complexes depicted in Figure~\ref{stair}. As seen in \cite{HM}, the fact that $\iota_K$ is grading preserving skew-filtered and it squares to identity is enough to uniquely determine it for both L-space knots and the mirrors of L-space knots. Hence the claim follows from \cite[Proposition 7.3]{HM} after observing the relation between $\underline{d}$ and $\bar{d}$ invariants with the $\underline{V}_{0}$ and $\overline{V}_{0}$ invariants.
\end{proof}

Using the above we also prove Proposition~\ref{iotaidentify}.
\begin{proof}[Proof of Proposition~\ref{iotaidentify}]
Montesinos trick \cite{Montesinos1975} imply $(S^{3}_{n}(\overline{T}_{p,q}),\tau)$ is equivariantly diffeomorphic to double branch cover $(\Sigma_{2}(L_{n,p,q}),\tau_{n,p,q})$. The proof then follows from the Theorem~\ref{t1}, \cite[Theorem 1.5]{HM} and Proposition~\ref{lspacestrong} after observing that $n \geq (p-1)(q-1)/2$ is a large surgery for the mirrored torus knots $\overline{T}_{p,q}$.

\end{proof}
 
\subsection{Periodic involution on  L-space knots and their mirrors}
Several L-space knots admit a periodic involution. For Example, it is easy to visualize that $(2,2n+1)$-torus knots admit a periodic involution. For such knots, we have the following.

\begin{proposition}\label{lspaceperiodic}
Let $K$ be an L-space knot (or mirror of an L-space knot) with a periodic involution $\tau_K$. we have 
\[
\tau_K \simeq \mathrm{id}.
\]
In particular, we get,
\begin{enumerate}
\item $\underline{d}_{\tau}(S^{3}_{p}(K),[0])=\bar{d}_{\tau}(S^{3}_{p}(K),[0])=\frac{p-1}{4} - 2n(K)$, \;  when $K$ is an L-space knot,
\vspace{1mm}
\item $\underline{d}_{\tau}(S^{3}_{p}(K),[0])= \bar{d}_{\tau}(S^{3}_{p}(K),[0])=\frac{p-1}{4}$, \; when $K$ is mirror of an L-space knot.
\end{enumerate}
Here $p$ is an integer such that $p \geq g(K)$.
\end{proposition}

\begin{proof}
Note that when $\tau_K$ is a periodic involution, the induced map on the knot Floer complex is a grading preserving filtered map that squares to the Sarkar map (see Proposition~\ref{prop:tauK}). For the case in hand, we get $\tau^{2}_{K} \simeq \textrm{id}$. The constraints then imply $\tau_K= \textrm{id}$. Finally note that $d(S^{3}_{p}(K),[0])=\frac{p-1}{4} - 2n(K)$ if $K$ is an L-space knot, and $d(S^{3}_{p}(K),[0])=\frac{p-1}{4}$ of $K$ is the mirror of an L-space knot and that $\tau= \textrm{id}$ implies $\underline{d}^{\tau}=d=\bar{d}^{\tau}$ \cite[Proposition 4.6]{hendricks2021applications}.

\end{proof}

\begin{remark}
Using Corollary~\ref{corollary}, we can also compute the $\underline{d}_{\iota \tau}$ and $\bar{d}_{\iota \tau}$ invariants for L-space knots. Similarly, as a parallel to \cite[Proposition 8.2]{HM}, we can compute $\underline{d}_{\tau}$, $\bar{d}_{\tau}$  for thin knots using   Theorem~\ref{thin_periodic}. We omit the computation for brevity. 
\end{remark}

\subsection{Computing $\tau$-local equivalence class using the surgery formula}
Associated to $(Y,\tau)$, an integer homology sphere $Y$ equipped with an involution $\tau$, the invariant $h_{\tau}:= [(CF^{-}(Y), \tau)]$ was studied in \cite{dai2020corks}. Here $ [(CF^{-}(Y), \tau)]$ represents the local equivalence class of the $\iota$-complex $(CF^{-}(Y), \tau)$, see Definition~\ref{iota}. In \cite{dai2020corks}, $h_{\tau}$-invariant was bounded using equivariant cobordism and monotonicity results in local equivalence associated with it. Here, we show that in certain situations the equivariant surgery formula is enough to bound $h_{\tau}$ (in the sense of the partial order shown in \cite{dai2018infinite}).

\begin{example}
We consider the right-handed trefoil $T_{2,3}$ and $(-1)$-surgery on $T_{2,3}$. Note that $T_{2,3}$ has two different symmetries, one strong $\tau^{s}$ and one periodic $\tau^{p}$, as seen in Figure~\ref{trefoillocal}.
\begin{figure}[h!]
\centering
\includegraphics[scale=.7]{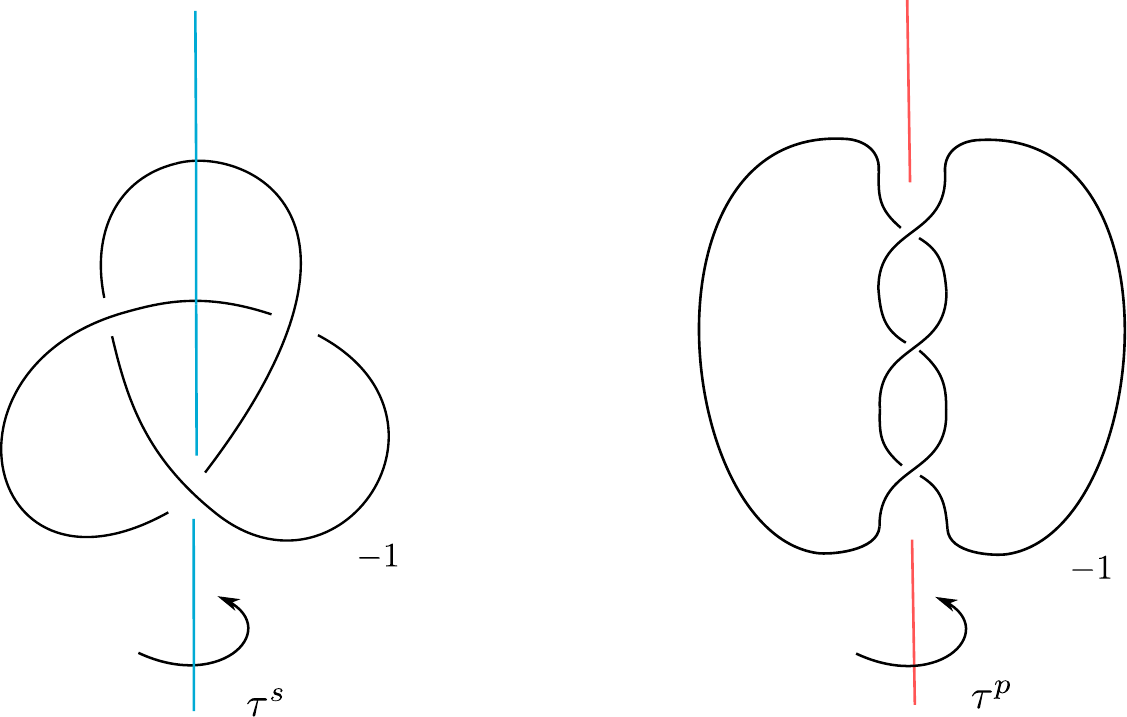} 

\caption
{Symmetries of $T_{2,3}$.}\label{trefoillocal}
\end{figure}
Proposition~\ref{lspacestrong} and the Proposition~\ref{lspaceperiodic} immediately imply that  $\tau^{s}$ and $\tau^{p}$ action on $HF^{+}(S^{3}_{+1}(\overline{T}_{2,3}))$ are as in Figure~\ref{trefoilintro}.

\noindent
It follows that, $\tau$-local equivalence class $h_{\tau}(S_{-1}^{3}(T_{2,3}), \tau^{s}) < 0$ and $h_{\tau}(S_{-1}^{3}(T_{2,3}), \tau^{p}) = 0$. This gives an alternative proof of \cite[Lemma 7.1]{dai2020corks}. Note that we are implicitly using that $\tau$-action is well-behaved under orientation reversal isomorphism, the proof of which follows from a discussion similar to that in \cite[Subsection 4.2]{HM}.
\end{example}
Readers may compare the Example above with  \cite[Proposition 6.27]{flexible_equivariant}.

\begin{example}
We now look at the $(+1)$-surgery on figure-eight knot $4_1$, and the periodic involution $\tau^{p}$ on it, as in Figure~\ref{figure_eight_periodic}. We identify this involution on the knot Floer complex of $4_1$ using Theorem~\ref{thin_periodic}. The resulting action on $HF^{-}(S^{3}_{p}(4_1))$ is shown in Figure~\ref{figure_eight_periodic}. As before we get that the local equivalence class $h_{\tau}((S^{3}_{+1}(4_1),\tau^{p}) < 0$, recovering \cite[Lemma 7.2]{dai2020corks}.
 \end{example}

\begin{example}
Here we consider the periodic involution $\tau^{p}$ on the Stevedore knot $6_1$ as depicted in Figure~\ref{stevedore}. Again as a consequence Theorem~\ref{thin_periodic}, we obtain the resulting action on the $CFK^{\infty}$ complex. 
\begin{figure}[h!]
\centering
\includegraphics[scale=.75]{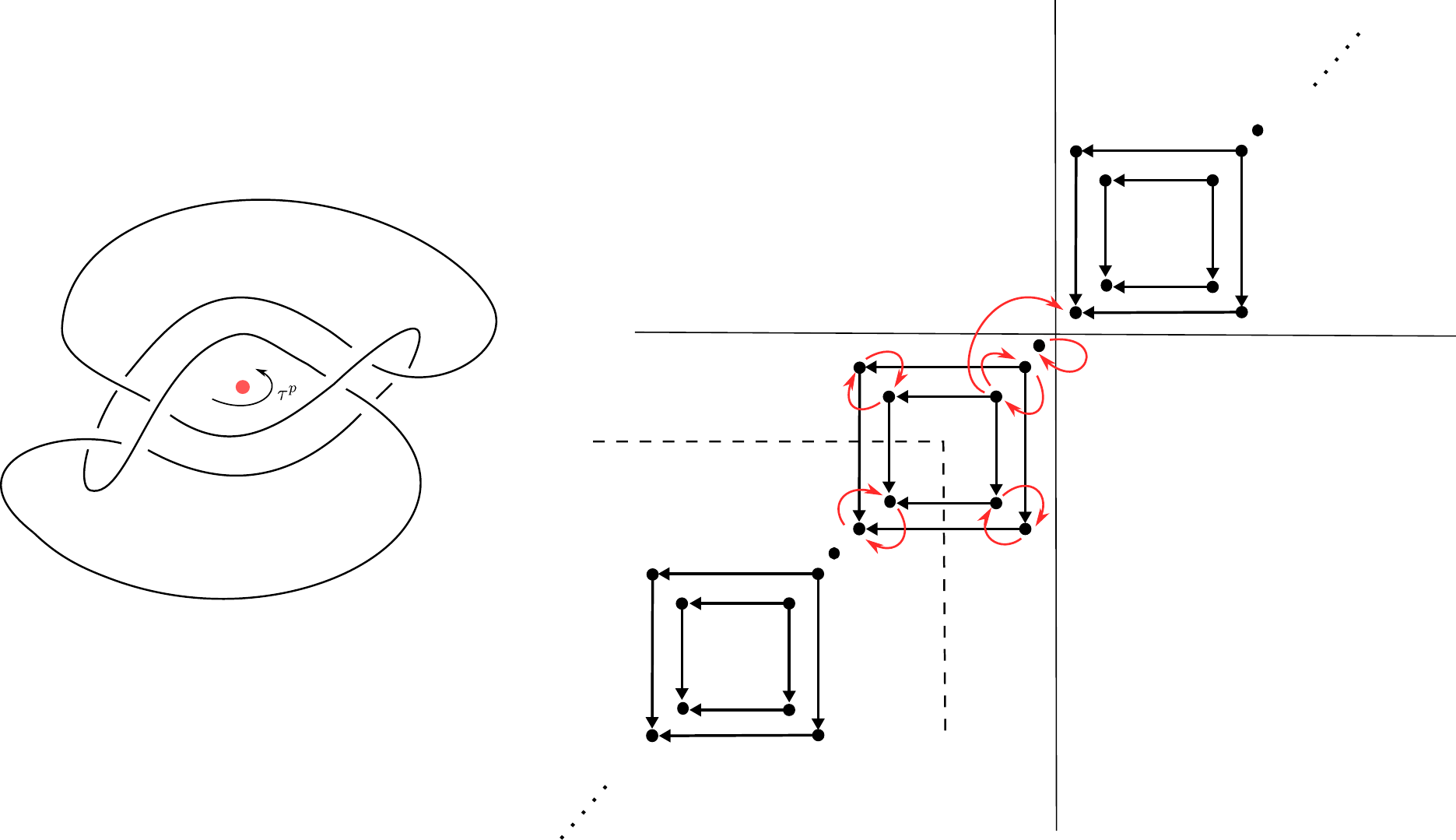} 
\caption
{Left: Periodic symmetry $\tau^{p}$ of $6_1$. Right: action of $\tau^{p}$ on $CFK^{\infty}(6_1)$.}\label{stevedore}
\end{figure}
It follows that $\tau^{p}$ acts on $HF^{-}(S^{3}_{+1}(6_1))$ as shown in Figure~\ref{stevedorehomology}.
\begin{figure}[h!]
\centering
\includegraphics[scale=.8]{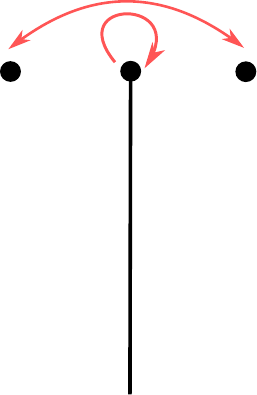} 
\caption
{Left: The action of $\tau^{p}$ on $HF^{-}(S^{3}_{+1}(6_1))$.}\label{stevedorehomology}
\end{figure}
The above computation also implies that the local equivalence class $h_{\tau}(S^{3}_{+1}(6_1),\tau^{p})$ is trivial. 
\end{example}

\begin{remark}
 In fact, it is known that $(6_1, \tau^{p})$ is equivariantly slice \cite[Example 1.12]{naikequivariant}. By a standard argument, it can be shown that this implies the induced involution extends over the contractible manifold that $S^{3}_{+1}(6_1)$ bounds. On the contrary, it was shown in \cite[Lemma 7.3]{dai2020corks} that the two strong involutions on $S^{3}_{+1}(6_1)$ coming from the two strong involutions on $6_1$ does not extend over any homology ball that $S^{3}_{+1}(6_1)$ may bound. 
\end{remark}

\begin{example}
In Lemma~\ref{thin}, we show that for diagonally supported thin knots there is at most one periodic type action. However, sometimes the grading and filtration restrictions allow us to deduce the uniqueness of periodic type actions for thin knots that are not diagonally supported. For example, the action of the periodic involution $\tau^{p}$ on the $\overline{6}_2$ knot, shown in Figure~\ref{62knot} can be uniquely determined from the constraints. We then compute the action of $\tau^{p}$ on $HF^{+}(S^{3}_{2}(\overline{6}_2),[0])$ in Figure~\ref{62homology}. 
\begin{figure}[h!]
\centering
\includegraphics[scale=.9]{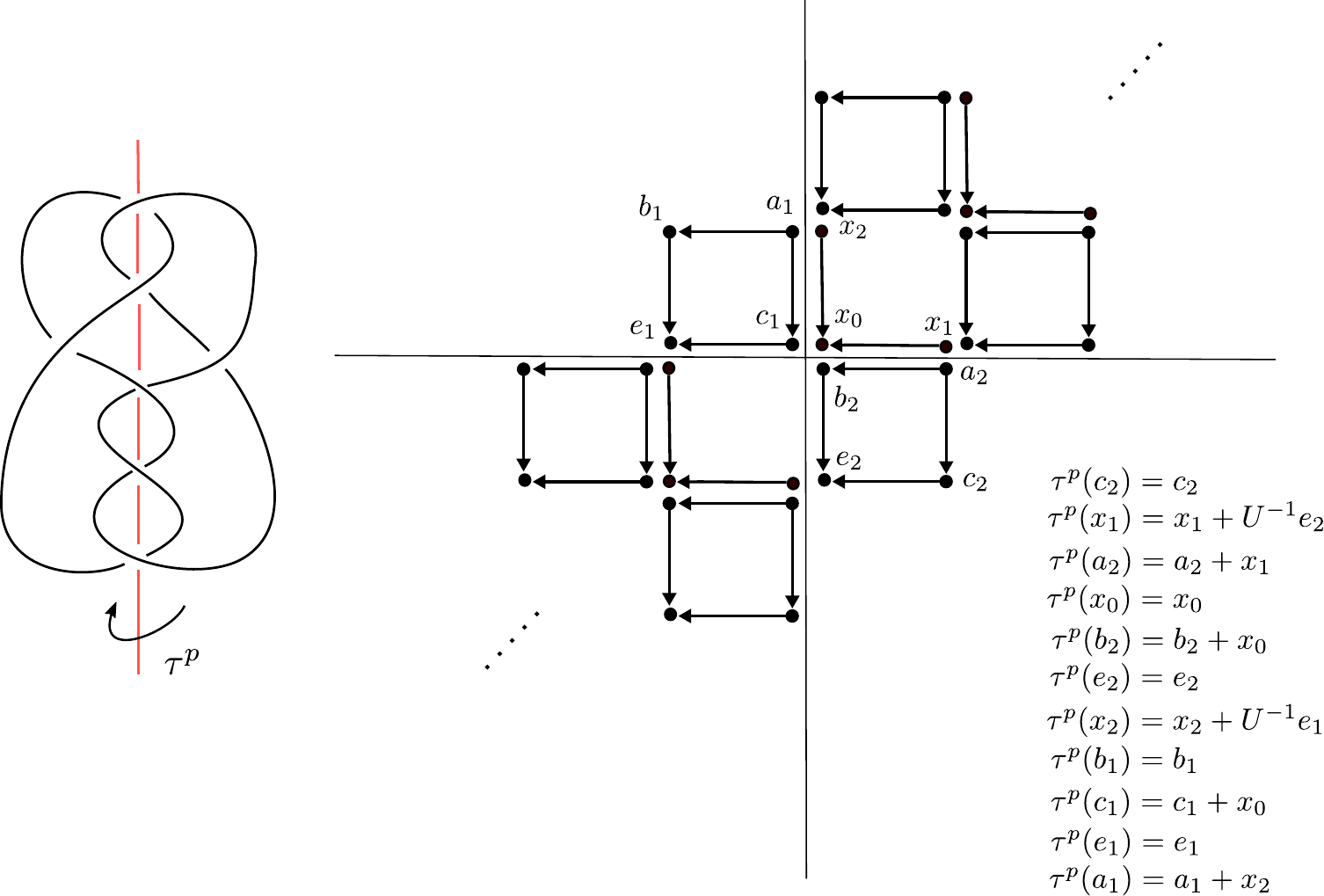} 
\caption
{Left: Periodic involution $\tau^{p}$ on $\overline{6}_2$, Right: $CFK^{\infty}(\overline{6}_2)$, with $\tau^{p}$ action.}\label{62knot}
\end{figure}
\begin{figure}[h!]
\centering
\includegraphics[scale=.75]{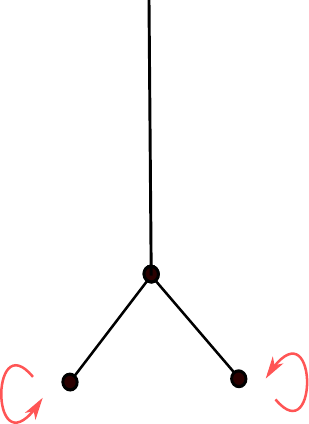} 
\caption
{The action of  $\tau^{p}$ on $HF^{+}(S^{3}_{2}(\overline{6}_2),[0])$.\label{62homology}}\label{62knot}
\end{figure}

\end{example}

\begin{example}\label{strong_figure_eight}
In \cite[Example 4.6]{dai2020corks} it was shown two different strong involutions $\tau$ and $\sigma$ acting on figure-eight knot induce different actions on $HF^{-}(S^{3}_{+1}(4_1))$. This was done by showing $h_{\tau}(S^{3}_{+1}(4_1)) < 0$, while $h_{\sigma}(S^{3}_{+1}(4_1))=0$. By the equivariant surgery formula, we at once obtain that $\tau$ and $\sigma$ acts differently on $CFK^{\infty}(4_1)$. This is in contrast to the periodic type and $\iota_K$ action on $4_1$, both of which are unique. Interested readers may consult \cite[Example 2.26]{abhishek2022equivariant} for the explicit computation of the $\tau$, $\sigma$ actions on $CFK^{\infty}(4_1)$. 
\end{example}

\subsection{Identifying Cork-twists involutions}\label{cork}
As discussed in the introduction, the surgery formula enables us to identify the cork-twist involution for corks that are obtained as surgeries on symmetric knots with the action of the symmetry on the knot. Note that we do need that the $3$-genus of the knot to be $1$. Such type of corks include the $(+1)$- surgery on the Stevedore knot \cite{dai2020corks}, $(+1)$-surgery on the $P(-3,3,-3)$ pretzel knot, \cite{dai2020corks} and the positron cork \cite{hayden2021corks}. For example, it was shown in \cite[Lemma 7.3]{dai2020corks} that two strong involutions $\tau$ and $\sigma$ of the Stevedore knot results in strong corks. Hence, the equivariant surgery formula identifies those cork-twists with the action of $\tau$ and $\sigma$ on $A^{-}_{0}(6_1)$. We refer readers to \cite[Example 2.27]{abhishek2022equivariant} for an explicit computation of these actions.

\subsection{Strong involution on $T_{2,3} \# T^{r}_{2,3}$}\label{KconnectKcompute}

Here, using Theorem~\ref{KconnectKstrong}, we compute the strong involution $\tilde{\tau}_{sw}$ on connected sum of two copies of the right-handed trefoil (See Subsection~\ref{KK} for the definition of the $\tilde{\tau}_{sw}$ action). We also compute the concordance invariant associated with this action, see Section~\ref{concordanceinvariant}. We state the computation below.

\begin{proposition}
Let $K$ represent the strongly invertible knot $ T_{2,3} \# T^{r}_{2,3}$. We have,
\begin{align*}
\underline{d}_{\tau}(A^{-}_{0}(K), \tilde{\tau}_{sw})& = \bar{d}_{\tau}(A^{-}_{0}(K), \tilde{\tau}_{sw}) = - 2. \\
\underline{d}_{\iota \tau}(A^{-}_{0}(K), \tilde{\tau}_{sw}) &= -4, \; \bar{d}_{\iota \tau}(A^{-}_{0}(K), \tilde{\tau}_{sw}) = - 2.
\end{align*}

\end{proposition}

\begin{proof}
From Proposition~\ref{KconnectKstrong}, we obtain that the action of $\tilde{\tau}_{sw}$ on $CFK^{\infty}(K \# K^{r})$ in  can be identified with the action of $(\mathrm{id} \otimes \mathrm{id} + \Psi \otimes \Phi) \circ \tilde{\tau}_{exch}$ on $CFK^{\infty} \otimes CFK^{\infty}(K^{r})$ (although we used the full version $\mathcal{CFK}$ for the Proposition, it can be checked that the maps involved have $0$ Alexander grading shift, hence they induce maps between the ordinary $CFK^{\infty}$ complexes). The knot Floer complex for $T_{(2,3)}$ and $T^{r}_{(2,3)}$ is shown in Figure~\ref{tausw}.  
\begin{figure}[h!]
\centering
\includegraphics[scale=1.1]{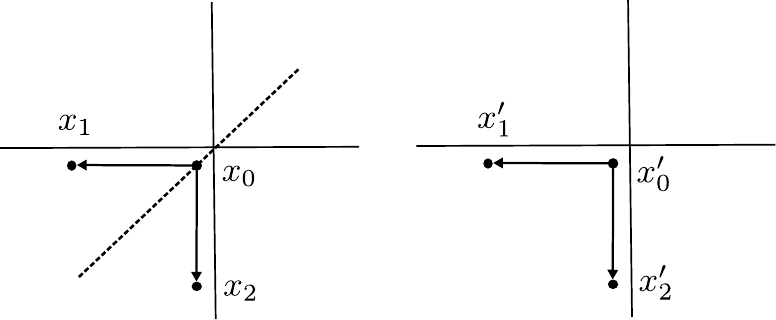}

\caption
{$CFK^{\infty}(T_{2,3})$ on the left. For $CFK^{\infty}(T^{r}_{2,3})$, we use the complex that is obtained by reflecting the stair along the $y=x$ diagonal.}\label{tausw}
\end{figure}
The complex $CFK^{\infty}(T_{2,3}) \otimes CFK^{\infty}(T^{r}_{2,3})$ is shown in Figure~\ref{trefoilconnect}.
\begin{figure}[h!]
\centering
\includegraphics[scale=1.1]{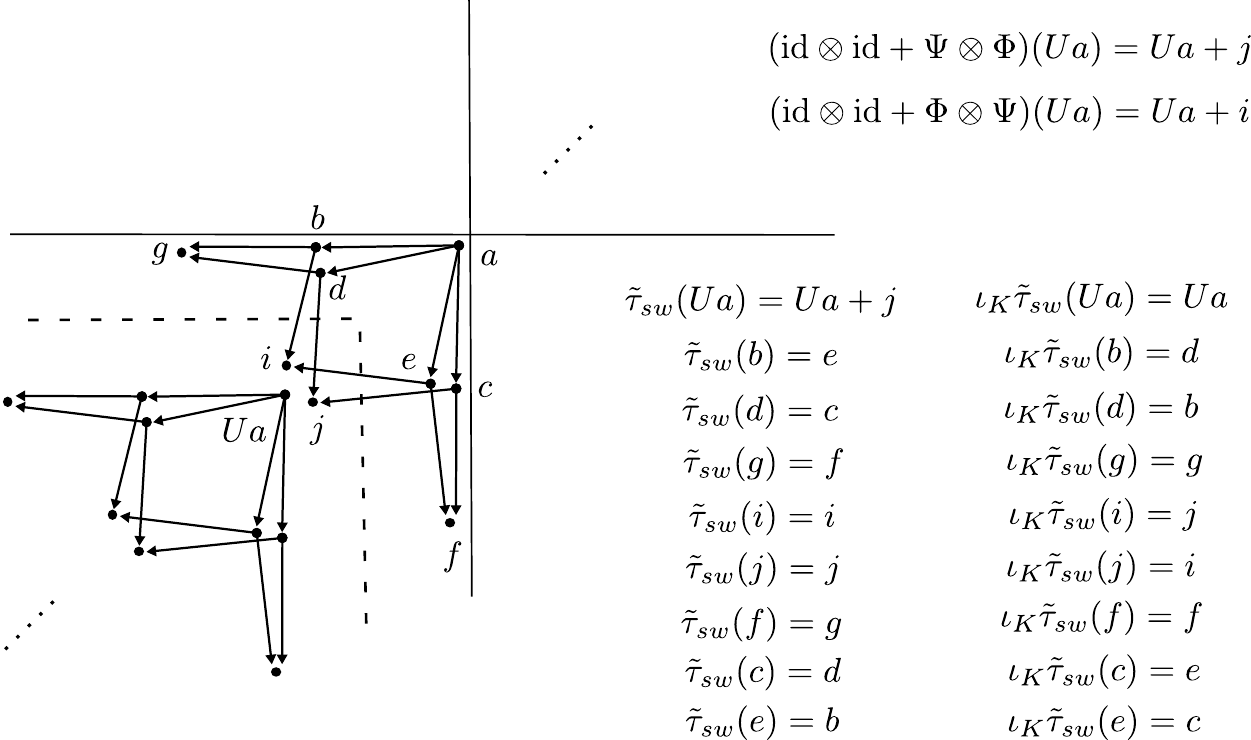} 

\caption
{Left:  $CFK^{\infty}(T_{2,3}) \otimes CFK^{\infty}(T^{r}_{2,3})$, Right: Action of $\tilde{\tau}_{sw}$ and $\Psi \otimes \Phi$ on the basis elements. }\label{trefoilconnect}
\end{figure}
Note that for the case in hand, $\tilde{\tau}_{exch}$ map can be computed explicitly. It exchanges the two factors, followed by a skew-filtered map in each factor. For example, we get $\tilde{\tau}_{exch}(x_1 \otimes x^{\prime}_2)=x_1 \otimes x^{\prime}_2$, while $\tilde{\tau}_{exch}(x_1 \otimes x^{\prime}_0)=x_0 \otimes x^{\prime}_2$ and so on. Also note that the map $\Psi \otimes \Psi$ is non-trivial only for the generator $a$. Hence the action of $\tilde{\tau}_{sw}$ is as shown in Figure~\ref{trefoilconnect}. Using Equation~\ref{iotaconnect} we also compute the action of $\iota_K \tilde{\tau}_{sw}$.

Now, by looking at the $A^{-}_{0}$-complex of the tensor product, it follows that $[i]$ and $[j]$ both are in grading $-2$ and are fixed by $(\mathrm{id} \otimes \mathrm{id} + \Psi \otimes \Phi) \circ \tilde{\tau}_{exch}$. Hence from the Definition of $\underline{d}^{\tau}$, we get,
\[
\underline{d}_{\tau}(A^{-}_{0}(K), \tilde{\tau}_{sw})=-2.
\]
By a similar reasoning as in \cite[Proposition 8.2]{Zemkeconnected}, we can also see that the $gr(i)$ determines the $\bar{d}_{\tau}$ invariant. Hence we get,
\[
\bar{d}_{\tau}(A^{-}_{0}(K), \tilde{\tau}_{sw})=-2.
\]
For $\underline{d}_{\iota \tau}$, we observe that the $i$, and $j$ are the only two generators in grading $-2$, which are $U$-non-torsion. However, neither of them is fixed by $\iota_K \tilde{\tau}_{sw}$. It can be checked that $[Ui]$ is the highest graded $U$-non-torsion element which is fixed by $\iota_K \tilde{\tau}_{sw}$, hence we get 
\[
\underline{d}_{\tau}(A^{-}_{0}(K), \tilde{\tau}_{sw})=-4.
\]
Calculation for $\bar{d}_{\iota \tau}$ is similar to the $\bar{d}$. Indeed, although the actions of $\iota_K$ and $\iota_K \tilde{\tau}_{sw}$ different, we see that $gr(i)$ still determines the $\bar{d}_{\iota \tau}$ invariant. We get,
\[
\bar{d}_{\tau}(A^{-}_{0}(K), \tilde{\tau}_{sw})=-2.
\]
\end{proof}

\begin{remark}\label{questiontausw}
Let us now compare the above invariants with similar invariants from the literature. The $d$-invariant defined by Ozsv{\'a}th and Szab{\'o} \cite{OSabsgr}, satisfies $d(A^{-}_0(K)) = -2$. While the involutive correction terms are $\bar{d}(A^{-}_{0}(K))= -2$ and $\underline{d}(A^{-}_{0}(K))=-4$, \cite[Proposition 8.2]{Zemkeconnected}. Hence this motivates the following question:
\begin{question}\label{questiondinvariant}
\noindent
\begin{enumerate}

\item Are there knots $K$, such that:
\[
\underline{d}_{\tau}(A^{-}_{0}(K \# K^{r}), \tilde{\tau}_{sw}) \neq {d}(A^{-}_{0}(K \# K^{r})) \; \text{or} \; \bar{d}_{ \tau}(A^{-}_{0}(K \# K^{r}), \tilde{\tau}_{sw}) \neq {d}(A^{-}_{0}(K \# K^{r}))?
\]

\item Are there knots $K$, such that:
\[
\underline{d}_{\iota \tau}(A^{-}_{0}(K \# K^{r}), \tilde{\tau}_{sw}) \neq \underline{d}(A^{-}_{0}(K \# K^{r})) \; \text{or} \; \bar{d}_{\iota \tau}(A^{-}_{0}(K \# K^{r}), \tilde{\tau}_{sw}) \neq \bar{d}(A^{-}_{0}(K \# K^{r}))?
\]
\end{enumerate}
\end{question}  
\end{remark}
\noindent
Perhaps it is note worthy that
\begin{align*}
\iota_{K \# K^{r}} \circ \tilde{\tau}_{sw} &= (\mathrm{id} \otimes \mathrm{id} + \Phi \otimes \Psi) \circ (\iota_{K} \otimes \iota_{K^{r}}) \circ (\mathrm{id} \otimes \mathrm{id} + \Psi \otimes \Phi) \circ \tilde{\tau}_{exch} \\
&\simeq (\mathrm{id} \otimes \mathrm{id} + \Phi \otimes \Psi) \circ (\mathrm{id} \otimes \mathrm{id} + \Phi \otimes \Psi) \circ (\iota_{K} \otimes \iota_{K^{r}}) \circ \tilde{\tau}_{exch} \\
&\simeq \iota_{K} \otimes \iota_{K^{r}} \circ \tilde{\tau}_{exch}.
\end{align*}
Where in the second line we have used that $\iota_K$ interchanges $\Phi$ and $\Psi$, and in the third line we have used $\Phi^{2}=\Psi^{2} \simeq 0$, see \cite[Lemma 2.8, Lemma 2.11]{Zemkeconnected}. Hence the map $\iota_{K \# K^{r}} \circ \tilde{\tau}_{sw}$  do indeed (non-trivially) differ from $\iota_{K \# K^{r}}$, which makes the case for an affirmative answer to Question~\ref{questiondinvariant}, number $(2)$.

\section{An invariant of equivariant knots}\label{knotCI} We end with an observation, which leads us to define a Floer theoretic invariant of the equivariant knots. Let us begin by recalling that given a doubly-based knot $K$ inside a $\mathbb{Z}HS^{3}$, $Y$, Hendricks and Manolescu \cite[Proposition 6.3]{HM} considered a mapping complex $CI$ of the map
\[
Q.(\mathrm{id} + \iota_K): CFK^{\infty}(Y, K, w, z) \rightarrow Q.CFK^{\infty}(Y, K, w, z)[-1].
\]
They show that $CI(Y,K,w,z)$ is an invariant of the tuple $(Y,K,w,z)$. However, this mapping cone does not have a $\mathbb{Z} \oplus \mathbb{Z}$-filtration induced from that of $CFK^{\infty}(Y,K,w,z)$. This is because $\iota_K$ is skew-filtered, but the identity map is filtered, so there is no natural filtration in the mapping cone.
On the contrary, if we start with a periodic knot $(Y,K, \tau, w, z)$ and consider a similar mapping cone complex $CI_{\tau}$ of the map;
\[
Q.(\mathrm{id} + \tau_K): CFK^{\infty}(Y, K, w, z) \rightarrow Q.CFK^{\infty}(Y, K, w, z)[-1],
\]
there is a natural $\mathbb{Z} \oplus \mathbb{Z}$-filtration on it. This holds because the action $\tau_K$ is filtered. In particular, we obtain an invariant of periodic knots that has more information than its involutive counterpart. Recall the notion of equivalence between two periodic knots,

\begin{definition}\label{periodic_equivalence}
We call two periodic knots $(Y_1, K_1, \tau_1, w_1, z_1)$ and $(Y_2, K_2, \tau_2, w_2, z_2)$ equivalent if there exist a diffeomorphism
\[
\phi:(Y_1, K_1, w_1, z_1) \rightarrow (Y_2, K_2, w_2, z_2),
\]
which intertwines with $\tau_1$ and $\tau_2$.
\end{definition}

\begin{proposition}\label{periodic_invariant}
Let $(Y,K, \tau, w,z)$ be a periodic knot inside a $\mathbb{Z}HS^{3}$, $Y$. Then the $\mathbb{Z} \oplus \mathbb{Z}$-filtered quasi isomorphism type of the mapping cone complex $CI_{\tau}(Y, K, \tau, w, z)$ (defined above) over $\mathbb{Z}_{2}[Q, U, U^{-1}]/(Q^{2})$ is an invariant of the equivalence class of $(Y,K,\tau, w,z)$.
\end{proposition}

\begin{proof}
Let $(Y_1,K_1,\tau_1)$ and  $(Y_2,K_2,\tau_2)$ be two tuples in the same equivalence class. The \textit{diffeomorphism invariance of the link cobordisms} proved by Zemke \cite[Theorem A]{Zemkelinkcobord} implies that the following diagram commutes up to chain homotopy. 
\[\begin{tikzcd}[column sep =large, row sep =large]
CFK^{+}(Y_1, K_1) \arrow{r}{\phi} \arrow{d}{\tau_{K_1}} & CFK^{+}(Y_2, K_2) \arrow{d}{\tau_{K_2}} \\
CFK^{+}(Y_1, K_1) \arrow{r}{\phi}  & CFK^{+}(Y_2, K_2)
\end{tikzcd}
\]
Hence this induces a quasi-isomorphism between the corresponding mapping cone complexes which respects the filtration.
\end{proof}

In a similar manner, for a strongly invertible knot $(Y,K, \tau, w,z)$, we refer to the mapping cone of 
\[
Q.(\mathrm{id} + \iota_K \tau_K): CFK^{\infty}(Y, K, w, z) \rightarrow Q.CFK^{\infty}(Y, K, w, z)[-1]
\]
as $CI_{\iota \tau}(Y, K, \tau, w, z)$. We have:
\begin{proposition}\label{strong_invariant}
Let $(Y,K, \tau, w,z)$ be a strongly invetible knot inside a $\mathbb{Z}HS^{3}$, $Y$. Then the $\mathbb{Z} \oplus \mathbb{Z}$-filtered quasi isomorphism type of the mapping cone complex $CI_{\iota \tau}(Y, K, \tau, w, z)$ over $\mathbb{Z}_{2}[Q, U, U^{-1}]/(Q^{2})$ is an invariant of the equivalence class of $(Y,K,\tau, w,z)$.
\end{proposition}

Hence we have:
\begin{proof}[Proof of Proposition~\ref{knotinvariant}]
Follows from Proposition~\ref{periodic_invariant} and Proposition~\ref{strong_invariant}.

\end{proof}

It follows from the computations in Section~\ref{computation} that for any L-space knots both $CI_{\iota \tau}$ (for strongly invertible knots) and $CI_{\tau}$ (for periodic knots), become trivial i.e. determined by the $CFK^{\infty}(K)$. On the other hand for thin knots, both invariants are potentially non-trivial owing to the Proposition~\ref{thin}. Moreover, by sending a conjugacy class of $[(Y, K ,\tau, w, z)]$ to the conjugacy class of $[(Y_{p}(K),\tau,[0])]$ (for a large $p$, which is an odd number), we see that equivariant involutive correction terms are invariants of $[(Y, K ,\tau, w, z)]$. Hence they can be used to demonstrate the non-triviality of $CI_{\tau}$ and $CI_{\iota \tau}$. 

\begin{example}
Let $(4_1, \tau^{p})$ be the figure-eight knot with a periodic symmetry $\tau^{p}$, as in Figure~\ref{figure_eight_periodic}. Then the quasi-isomorphism type of the mapping cone $CI_{\tau}(4_1, \tau^{p}, w, z)$ is non-trivial. One may see this by computing the mapping cone explicitly, staring from the description of $\tau^{p}$ action on $CFK^{\infty}(K)$. However, instead of going through the full computation, we show that the invariant is non-trivial by appealing to the invariant $\underline{d}_{\tau}$ and showing that it differs from the $d$-invariant. By examining the action in Figure~\ref{figure_eight_periodic}, we see that $\underline{d}_{\tau}(A^{-}_{0})=-2$, while $d(A^{-}_{0})=0$. The grading shift from the Relation~\ref{gradingshift} then implies 
\[
\underline{d}_{\tau}(S^{3}_{p}(4_1),[0])=\frac{p-1}{4}-2, \; \; d(S^{3}_{p}(4_1),[0])=\frac{p-1}{4}.
\]
\end{example}

\bibliographystyle{amsalpha}
\bibliography{equivariant_surgery}

\end{document}